\documentclass[a4paper] {article}[12pt]

\usepackage[top=3cm, bottom=3cm, left=2.2cm, right= 2.2cm]{geometry}

\makeatletter
\renewcommand*\l@section{\@dottedtocline{1}{1.5em}{2.3em}}
\makeatother

\usepackage{amsfonts}
\usepackage{amssymb}
\usepackage[T1]{fontenc}

\usepackage{tikz}
\usetikzlibrary{calc}

\usepackage{CJK}
\usepackage{amsmath}
 
\usepackage{amsfonts}
\usepackage{amssymb}
\usepackage{amsthm}
\usepackage{amssymb}
\usepackage{enumerate}
\usepackage[calc]{picture}
\usepackage[all,cmtip]{xy}

\usepackage[mathscr]{eucal}
\usepackage{eqlist}

\usepackage{color}
\usepackage{abstract} 
\usepackage[T1]{fontenc}
 
\theoremstyle{plain}
\newtheorem{theorem}{Theorem}
\newtheorem{proposition}[theorem]{Proposition}
\newtheorem{lemma}[theorem]{Lemma}

\newtheorem{example}[theorem]{Example}
\newtheorem{corollary}[theorem]{Corollary}

\theoremstyle{definition}

\usepackage{etoolbox}
\newtheoremstyle{myrem}%name
 {3pt}%Space above
 {3pt}%Space below
 {\normalsize}%Body font
 { }%Indent amount
 {\itshape}% Theorem head font
 {:}%Punctuation after theorem head
 { }%Space after theorem head 2
 {}%Theorem head spec (can be left empty, meaning ?normal?)

 \theoremstyle{myrem}
 \newtheorem{remark}{Remark}
 \appto\remark{\leftskip\parindent}
 \appto\remark{\rightskip\parindent}

\usepackage{amsmath}

\numberwithin{equation}{section}
\numberwithin{theorem}{section}

\begin{document}

\begin{center}
{\Large{\textbf{ Homological  obstructions       for   
regular  embeddings  of  graphs
 }}}

 \vspace{0.5cm}
 
 Shiquan Ren 

  \vspace{0.5cm}

{\small
\begin{quote}
\begin{abstract}
In  \cite[Section~8]{jgp},  the  present  author  proposed  the  
hypergraph  obstruction  for  the  existence  of  $k$-regular  embeddings.  
In  this  paper,  we  develop  the  hypergraph  obstruction  concretely 
and  give  some    homological  obstructions  for  the  $k$-regular  embeddings  of  graphs
by  using  the embedded  homology  of  sub-hypergraphs  of  the 
$(k-1)$-skeleton of  the  independence complexes.   Regular  embeddings
of  graphs  can  be  regarded  equivalently  
as  geometric  realizations  of  the  independence  complexes and  consequently  
be regarded  equivalently  as  simplicial  embeddings  of  the  independence  complexes
into  the  vectorial  matroids.   
We  prove that  if  there  exists  a  $k$-regular  embedding of a graph, 
then there  is an induced  homomorphism from  the embedded  homology 
of  the  sub-hyper(di)graphs  of the $(k-1)$-skeleton  of  the (directed)   independence  complexes 
to  the  homology  of  (directed)  matroids.  
Moreover,  if there  exists  certain     triple of graphs
where  each graph  has  a  $k$-regular  embedding,
then  there  are  induced  commutative  diagrams   of certain  Mayer-Vietoris  sequences
of  the  embedded  homology  of   hyper(di)graphs,
the  homology  of    (directed)  independence  complexes 
  and  the  homology  of  matroids.  
Furthermore,  if there  exists  certain     couple  of  graphs
where  each graph  has  a  $k$-regular  embedding,
then  there  are  induced  commutative  diagrams   of certain  Kunneth type  short  exact sequences 
of  the  embedded  homology  of   hyper(di)graphs,
the  homology  of     (directed)  independence  complexes 
  and  the  homology  of  matroids.  
  \end{abstract}
\end{quote}
}

\begin{quote}

\end{quote}

\end{center}

\section{Introduction}

  \subsection{The regular  maps}\label{ss1.1}

  Let  $X$  be  a  topological  space.  
  Let  $\mathbb{F}$  be  a field.  
  Let  $\mathcal{T}$  be  a  topology  on  $\mathbb{F}$.  
  Then  the  vector  space $\mathbb{F}^N$  has  the  product  topology $\prod_N \mathcal{T}$.   
  A   continuous  map $f:  X\longrightarrow  \mathbb{F}^N$  is  called  {\it $k$-regular}
  if  for  any distinct  $k$-vertices  $x_1,\ldots, x_k\in  X$,  
  their  images  $f(x_1),\ldots, f(x_k)$  are  linearly  independent  over   $\mathbb{F}$.  
  For  any  $k\geq  2$,  such  a  $k$-regular  map  is   an  embedding.  
  Going back to Chebychev,
Haar \cite{haar17} and Kolmogorov  \cite{kol48} 
(cf.  \cite{cohen1,alg-geom}),
 the  $k$-regular  embedding     is  related  to the  interpolation  problem. 
 In  particular,  when    
   $X=\mathbb{R}^d$,     
  the  existence of   an  $N$-dimensional  $k$-interpolating  space  on $\mathbb{R}^d$   
     was  proved  to  be   equivalent  to  
       the existence  a $k$-regular  embedding  of  $\mathbb{R}^d$  into  $\mathbb{R}^N$
   (cf. \cite[Theorem~1.3]{alg-geom}).

  In  1957,   K.  Borsuk \cite{Borsuk}  studied  
the   regular  embedding   problem  of  Euclidean  spaces.  
In  1978,    F.R. Cohen and D. Handel  \cite{cohen1} 
used the  Stiefel-Whitney  classes  of  the  canonical  vector  bundles  over  configuration  spaces
 as     obstructions    
 for  regular  embeddings  and    gave  the  infimum  of  $N$ 
  for any  $k\geq  2$
   such that  there  exists  a   $k$-regular  embedding  of  $\mathbb{R}^2$     into  $\mathbb{R}^N$.  
   In 1979, M.E. Chisholm \cite{chisholm} extended  \cite{cohen1}  and  gave     lower  bounds  
    of  
   $N$  for  the  existence  of    $k$-regular  embeddings  of    $\mathbb{R}^d$  into  $\mathbb{R}^N$ 
   where   $d$  is   a  power of  $2$. 
      In   2016,  P. Blagojevi\'{c},   W. L{\"u}ck and G. Ziegler \cite{high1}  
      extended  \cite{cohen1,chisholm} and  gave     lower  bounds  
    of  
   $N$  for  the  existence  of    $k$-regular  embeddings  of    $\mathbb{R}^d$  into  $\mathbb{R}^N$ 
   for  any    $d\geq  2$;  and  
   P. Blagojevi\'{c},  F.R. Cohen, W. L{\"u}ck and G. Ziegler  \cite{high2}
used the  mod  $p$  Chern    classes    of  the  complexification  of  the  canonical  vector  bundles  over  configuration  spaces  for any  odd  primes  $p$  
 as     obstructions    
 for  complex  regular  embeddings 
 of  $\mathbb{C}^d$  into  
 $\mathbb{C}^N$  and  gave    lower  bounds  
    of  
   $N$  for  the  existence  of   complex  $k$-regular  embeddings  of    $\mathbb{C}^d$  into  $\mathbb{C}^N$ 
   where     $d $   is  a  power  of  $p$.    
   In  2017,       M. Michalek and C. Miller
   \cite{alg-geom} 
   applied  tools from algebraic geometry   and  
constructed    a  complex   $4$-regular polynomial map from $\mathbb{C}^3$  to  
$\mathbb{C}^{11}$  and  
  a complex  $5$-regular polynomial map from  $\mathbb{C}^3$  to  $\mathbb{C}^{14}$.  
In  2019,  J. Buczy\'nski,   T.  Januszkiewicz,  J.  Jelisiejew   and  M.  Michalek \cite{jems}
constructed  new   complex  $k$-regular  maps  and  
studied  the  upper  bounds  of  $N$  such  that  there  exists  a  complex  $k$-regular  map
 of   $\mathbb{C}^d$  into  $\mathbb{C}^N$  
by investigating   
the dimension of the locus of certain Gorenstein schemes in the punctual Hilbert scheme.

   Besides  the  regular  embeddings  of   Euclidean  spaces,
   regular   embeddings  of  manifolds  have been  studied.   
   In  1980,  D.  Handel  \cite[Theorem~2.4]{handel2}  
   used  the  Stiefel-Whitney  classes  of  the  tangent  bundles  
   to  give  lower  bounds  of  $N$   for the existence   of  regular  embeddings  
   of  disjoint  unions  of  closed,  connected  manifolds  into  $\mathbb{R}^N$.  
   In  2016, the    $3$-regular embedding   of  the  sphere   $S^N$ into  $\mathbb{R}^{N+2}$  
   was  given  in  \cite[Example~2.6~(3)]{high1}. 
   In  2018,  the  present   author  \cite{reg-2018}
   applied  the  Stiefel-Whitney  classes  \cite{cohen1}  as  well  as  the  mod  $p$  Chern  classes 
   \cite{high2}  of  the  canonical  vector  bundles  over  configuration  spaces 
   and  gave  lower  bounds  of  $N$  for   $2$-regular  embeddings  of  
   real,  complex  and  quaternionic  projective  spaces into
     $\mathbb{R}^N$  
   as  well  as  for  complex  $2$-regular  embeddings  of  complex  projective  spaces 
   into  $\mathbb{C}^N$.

      Let  $d:  X\times  X\longrightarrow  [0,+\infty]$  be  a  distance     on  $X$    
  so  that    $(X,d)$  is  a  metric space.  
Let   $f:  (X,d)\longrightarrow  \mathbb{F}^N$    be  a    continuous  map.  
 We  say that  $f$  is  
  {\it $k$-regular 
with  respect  to  $r$} 
  if  for  any distinct  $k$-vertices  $x_1,\ldots, x_k\in  X$  such that  
  $d(x_i,x_j)>2r$,     $1\leq  i<j\leq  k$,  
  their  images  $f(x_1),\ldots, f(x_k)$  are  linearly  independent  over   $\mathbb{F}$.    
  Such  regularities  with  respect  to  $r$ 
  were  considered  by  the  present  author \cite[Section~6]{2025-reg}.  
      We  call  a  vector space  $\mathcal{F}$  of continuous functions on  $X$  with  values  in  $\mathbb{F}$
    a  {\it    $k$-interpolating  space  
  with  respect to  $r$}
   if for any distinct points $x_1,\ldots,x_k\in X$  such  that  $d(x_i,x_j)>2r$,     $1\leq  i<j\leq  k$,      and any scalars  $\lambda_1, \ldots,  \lambda_k\in  \mathbb{F}$,  
there exists an   $\mathbb{F}$-valued  function $f \in  \mathcal{F}$  such that $f(x_i) = \lambda_i$ 
for all $ i = 1, \ldots, k$.  
Analogous  to  \cite[Sect. 2,  Proof  of  Theorem~1.3]{alg-geom}, 
 it  can  be  proved  that  there  exists     an  $N$-dimensional $k$-interpolating space $\mathcal{F}$ on  $X$  
 with  respect to  $r$
 if and only if there
exists a  continuous  map $f: X
\longrightarrow  \mathbb{F}^N$ that  is $k$-regular  with  respect to  $r$.

 The  significance  of  the  above   hypothesis   that  the  interpolating  points  $x_1,\ldots,x_k\in  X$
  have  their  mutually distances  greater  than  $2r$  (cf.  \cite[Section~6]{2025-reg}) 
   is  to  avoid  excessive  concentration  
   of  the  interpolating  points.  With  the  help  of  this  hypothesis,  
   it    will   
  reduce  the  effects  of  small  perturbations  of  the  values  at  the  interpolating  points 
      on  the  global  behavior of  the  interpolation functions.

     \subsection{The  topology  of  configuration  spaces,   independence  complexes  and  matroids}

For  any  topological  space  $X$,  
   the  $k$-th  {\it  ordered  configuration  space}  ${\rm  Conf}_k(X)$  on  $X$  is  the open  subspace 
   of  $\prod_k  X$  consisting   of  all  the  ordered  $k$-tuples  $(x_1,\ldots,x_k)$  
   such  that  $x_i\neq  x_j$  for  any  
   $i\neq  j$.  
   The  $k$-th  symmetric  group  $\Sigma_k$  acts  on  ${\rm  Conf}_k(X)$  
   freely  and  properly  discontinuously  on   ${\rm  Conf}_k(X)$  by  permuting  the  coordinates
   from  the  right.  
   The  orbit  space   ${\rm  Conf}_k(X)/\Sigma_k$  is  called  the  
   $k$-th  {\it  unordered  configuration  space}  on  $X$,
   which  consists  of  all  the  $k$-subsets  $\{x_1,\ldots,x_k\}$  of  $X$.

   Since  1980's,  
   the  homology  of  configuration  spaces  have  been  extensively  studied. 
      In  1987    and  1989,   
   C.-F.   B\"odigheimer,  F.  Cohen  and  L.   Taylor    \cite{87-h-conf,89-h-conf} 
   studied  the  additive  structure  of  the  
   homology  of  the  configuration  spaces  ${\rm  Conf}_k(M)$  
    and    ${\rm  Conf}_k(M)/\Sigma_k$   for  manifolds  $M$   with   coefficients  from    fields.  
      In  1988,   C.-F.   B\"odigheimer  and  F.R.    Cohen 
      \cite{88-h-conf}
      studied    the  rational  cohomology of  the  configuration  spaces  ${\rm  Conf}_k(M_g)$  
    and    ${\rm  Conf}_k(M_g)/\Sigma_k$  where  $M_g$  is  a  closed  surface  of  genus  $g$.  
    In 2012,  T.  Church  \cite{invent2012}     
    proved  the  homological  stability  
    for  the configuration  spaces  ${\rm  Conf}_\bullet(M)/\Sigma_\bullet$  with  rational  coefficients.   
    In  2014,    T. Church, J. S. Ellenberg and B. Farb   \cite{duke} applied 
    the  theory  of   FI-modules and stability for representations of symmetric groups   to obtain results about  the  cohomology of configuration spaces.       
       In  the  study  of  the  $k$-regular  embedding  problems  of  $X$,  
   the   cohomology  of  the  configuration  space  ${\rm  Conf}_k(X)/\Sigma_k$ 
    is  used  to  give  obstructions  for  the  existence  of  $k$-regular  embeddings
    (cf.    \cite{cohen1,high1,high2}).  
    Since  the  associated  vector bundle  of  the  canonical  covering  map  
    over  ${\rm  Conf}_k(X)/\Sigma_k$    has   flat  connections,
    if  we  consider  the  characteristic  classes  proposed  in  \cite{cohen1,high1,high2}, 
     then 
    only  the  torsion  part  of  the  cohomology  of      ${\rm  Conf}_k(X)/\Sigma_k$  
    or  the  cohomology  with  mod  $p$  coefficients,  
    where  $p$  is  a  prime,   is  essential  
    for  the  obstructions  of    $k$-regular  embeddings.

  Besides  the  homology  of  configuration spaces,
   the  homotopy  type  of  configuration  spaces   attracts  lots  of  attention.    
    Note  that   a  homeomorphism   $X\cong  X'$  induces 
     a  homeomorphism ${\rm  Conf}_k(X)\cong   {\rm  Conf}_k(X')$
     while 
   a  homotopy  equivalence         $X\simeq  X'$ 
    does  not   necessarily  induce   a  homotopy  equivalence 
     ${\rm  Conf}_k(X)\simeq   {\rm  Conf}_k(X')$.     
   In  2005,  
   R. Longoni and P. Salvatore \cite{topol05} 
   gave   an   example   and  showed   that  even  if  $X$  is  a  closed   manifold  $M$,  
   the  homotopy type of  
   $M$  cannot   determine  the homotopy type of   ${\rm   Conf}_k(M)$.       
   Nevertheless,  
   if  $X$  is  a  closed  simply-connected  manifold  $M$,  then   in  2008,  
   P. Lambrechts  and  D. Stanley \cite{agt,poincare}  
    constructed  some  explicit  CDGA-models  that  satisfy  the  Poincar\'e  duality,      
     and  in  2019, 
    N.  Idrissi \cite{najib}  used  the  
    Lambrechts-Stanley model  and  
   proved  that   the  rational  homotopy  type  of   ${\rm  Conf}_k(M)$  
   is  determined  by  the  rational  homotopy type  of  $M$.  
   To  calculate  the  characteristic  classes  of  the    classifying  maps
   and  thus  give   obstructions  for   $k$-regular  embeddings  of  $X$, 
   the  homotopy  classes  of  maps  from  ${\rm  Conf}_k(X) /\Sigma_k$  
   to  the  Grassmannians  were  investigated  
     (cf. \cite{cohen1,high1,high2}).  
     Therefore,  the  homotopy  type  of  ${\rm  Conf}_k(X) /\Sigma_k$  
     is  essential  for  the  obstructions  of  $k$-regular  embeddings.

   Let  $r\geq  0$.  
   For any  metric space  $(X,d)$,  
   the  $k$-th  {\it  ordered  configuration  space of  $r$-spheres}  ${\rm  Conf}_k(X,r)$     is  the open  subspace 
   of  ${\rm  Conf}_k(X)$   consisting   of  all  the  ordered  $k$-tuples  $(x_1,\ldots,x_k)$  
   such  that  $d(x_i,  x_j)>2r$  for  any  
   $i\neq  j$  (cf.  \cite{gt,imrn1,phys}).  
   When  $r=0$,   ${\rm  Conf}_k(X,0)$   is  ${\rm  Conf}_k(X)$;  and  
   when  $r\to +\infty$,   $\bigcap_{r\geq  0}   {\rm  Conf}_k(X,r) =\emptyset$.  
         The  orbit  space   ${\rm  Conf}_k(X,r)/\Sigma_k$  is     the 
     {\it  unordered  configuration  space of  $r$-spheres}.
      We  have
        a  
        $\Sigma_k$-equivariant  filtration   ${\rm  Conf}_k(X,-)$  of    ${\rm  Conf}_k(X)$   
      and  an  induced    filtration   ${\rm  Conf}_k(X,-)/\Sigma_k$  of   ${\rm  Conf}_k(X)/\Sigma_k$.             
      As  $r$  goes from  $0$  to  $+\infty$,  
      there  are  finite  or  countable   points       
      such that  at  each  point,   the  homotopy type  of  ${\rm  Conf}_k(X,-)$ 
      or   the  homotopy type  of  ${\rm  Conf}_k(X,-)/\Sigma_k$  changes.   
      Applying  the  homology  functor,  we  obtain  
      the  $\Sigma_k$-equivariant 
        persistent homology  of  ${\rm  Conf}_k(X,-)$ 
      and  the  persistent homology  of  ${\rm  Conf}_k(X,-)/\Sigma_k$
       (cf.  \cite{2025-reg}).  
       By  considering the  homology  of  homotopy  type  of  
       the  configuration   spaces  of  $r$-spheres,   
       we  can   obtain   obstructions  for  the  existence  of  
      $k$-regular  embeddings  of  $(X,d)$  with  respect  to  $r$
       in  terms  of  the   Stiefel-Whitney  classes  and  the  mod  $p$  
       Chern  classes  in  a  similar  way  of   \cite{cohen1,high1,high2}.

      From the  unordered configuration  spaces  of  $1/2$-spheres  on  the  vertex  set 
      with  respect  to the  geodesic  distances  of  graphs,
      we  can  construct the  skeletons  of   the  independence  complexes.  
      In  2006,  A.J. Berrick, F.R. Cohen, Y.L. Wong and J. Wu  \cite{jams-06}  
      studied  the 
        simplicial  structures  of  the  configuration  spaces and  gave 
        certain connections between braid groups and the homotopy groups of the sphere.    
        Let   $G=(V,E)$   be  a  graph.
        Let  $d_G$  be  the  geodesic distance  of  $G$.        
        We  have a  metric  space  $(V, d_G)$.  
          By   the  simplicial  structures  of  the  configuration  spaces  of $(V, d_G)$,  
      the   independence  complex  ${\rm  Ind}(G)$     
      can  be  constructed  such  that  for  any  $k\geq  1$,  
      the  $(k-1)$-simplices  are  elements  of      
      ${\rm  Conf}_k(V,1/2)/\Sigma_k$.  
      The  homology  (for  example,  \cite{zhwt-whh})  and  the  homotopy  type 
      (for   examples,  \cite{split-ind,kozlov}) of  
     ${\rm  Ind}(G)$   have  been  found      significant   and  interesting   in  graph  theory.

      In  1935,  H.  Whitney  \cite{1935mat}  introduced  
       {\it  matroids}  to  capture    the  abstract   essence of  linear  dependence. 
      For  any  finite  subset  $S$  in  a  fixed  vector  space,  
      a  subset   $\sigma$  of  $S$   is  called  an   {\it  independent  set}  if  
      the  vectors  in  $\sigma$  are  linearly  independent.  
      the    collection  of  all  the   independent  sets  $\sigma\subseteq  S$  gives  
      a  matroid  $\mathcal{M}$,  called  the  {\it  vectorial  matroid}  (cf.  \cite[Section~1.3]{matroid}).  
        The    vectorial  matroids  are    geometric  ramifications  
      of  the  Stiefel  manifolds   and   certain  subspaces  of  
       the  unordered  configuration  spaces  of   projective  spaces  (cf.     \cite{matroid-conf-77}).  
      Precisely, 
        if  the  vector  space  is  $\mathbb{R}^N$  or  $\mathbb{C}^N$,  
      then  the  elements  in    $\mathcal{M}$  are  $\Sigma_k$-orbits  
      of  the  elements  in  the  Stiefel manifolds $V_k(\mathbb{R}^N)$  or    $V_k(\mathbb{C}^N)$,  
      where  $1\leq  k\leq  N$  and  $\Sigma_k$  acts  on  $V_k(\mathbb{R}^N)$  or    $V_k(\mathbb{C}^N)$ by  permuting  the  coordinates'  orders  of  the  frames.   
      Alternatively,  the   elements    in    $\mathcal{M}$ 
      are   sets  of  linearly  independent  vectors  in  $\mathbb{R}^N$
      or  $\mathbb{C}^N$,  
      which  can  be  seen  as  
      certain  points  in  the  unordered  configuration  spaces  of  
      $\mathbb{R}P^{N-1}$  or  $\mathbb{C}P^{N-1}$.    
      The  simplicial  structures  of  the  Stiefel  manifolds  and  the  configuration  spaces  
      where  the   face  maps  and  the  degeneracies  are   
      obtained  by  removing  and  adding  coordinates  respectively  (cf.  \cite[Section~3]{jams-06}) 
       will  give  the  
      combinatorial  structures  of  the  vectorial  matroids. 

    \subsection{The  embedded  homology  of  hypergraphs}

The  embedded  homology  of  hypergraphs  is  
a  generalization  of  the  simplicial  homology  by  regarding  a  
hypergraph  as  a  simplicial  complex  with  certain  non-maximal  simplices  removed.  
Given  a  vertex  set  $V$,  
a  {\it  hypergraph}  $\mathcal{H}$  on  $V$   is  
a  family  of  non-empty  finite  subsets of  $V$.  
The   elements  of  $\mathcal{H}$  are  {\it  hyperedges}  (cf.  \cite{berge}).  
In  1991,  A. D. Parks and S. L. Lipscomb  \cite{parks} 
considered  the  associated  simplicial  complex of  a  hypergraph  $\mathcal{H}$, 
which  is  the  smallest  simplicial  complex  containing $\mathcal{H}$,  
and  studied the  simplicial  homology.   
 In  2019,  S.  Bressan,  J.  Li,    J.  Wu   and  the  present  author  \cite{h1}
 proved   that  the  smallest  chain  complex  containing  the graded  module  of  
      the  hyperedges  and  the  largest  chain  complex  contained  in  the  
      graded  module  of  the  hyperedges  are  quasi-isomorphic  and  thereby   
    defined  the   embedded  homology  of  a  hypergraph  as  a  generalization
      of  the  usual  homology  of  simplicial  complexes. 
  A  Mayer-Vietoris  sequence   of  the  embedded  homology  
     was  proved  in  \cite[Theorem~3.10]{h1}.    
  In  2022, 
       some  more  Mayer-Vietoris  sequences for  the  embedded   homology  of  hypergraphs  and
  the  homology  of  the  (lower-)associated  simplicial  complexes
  were  proved  by  J.  Wu, M.  Zhang  and  the  present  author   \cite{jktr2022-2}.  
 In  2023,   the  persistence  of  
  some   K\"unneth-type  formulae for  the  embedded  homology  of    hypergraphs  and  the  homology  of  the  (lower-)associated  simplicial  complexes 
 was   studied  by    J.  Wu  and  the  present  author  \cite[Section~4.4]{stab-hg};  
  and  some  random  versions  of  
    the   K\"unneth-type  formulae  for    the  embedded  homology  of  random   hypergraphs  and  the  homology  of  the  random  (lower-)associated  simplicial  complexes
    were  proved     
    by  C.  Wu,  J.  Wu  and  the  present  author  \cite[Theorem~1.3]{jktr2023}.

     Besides  the  embedded  homology  of  hypergraphs,  
     the  embedded  homology  of  hyperdigraphs  and  
     double  complexes  of  hyper(di)graphs  on  manifolds  have  also  been  studied.  
     In  2023, D. Chen,    J.  Liu,  
  J. Wu   and G.-W.  Wei  \cite{hdg}  considered  hyperdigraphs by  
  assigning  orientations  on  the  hyperedges  and  
   studied   the  persistence  of   the  embedded  homology  as  well  as    
    the  Laplacians  of  hyperdigraphs.  
   In  2025,   the  present  author  \cite{jgp}
  considered  a  hyper(di)graph   whose  vertices  are  moving  on  
   a  manifold.   Consequently,  \cite{jgp}   embedded  the  hyperdigraph  as  
 a  graded  submanifold   of   the  ordered  configuration  spaces 
  and  embedded  the  hypergraph  as  a  graded  submanifold  of  the  
  unordered  configuration  spaces.  
  The  infimum  double  complex  and  the  supremum  double  complex  
 for  hyper(di)graphs  on  manifolds  were  constructed  and  were  proved  
 to  be  quasi-isomorphic  with  respect  to  the  boundary  map  induced  
   by  vertex-deletions. 
       For    the   hyper(di)graphs   on  manifolds  
   which are      not     closed  under  the  vetex-deletion  operations,   
   for  example,   the  collection  of  all  the   finite  open  coverings  of  a  manifold, 
     the  collection  of  all  the   finite  atlases  of  a  manifold,  
   the  collection  of  all  the  finite  tight  packings  of   geodesic   balls  of  a  fixed  radius 
   in  a  Riemannian  manifold,  etc.,       
   the  infimum  double  complex  and  the  supremum  
   double  complex  are  not  equal.

    The  associated   vector bundle  of  the  canonical  covering  map  
    over   a  $k$-uniform   hypergraph  on  a  manifold  
    is  the  restriction  of  the  associated   vector bundle  of  the  canonical  covering  map  
    over   the  $k$-th  unordered  configuration  space.  
    Since the  cohomology  of  the    $k$-uniform   hypergraphs  on  a  manifold
      contains  more  information 
      than the  cohomology  of  the  $k$-th  unordered  configuration  space, 
    it  is  reasonable  to  expect  that  
    the  characteristic  classes  in  the  cohomology  of   the    $k$-uniform   hypergraphs
    on    manifolds   would  hopefully  give  more  obstructions  for  the  
    existence  of  $k$-regular  embeddings  of  the  manifolds  (cf.  \cite[Section~8]{jgp}).

                  \subsection{Results  of  this  paper}

    In  this  paper,  we  consider   the     extremal  case  of  the  
    $k$-regular  embedding problem  of  the  metric  space  $(X,d)$  
    into  $\mathbb{F}^N$  with  respect  to  $r>0$  
    such  that  $X$  is  a  discrete  set  $V$   and     $d=d_G$  
    is  the  geodesic  distance  of  a  graph  $G$  on  $V$.     
     We  give  obstructions  for  the  existence 
     of  $k$-regular  embeddings 
     by  the  embedded  homology  of  sub-hyper(di)graphs
      of  the  (directed)  independence  complexes,  
      the  Mayer-Vietoris  sequences  and  the  K\"unneth-type  formulae.

 Let  $G=(V,E)$  be  a  graph.  
 For  any  $v,u\in  V$,  
 we  say  that  the  sequence  $v_0v_1\ldots  v_l$   is  a  path of  length  $l$ 
  in  $G$  from $v$ to  $u$
 if  $v_0,v_1,\ldots,v_l\in  V$,  $v_0=v$,  $v_l=u$ 
 and $(v_{i-1},v_i)\in  E$  for any  $1\leq  i\leq  l$
 \footnote[1]{An  edge  in  $E$  is  $2$-set  $\{v,u\}\subseteq  V$.  Nevertheless,  
 by  an  abuse  of  notation,  we  also  write  $\{v,u\}$  as  $(v,u)$  or  $(u,v)$  with  the  equivalence 
 relation  
 $(v,u)\sim  (u,v)$  (cf.  Example~\ref{ex-25-10-19-1}).   }.  
 The  {\it  geodesic  distance}  of  $G$  is  a  function  $d_G:  V\times  V\longrightarrow [0,+\infty]$ 
 where   for  any  $v,u\in  V$,
    $d_G(v,u)$  is  defined  as   the  infimum of the   lengths  $l$  of  all  the  paths  in  $G$ 
  from   $v $  to  $u$. 
  In  particular,  (i)  $d_G(v,u)=0$  if  and  only  if $v=u$;  
  (ii)  $d_G(v,u)=1$  if  and  only  if $(v,u)\in  E$;  
  (iii)  $d_G(v,u)=+\infty$  if   and  only  if  $v$  and  $u$  are  not  connected  in  $G$;
   (iv)     for  any  $l\geq  1$,   taking  the  $l$-th  distance  power  $G^l$    
   (cf.  \cite{split-ind}),    we  have        $d_{G^l}(v,u)>1$  
    if  and  only  if       $d_{G}(v,u)>l$.   
  Note  that  $(V,d_G)$  is  a  metric  space  and $V$  is equipped  with     the  discrete  topology.  
 Hence a  map  $f:  (V,d_G)\longrightarrow \mathbb{F}^N$  is    $k$-regular   with  respect to  $r=1/2$  
 if  and  only  if     for  any  distinct  vertices  $v_1,\ldots,v_k\in V$  
  that   are  mutually  non-adjacent  in  $G$,  
 their images  $f(v_1)$, $\ldots$,  $f(v_l)$  are  linearly independent  in  $\mathbb{F}^N$. 
Moreover,   by  (iv),  
 the  $k$-regular  embedding  problem  of  $(V,d_{G^{l}})$ 
    into  $\mathbb{F}^N$  
    with respect  to  $r=1/2$   is  equivalent  to  the  
    $k$-regular  embedding  problem  of  $(V,d_G)$  into  $\mathbb{F}^N$
     with respect  to  $r=l/2$.  
    Therefore,  in  order  to  consider       $k$-regular  embeddings  
    of  $(V,d_G)$  with  respect  to  any  $r>0$  for  any  graph   $G$,
    it  is  sufficient  to   consider  the  special  case  
      $r=1/2$.   
 For  simplicity,  we  call  $f$   a  {\it  $k$-regular  map} on $G$  and  denote it  as 
 $f: G\longrightarrow \mathbb{F}^N$. 
 Moreover,  if  $f: G\longrightarrow \mathbb{F}^N$  is  $k$-regular  for  any  $k\geq  2$, 
 then  we say that  $f$  is  {\it   $G$-regular}.

An  {\it  independent  set}    of  $G$  is  a  set  $\sigma\subseteq  V$         
such  that  the  vertices  in  $\sigma$  are  mutually  non-adjacent.  
The  independence  complex  of  $G$   is  the  simplicial  complex  ${\rm  Ind}(G)$
whose  simplices  are  the  finite  independent  sets  of  $G$.  
We  define  the  {\it  directed  independence  complex}    $\overrightarrow{\rm  Ind}(G)$ 
as  the   hyperdigraph  on  $V$  whose  directed  hyperedges    $\vec{\sigma}$
are   finite   ordered    tuples  $(v_0,v_1,\ldots,v_k)$  of  distinct 
and  mutually  non-adjacent  vertices  
 such  that  the  underlying  set $\{v_0,v_1,\ldots,v_k\}$  is  an   independent  set  of  $G$.  
Then    $\overrightarrow{\rm  Ind}(G)$  is  $\Sigma_\bullet$-invariant  and 
  ${\rm  Ind}(G)$ 
  is  the  underlying  hypergraph  of  
     $\overrightarrow{\rm  Ind}(G)$
     with  the  canonical  projection  $\pi: \overrightarrow{\rm  Ind}(G)\longrightarrow 
     {\rm  Ind}(G)$.  
    An  equivalent   characterization  for     a  
     $k$-regular  embedding   of  $G$  into  $\mathbb{F}^N$  is  
     a  geometric  realization   of the  $(k-1)$-skeleton  of      $ {\rm  Ind}(G)$  in  $\mathbb{F}^N$;  
     and  an equivalent   characterization   for  a  $G$-regular  embedding   of  $G$ 
      into  $\mathbb{F}^N$ 
     is  a  geometric  realization   of   $ {\rm  Ind}(G)$  in  $\mathbb{F}^N$.

 For  any  finite  set  $S\subseteq  \mathbb{F}^N$,  
 let  $\mathcal{M}$  be  the  matroid  consisting  of  all  the   subsets  of  $S$ 
whose  vectors  are  linearly  independent.    
Then  $\mathcal{M}$ 
 is  an  {\it   augmented  simplicial  complex}  (cf.    \cite[Section~5]{comalg}),    
 i.e.  $\mathcal{M}\setminus\{\emptyset\}$  is  a  simplicial  complex.  
     We  define  the  {\it  directed  matroid}  $\vec{\mathcal{M}}$  
     as  the  collection  of  all   the   finite  ordered  tuples  of  the  vectors  in  $S$
   that   are  linearly  independent.  
   Then  $\vec{\mathcal{M}}\setminus \{\emptyset\}$  is  a  hyperdigraph.  
    There  is  a  canonical  projection  $\pi: \vec{\mathcal{M}}\longrightarrow 
    \mathcal{M}$  sending  each  ordered  tuple  to  its  underlying  set.    
    Hence 
      we  say that  $\mathcal{M}$  is  the {\it  underlying  matroid}  of  $\vec{\mathcal{M}}$.  
     An  equivalent   characterization  for      a    
     geometric  realization   of the  $(k-1)$-skeleton  of      $ {\rm  Ind}(G)$  in  $\mathbb{F}^N$
     is  a  simplicial  embedding  of   the  $(k-1)$-skeleton  of      $ {\rm  Ind}(G)$
     into  $\mathcal{M}$;  
     and  an equivalent   characterization   for   a geometric  realization   of   $ {\rm  Ind}(G)$  in  $\mathbb{F}^N$  is  a  simplicial  embedding  of   $ {\rm  Ind}(G)$
     into  $\mathcal{M}$.

 Let  $f:  G\longrightarrow  \mathbb{F}^N$  be  a  $k$-regular  embedding  
such  that  $f(V)\subseteq  S$.   
  The  first  main  result  of  this  paper  is  that  for  
  any  sub-hyperdigraph  $\vec{\mathcal{H}}$  of  ${\rm  sk}^{k-1}(\overrightarrow{\rm  Ind}(G))$
  whose  underlying  hypergraph  $\mathcal{H}$  is   a  sub-hypergraph  
  of  ${\rm  sk}^{k-1}{\rm  Ind}(G)$,  
         there   is  an  induced  commutative  diagram  of  homology  groups  
  \begin{eqnarray}\label{diag-251230-1}
\xymatrix{
H_\bullet(\vec{\mathcal{H}}) \ar[rr]^-{f_*} \ar[d]_-{\pi_*}
&&
 H_\bullet(\vec{\mathcal{M}})\ar[d]^-{\pi_*}\\
H_\bullet(\mathcal{H}) \ar[rr]^-{f_*}   &&
H_\bullet(\mathcal{M})  
}
\end{eqnarray}  
which  is  functorial  with  respect  to  morphisms  between  sub-hyperdigraphs  
of  ${\rm  sk}^{k-1}(\overrightarrow{\rm  Ind}(G))$
as  well  as  their  induced  morphisms  between  sub-hypergraphs  of 
          ${\rm  sk}^{k-1}{\rm  Ind}(G)$  (cf.  Theorem~\ref{diag-251121-h3}).  
   Here  $H_\bullet( {\mathcal{H}})$  is the  embedded  homology  of hypergraphs  
   (cf.  \cite{h1})  
   and   $H_\bullet(\vec{\mathcal{H}})$   is  the  embedded  homology  of  hyperdigraphs   
   (cf.  \cite{hdg}).  
   In  particular,   there   is  an  induced  commutative  diagram  of  homology  groups  
     (cf.  Theorem~\ref{diag-251121-h3})
  \begin{eqnarray} \label{diag-251230-2}
\xymatrix{
H_\bullet({\rm  sk}^{k-1}(\overrightarrow{\rm  Ind}(G))) \ar[rr]^-{f_*} \ar[d]_-{\pi_*}
&&
 H_\bullet(\vec{\mathcal{M}})\ar[d]^-{\pi_*}\\
H_\bullet({\rm  sk}^{k-1}{\rm  Ind}(G)) \ar[rr]^-{f_*}   &&
H_\bullet(\mathcal{M}).    
}
\end{eqnarray}  
The  diagrams  (\ref{diag-251230-1})  and  (\ref{diag-251230-2})
give  homological  obstructions  for  the  existence  of  $k$-regular  embeddings  
of  $G$  into  $\mathbb{F}^N$  such  that  the  images  of  the  vertices  are  in  $S$.

Let $G'$,  $G''$  and  $G'''$  be  graphs  with  disjoint  sets  of  vertices.  
Suppose   $f': G'\longrightarrow  \mathbb{F}^N$,   
 $f'': G''\longrightarrow  \mathbb{F}^N$  and  
 $f''': G'''\longrightarrow  \mathbb{F}^N$  are   $k$-regular  maps 
   such  that  the  images  of  the  vertices  are  in  $S$.  
   Let  $\tilde  *$  denote  the   join  of  graphs  (cf.   Example~\ref{ex-25-10-19-1}).     
   The  second  main  result  of  this  paper   is 
           that
             for  any  sub-hyperdigraph    $\vec{\mathcal{H}}'$  of  
             ${\rm  sk}^{k-1}(\overrightarrow{\rm  Ind}(G'  \tilde{*}  G'''))$        
           and  any   sub-hyperdigraph    $\vec{\mathcal{H}}''$  of  
             $ {\rm  sk}^{k-1}(\overrightarrow{\rm  Ind}(G''  \tilde{*}  G'''))$
             with  their  underlying  hypergraphs  $\mathcal{H}'$  and  $\mathcal{H}''$  respectively 
             such  that  (cf.  Subsection~\ref{ss-5.1})
              \begin{enumerate}[(I)]
 \item
  both $\vec{\mathcal{H}}'$  and  $\vec{\mathcal{H}}''$  are  $\Sigma_\bullet$-invariant, 
  \item
   for  any  $\sigma'\in \mathcal{H}'$  and  any  $\sigma''\in\mathcal{H}''$,  
   either  $\sigma'\cap\sigma''$  is  the  empty-set  or  
  $\sigma'\cap\sigma''\in \mathcal{H}'\cap\mathcal{H}''$,  
  \end{enumerate}
                   there  is   an  induced   commutative  diagram  of  Mayer-Vietoris  sequences
                   (in  (\ref{diag-251230-3})  and  (\ref{diag-251230-5}), 
                      MV   denotes     Mayer-Vietoris  type 
                   long  exact  sequences  of  homology  groups
                   and   the  arrows  are    homomorphisms  between  the  sequences) 
  \begin{eqnarray} \label{diag-251230-3}
 \xymatrix{
 {\bf  \rm  MV}(\vec{\mathcal{H}}', 
\vec{\mathcal{H}}'') 
\ar[d]_-{\pi_*} \ar[r] &  {\bf  \rm  MV}(\vec{\mathcal{M}}', \vec{\mathcal{M}}'') \ar[d]^-{\pi_*}\\
 {\bf  \rm  MV}( \mathcal{H}',  
 \mathcal{H}'')   \ar[r]  &  {\bf  \rm  MV}( {\mathcal{M}}',  {\mathcal{M}}'')  
 }
 \end{eqnarray}    
 where    $\vec{ \mathcal{M}}'$     and   $\vec{ \mathcal{M}}''$  
 are   the directed  matroids   of  all  the  finite  sequences   of  the linearly  independent  vectors 
   in  $ f'(V') \cup  f'''(V''')$   and   $ f'(V'') \cup  f'''(V''')$  respectively   with  their   
 underlying  matroids    $ \mathcal{M}'$  and   $ \mathcal{M}''$
 \footnote[2]{The  diagram  (\ref{diag-251230-3})  does  not  necessarily  satisfy  
 the  functoriality  with  respect  to  morphisms  between   hyperdigraphs  
 as  well  as their  induced  mophisms  of  the  underlying 
 hypergraphs  because  the  hypotheses
 (I)  and  (II)  may  not   be  functorial.  }
   (cf.  Theorem~\ref{th-mv-reg-hg}).    
  In  particular, 
              there  is   an  induced   commutative  diagram  of  Mayer-Vietoris  sequences 
 \begin{eqnarray} \label{diag-251230-5}
 \xymatrix{
 {\bf  \rm  MV}({\rm  sk}^{k-1}(\overrightarrow{\rm  Ind}(G'  \tilde{*}  G''')), 
 {\rm  sk}^{k-1}(\overrightarrow{\rm  Ind}(G''  \tilde{*}  G'''))) 
\ar[d]_-{\pi_*} \ar[r] &  {\bf  \rm  MV}(\vec{\mathcal{M}}', \vec{\mathcal{M}}'') \ar[d]^-{\pi_*}\\
 {\bf  \rm  MV}( {\rm  sk}^{k-1}({\rm  Ind}(G'  \tilde{*}  G''')),  
 {\rm  sk}^{k-1}({\rm  Ind}(G''  \tilde{*}  G''')))   \ar[r]  &  {\bf  \rm  MV}( {\mathcal{M}}',  {\mathcal{M}}'').  
 }
 \end{eqnarray}
The  diagrams  (\ref{diag-251230-3})  and  (\ref{diag-251230-5})
give    homological  obstructions  for  the  existence  of  $k$-regular  embeddings  
of  $G'$,  $G''$  and  $G'''$  into  $\mathbb{F}^N$  
such  that  the  images  of  the  vertices  are  in  $S$
  (cf.  Theorem~\ref{th-mv-reg}).

  Let  $G$  and  $G'$  be  graphs  with  disjoint sets  of  vertices.  
  Let  $S$  and  $S'$  be  finite  subsets  of  $\mathbb{F}^{N}$  and  $\mathbb{F}^{N'}$  
  respectively.  
  Suppose  $f: G\longrightarrow  \mathbb{F}^N$  
  is    $k$-regular  
  and   
  $f': G'\longrightarrow  \mathbb{F}^{N'}$
  is      $k'$-regular    
 such  that  $f(V)\subseteq  S$  and  $ f'(V')\subseteq  S'$.  
  The  third  main  result  of  this  paper  is  that  
  for  any sub-hyperdigraph  $\vec{\mathcal{H}}$
  of  ${\rm  sk}^{k-1}(\overrightarrow{\rm  Ind}(G))$  and  
any sub-hyperdigraph  $\vec{\mathcal{H}}'$
  of  ${\rm  sk}^{k-1}(\overrightarrow{\rm  Ind}(G'))$
with  their  underlying  hypergraphs  
$\mathcal{H}$  and  $\mathcal{H}'$  respectively,  
there  is  an  induced    commutative  diagram  of  short  exact  sequences  
  (in   (\ref{diag-251230-6})   and  (\ref{diag-251230-7}), 
    KU   denotes     K\"unneth   type 
                   short  exact  sequences  of  homology  groups
                   and   the  arrows  are   homomorphisms  between  the  sequences) 
\begin{eqnarray}\label{diag-251230-6}
 \xymatrix{
 {\bf  \rm  KU}(\vec{\mathcal{H}}, 
\vec{\mathcal{H}}') 
\ar[d]_-{\pi_*} \ar[r] &  {\bf  \rm  KU}(\vec{\mathcal{M}}, \vec{\mathcal{M}}') \ar[d]^-{\pi_*}\\
 {\bf  \rm  KU}( \mathcal{H},  
\mathcal{H})   \ar[r]  &  {\bf  \rm  KU}( {\mathcal{M}},  {\mathcal{M}}')
 }
 \end{eqnarray}
 which  is  functorial  with  respect  to  morphisms  between  sub-hyperdigraphs  
of  ${\rm  sk}^{k-1}(\overrightarrow{\rm  Ind}(G))$  and  
${\rm  sk}^{k-1}(\overrightarrow{\rm  Ind}(G'))$
as  well  as  their  induced  morphisms  between  sub-hypergraphs  of 
          ${\rm  sk}^{k-1}{\rm  Ind}(G)$  and
             ${\rm  sk}^{k-1}{\rm  Ind}(G')$   (cf.  Theorem~\ref{th-251111-ku-hg}).  
 In  particular, 
              there  is   an  induced   commutative  diagram  of  short  exact  sequences 
  \begin{eqnarray}\label{diag-251230-7}
 \xymatrix{
 {\bf  \rm  KU}({\rm  sk}^{k-1}(\overrightarrow{\rm  Ind}(G)), 
 {\rm  sk}^{k-1}(\overrightarrow{\rm  Ind}(G'))) 
\ar[d]_-{\pi_*} \ar[r] &  {\bf  \rm  KU}(\vec{\mathcal{M}}, \vec{\mathcal{M}}') \ar[d]^-{\pi_*}\\
 {\bf  \rm  KU}( {\rm  sk}^{k-1}({\rm  Ind}(G)),  
 {\rm  sk}^{k-1}({\rm  Ind}(G')))   \ar[r]  &  {\bf  \rm  KU}( {\mathcal{M}},  {\mathcal{M}}'). 
 }
 \end{eqnarray}
The  diagrams  (\ref{diag-251230-6})  and  (\ref{diag-251230-7})
give    homological  obstructions  for  the  existence  of  $k$-regular  embeddings  
of  $G $   into  $\mathbb{F}^N$   and  $G'$  into  $\mathbb{F}^{N'}$  
such  that  the  images  of  the  vertices  are  in  $S$  and  $S'$  respectively
  (cf.  Theorem~\ref{th-251111-ku}).

    The  results  in  this  paper    are  hopefully  to  be  applied    
    in   the  study  of  homological  obstructions  for  $k$-regular  embeddings 
  of  a  general  metric  space  $(X,d)$,    
   where   double  complexes   are  applied   instead  of  
   the  usual  chain  complexes  and  
   consequently     double  homology  or  the  Dolbeault  homology
   are  applied  
     instead  
   of  the  usual homology     (cf.  \cite{2025-reg}). 
   In  particular,  if   $(X,d)$  is  a  Riemannian  manifold  $(M,g)$
   with  the  geodesic  distance on  $M$  induced  by  $g$,  
   then  
   the  double  complexes  can  be  constructed  
   from  the  differential  forms  on  the  sub-hyperdigraphs  of  
   ${\rm  Conf}_\bullet(M)$,  where
   the  sub-hyperdigraphs  can  be  seen  as     
     graded  submanifolds   of    
     simplicial  manifolds    (cf.  \cite{jgp}).

  \subsection{Organization }

In  Section~\ref{s2},  we  give  some  basic  properties  and    examples  
for  the  regular  embeddings  of  graphs.  In  Sections~\ref{s3} - \ref{s5},  
 we  review  the  embedded  homology  of  hyper(di)graphs   and   
 the  homology  of  (lower)-associated   simplicial  complexes.   
 We  give  explicit  statements  on  the  Mayer-Vietoris  sequences
     as  well  as  the  K\"unneth-type  formulae  
     for   the  embedded  homology  of  hyper(di)graphs,  
     the  homology  of  (lower)-associated   directed  simplicial  complexes   and   
 the  homology  of  (lower)-associated   simplicial  complexes.            
                In  Section~\ref{s6},   by  applying  Section~\ref{s4}   to 
                the  independence  complexes and  the  directed  independence  complexes
                of  graphs,  
       we  study  the  homology  and   give  the  
       Mayer-Vietoris  sequences  as  well  as     the  K\"unneth-type  formulae  for  
         the  (directed)  independence  complexes.           
   In  Section~\ref{s7}, by  applying  Section~\ref{s4}   to 
               matroids  and  directed  matroids,    
       we  study  the  homology  and    give  the  
       Mayer-Vietoris  sequences  as  well as   the  K\"unneth-type  formulae  for  
         (directed)  matroids.     
                In  Section~\ref{s8},  we  prove the  main  results  of  this  paper
                with  the  help  of  Sections~\ref{s5} - \ref{s7}.      
                We  prove  the  first  main  result  in  Theorem~\ref{th-251107-1},  
                   prove  the  second  main  result  in  Theorem~\ref{th-mv-reg}  and  
                   Theorem~\ref{th-mv-reg-hg},  
                and   prove  the  third  main  result  in  Theorem~\ref{th-251111-ku}
                and  Theorem~\ref{th-251111-ku-hg}.

\section{Regular  maps  on  graphs}\label{s2}

In  this  section,  we  give  the  definition    of  (affinely)  regular  embeddings  of  graphs  
as  well  as   its   equivalent  characterization  by  using 
simplcial  embeddings  of  the  independence 
complexes  into  matroids.  
We  give  some  examples  for  regular  embeddings  of  graphs  
and  consequently  illustrate 
how  the  homotopy  types  of  the  independence  complexes 
 give  obstructions  for  the  existence  of  regular  embeddings  of   graphs.

Let   $G=(V,E)$  be  a  graph.  
Let  $\mathbb{F}$  be  a  field.  
Let $\mathbb{F}^N$ be the $N$-dimensional vector space over $\mathbb{F}$.
We  call  a map $f: V\longrightarrow  \mathbb{F}^N$      
{\it   (affinely)    $k$-regular}  with respect to  $G$     if   for  any  distinct   $k$-vertices 
$v_1,\ldots,v_k\in  V$ such that they are mutually non-adjacent in  $G$,  
their  images $f(v_1),\ldots, f(v_k)$ are    (affinely)   linearly  independent over $\mathbb{F}$. 
For  simplicity,  we  write  an   (affinely)   
{\it  $k$-regular}  map $f: V\longrightarrow  \mathbb{F}^N$
with respect to  $G$   as  $f:  G\longrightarrow   \mathbb{F}^N$ 
and  call  it  a   (affinely)   
{\it  $k$-regular}  map  on   $G$.  
 We say  that   a  map  $f: V\longrightarrow  \mathbb{F}^N$    is  
 (affinely)  
{\it  $G$-regular}  if  $f$  is   (affinely)   $k$-regular  with respect to  $G$   
for  any  positive  integer  $k$.

We  notice  that  when  $\mathbb{F}=  \mathbb{R}$  or  $\mathbb{C}$,  
we can  always  perturb  the  positions  of  the  images of the vertices     
such  that  $f(v_1),\ldots, f(v_N)$  are   linearly  independent  in  $\mathbb{F}^N$
  for  any  $v_1,\ldots,v_N\in  V$  and   $f(v_1),\ldots, f(v_{N+1})$  are   affinely  independent
   in  $\mathbb{F}^{N}$ 
  for  any  $v_1,\ldots,v_{N+1}\in  V$.  
  Thus  for  any  graph  $G$,  there  exist   
  $N$-regular  maps  $f:  G\longrightarrow  \mathbb{R}^N$
  and  complex  $N$-regular  maps  $f:  G\longrightarrow  \mathbb{C}^N$.  
  Moreover,  there  exist    
   affinely  $(N+1)$-regular   maps  $f:  G\longrightarrow  \mathbb{R}^N$
   and  complex   affinely  $(N+1)$-regular   maps  $f:  G\longrightarrow  \mathbb{C}^N$.   
  Furthermore,  for  any  $N\geq  3$  and    any  graph  $G$,   
  since  $G$  can  be  embedded  in  $\mathbb{R}^N$ 
  as  a  $1$-dimensional  geometric  simplicial  complex  
  such  that  the  images  of  the  
  edges  are  mutually  non-intersecting,  
  we can  always  perturb  the  positions  of  the  images of the vertices     
  such  that  after the  perturbation, 
    the  embedding  of  $G$   in  $\mathbb{R}^N$  without  edge-intersection
  is  an $N$-regular  map  on  $G$  or  is  an  
 affinely  $(N+1)$-regular   map  on  $G$.

   Let  $S$  be  a  fixed  finite  subset  of  $\mathbb{F}^{N}$.  
 Let  ${\rm  Ind}(G)$  be  the  independence  complex  of  $G$.  
 Let  $\mathcal{M}$  be  the  matroid  whose  independent sets are  
  subsets  of  $S$  consisting  of  linearly  independent  vectors  in  $\mathbb{F}^N$.  
We  have  the  followings
\begin{enumerate}[(1)]
\item
    any       $k$-regular  map 
    $f:  G\longrightarrow  \mathbb{F}^N$    such  that  
$f(V)\subseteq  S$   will  induce   a  simplicial  embedding 
    \begin{eqnarray*}
    {\rm  Ind}(f):    {\rm  sk}^{k-1}{\rm  Ind}(G)\longrightarrow  \mathcal{M}.    
    \end{eqnarray*}
    Conversely,   any  simplicial  embedding   of    
   $ {\rm  sk}^{k-1}{\rm  Ind}(G)$   into  $ \mathcal{M}$  
   will  give  a  $k$-regular  map 
   on  $G$  (cf.     Proposition~\ref{pr-251107-1});  
    \item 
      any       $G$-regular  map 
    $f:  G\longrightarrow  \mathbb{F}^N$    such  that  
$f(V)\subseteq  S$   will   induce   a  simplicial  embedding 
    \begin{eqnarray*}
    {\rm  Ind}(f):    {\rm  Ind}(G)\longrightarrow  \mathcal{M}.    
    \end{eqnarray*}
    Conversely,   any  simplicial  embedding   of    
   $ {\rm  Ind}(G)$   into  $ \mathcal{M}$  
   will  give  a  $G$-regular  map 
   on   $  G$    (cf.    Corollary~\ref{co-251107-1}).  
    \end{enumerate}
     Let  $\mathcal{A}$  be  the  matroid  whose  independent  sets  are  
  subsets  of  $S$  consisting  of  affinely    independent  vectors  in  $\mathbb{F}^N$
  (cf.  \cite[Section~1.3, Affine  dependence]{matroid}).
  Then  $\mathcal{A}$  is  a  simplicial  complex.    
    Similar to  (1)  and  (2) respectively,   the  followings  can  be  obtained  
    \begin{enumerate}[(1)']
\item
  any       affinely  $k$-regular  map 
    $f:  G\longrightarrow  \mathbb{F}^N$    such  that  
$f(V)\subseteq  S$   will  induce   a  simplicial  embedding  
    \begin{eqnarray*}
    {\rm  Ind}(f):    {\rm  sk}^{k-1}{\rm  Ind}(G)\longrightarrow  \mathcal{A}.   
    \end{eqnarray*}
    Conversely,   any   simplicial  embedding   of    
   $ {\rm  sk}^{k-1}{\rm  Ind}(G)$   into  $ \mathcal{A}$  
   will  give  an  affinely  $k$-regular  map 
   on  $G$;
    \item
      any      affinely   $G$-regular  map 
    $f:  G\longrightarrow  \mathbb{F}^N$    such  that  
$f(V)\subseteq  S$   will   induce   a  simplicial  map
    \begin{eqnarray*}
    {\rm  Ind}(f):    {\rm  Ind}(G)\longrightarrow  \mathcal{A}.    
    \end{eqnarray*}
     Conversely,   any  simplicial  embedding   of    
   $ {\rm  Ind}(G)$   into  $ \mathcal{A}$  
   will  give  an  affinely   $G$-regular  map 
   on   $  G$.
    \end{enumerate}
For  simplicity,   we   only  consider  the   regularity  condition  and  omit  the  
affinely  regularity  condition,  in  the  remaining  part  of  this  paper.

      \begin{example}
    \label{ex--251126-2.2}
    Let  $\mathbb{F}=\mathbb{R}$.  
    For  any  $R>0$,  
    let  
    \begin{eqnarray*}
    S_R=\{(x_1,\ldots,x_N)\in  \mathbb{R}^N\mid  x_1,\ldots,x_N\in \mathbb{Z}{\rm~and~}
    |x_1|, \ldots, |x_N|< R\}.
    \end{eqnarray*} 
    For  any  $N\geq  3$,  any  positive  integer  $k$  and   any  finite  graph  $G$  with   a  finite  vertex  set  $V$,
    we  claim  that  
    there  exists  a  $k$-regular  map $f:  G\longrightarrow \mathbb{R}^N$
      such  that  $f(V)\subseteq  S_R$  for  some  $R>0$.  
      In  fact,  we  choose      any  map  $f:  G\longrightarrow \mathbb{R}^N$  
    and  perturb  $f$  such  that  $f$  is    $k$-regular  and   
    the  image   of  $f$  has   rational  coordinates.  
    Hence  without  loss of  generality,  we   may  assume that  the  $k$-regular  map  is  given  by 
    $f:  G\longrightarrow \mathbb{Q}^N$. 
    For  sufficiently  large  $c$,  there  exists a  scalar  multiplication 
    by  $c$  such  that  $cf(V)\subseteq  \mathbb{Z}^N$,  i.e.   the  composition 
    \begin{eqnarray*}
    \xymatrix{
    G\ar[r]^-{f} &\mathbb{Q}^N \ar[r]^-{c  (-)}  &\mathbb{Q}^N
    }
    \end{eqnarray*}
    maps  $V$  into  $\mathbb{Z}^N$.  
    Since  $V$  is  finite,  there  exists   an  upper  bound  $R$  of  the  coordinates  
    of  $cf(V)$.  Therefore,      
      $cf:  G\longrightarrow  \mathbb{R}^N$  is  a  $k$-regular  map  satisfying   
     $cf(V)\subseteq  S_R$.  
     \end{example}

       In  the  next two  examples,  we  let  $G_1,\ldots,  G_n$  be  graphs  on   $V_1,\ldots, V_n$  respectively  
    where  $V_1,\ldots, V_n$   are  disjoint  sets  of  vertices.  
    
    \begin{example}
    \label{ex-25-12-16-1}
    Consider the   reduced   join  graph  $\tilde  *_{i=1}^n G_i$  
    obtained  by  adding  all  the  edges  
    $(v_k,v_l)$  for  any  $v_k\in  V_k$,  any  $v_l\in  V_l$  and  any  $1\leq  k<l\leq  n$  
    to  the  disjoint  union  $\bigsqcup_{i=1}^n  G_i$.  
    Then  ${\rm  Ind}(\tilde  *_{i=1}^n G_i)=\bigsqcup_{i=1}^n {\rm  Ind}(G_i)$.  
    The  followings  are  equivalent
    \begin{enumerate}[(a)]
    \item
        there  exists  a  simplicial  embedding    of  $\bigsqcup_{i=1}^n {\rm  sk}^{k-1}{\rm  Ind}(G_i)$  into  $\mathcal{M}$;  
    \item
    there  exists  a  $k$-regular  embedding  $f$  of  $\tilde  *_{i=1}^n G_i$  into  $\mathbb{F}^N$
     such that  $\bigsqcup_{i=1}^n  f(V_i)\subseteq  S$. 
     \end{enumerate}
    The  followings  are  equivalent
    \begin{enumerate}[(a)']
    \item
        there  exists  a  simplicial  embedding    of  $\bigsqcup_{i=1}^n {\rm  Ind}(G_i)$  into  $\mathcal{M}$;  
    \item
    there  exists  a  $(\tilde  *_{i=1}^n G_i)$-regular  embedding  $f$  of  $\tilde  *_{i=1}^n G_i$  into  $\mathbb{F}^N$
     such that  $\bigsqcup_{i=1}^n  f(V_i)\subseteq  S$. 
    \end{enumerate}
    \end{example}
    
    \begin{example}
        \label{ex-25-12-16-2}
     Consider the  (disjoint)  union   graph  $\bigsqcup  _{i=1}^n G_i$.   
     Then  ${\rm   Ind}(\bigsqcup  _{i=1}^n G_i)=  *_{i=1}^n{\rm  Ind}(G_i)$
     is  the  join of  the   independence  complexes  of  $G_1,\ldots,  G_n$    
     so  that 
     \begin{eqnarray*}
     {\rm  sk}^{k-1}(*_{i=1}^n{\rm  Ind}(G_i))=
      \bigcup_{\sum_{i=1}^n k_i=k} *_{i=1}^n\Big({\rm  sk}^{k_i-1}{\rm  Ind}(G_i) \Big)
      \end{eqnarray*}
         The  followings  are  equivalent
    \begin{enumerate}[(a)]
    \item
        there  exists  a  simplicial  embedding    of 
         $  \bigcup_{\sum_{i=1}^n k_i=k} *_{i=1}^n\Big({\rm  sk}^{k_i-1}{\rm  Ind}(G_i) \Big)$  into  $\mathcal{M}$;  
    \item
    there  exists  a  $k$-regular  embedding  $f$  of  $\bigsqcup  _{i=1}^n G_i$  into  $\mathbb{F}^N$
     such that  $\bigsqcup_{i=1}^n  f(V_i)\subseteq  S$. 
     \end{enumerate}
    The  followings  are  equivalent
    \begin{enumerate}[(a)']
    \item
        there  exists  a  simplicial  embedding    of  $*_{i=1}^n{\rm  Ind}(G_i)$  into  $\mathcal{M}$;  
    \item
    there  exists  a  $(\bigsqcup  _{i=1}^n G_i)$-regular  embedding  $f$  of  $\bigsqcup  _{i=1}^n G_i$  into  $\mathbb{F}^N$
     such that  $\bigsqcup_{i=1}^n  f(V_i)\subseteq  S$. 
    \end{enumerate} 
    \end{example}
    
    Applying  Example~\ref{ex-25-12-16-1}  and  Example~\ref{ex-25-12-16-2}
    to  the  $5$-cycle  and  the  $4$-path  respectively,  we  have  the  next  two  examples.  
    
    \begin{example}
    \label{ex--251211-2.3}
    Let  $C_5$  be  the  $5$-cycle with  vertices  $\{v_i\mid  1\leq  i\leq  5\}$
    and  edges  $\{(v_i,v_{i+1})\mid  1\leq  i\leq  5\}$,
      where  $v_6=v_1$.  Then  ${\rm  Ind}(C_5)=C_5$  whose  set  of  edges  is
    $\{(v_1,v_3),(v_1,v_4),(v_2,v_4), (v_2,v_5),(v_3,v_5)\}$.  
    For  any  field  $\mathbb{F}$,  
    any  vector  space  $\mathbb{F}^N$
    and  any  finite  set  $S\subseteq  \mathbb{F}^N$,  
    the  followings  are  equivalent
    \begin{enumerate}[(a)]
    \item
    there  exists  a  simplicial  embedding    of  $C_5$  into  $\mathcal{M}$;  
    \item
    there  exists  a  $2$-regular  embedding  $f$  of  $C_5$  into  $\mathbb{F}^N$ such that  $f(V)\subseteq  S$;
    \item
    there  exists  a  $C_5$-regular  embedding  $f$  of  $C_5$  into  $\mathbb{F}^N$
     such that  $f(V)\subseteq  S$.  
    \end{enumerate}
    By  Example~\ref{ex-25-12-16-1},  
        the  followings  are  equivalent
        \begin{enumerate}[(a)']
        \item
          there  exists  a  simplicial  embedding    of  $\bigsqcup_n   C_5$  into  $\mathcal{M}$; 
        \item
        there  exists  a  $2$-regular  embedding  $f$  of  $\tilde  *_n   C_5$  into  $\mathbb{F}^N$
     such that  $\bigsqcup_n  f(V(C_5))\subseteq  S$;
     \item
     there  exists  a  $(\tilde  *_n   C_5)$-regular  embedding  $f$  of  $\tilde  *_n   C_5$  into  $\mathbb{F}^N$
     such that  $\bigsqcup_n  f(V(C_5))\subseteq  S$.
             \end{enumerate}
        By  Example~\ref{ex-25-12-16-2},  
        the  followings  are  equivalent
        \begin{enumerate}[(a)"]
          \item
        there  exists  a  simplicial  embedding    of  $*_n  C_5$  into  $\mathcal{M}$;  
    \item
    there  exists  a  $(\bigsqcup  _n  C_5)$-regular  embedding  $f$  of  $\bigsqcup  _n  C_5$  into  $\mathbb{F}^N$
     such that  $\bigsqcup_n  f(V(C_5))\subseteq  S$. 
     \footnote[3]{The  (reduced)  join   $\tilde  *_n   C_5$  in  (b)'  is  a  graph  while  
     the  join   $*_n  C_5$  in  (a)"  is  a  simplicial  complex. 
     In  fact,   $\tilde  *_n   C_5$  is  the  $1$-skeleton  of   $*_n  C_5$  
     (cf.  Example~\ref{ex-25-10-19-1}).  }
        \end{enumerate}
    \end{example}

     \begin{example}
    \label{ex--251211-2.3a}
    Let  $P_4$  be  the  $4$-path  with  vertices  $\{v_i\mid  1\leq  i\leq  4\}$
    and  edges  $\{(v_i,v_{i+1})\mid  1\leq  i\leq  3\}$.  Then  ${\rm  Ind}(P_4)=P_4$  whose  set  of  edges  is
    $\{(v_2,v_4),(v_1,v_4),(v_1,v_3)\}$.  
    For  any  field  $\mathbb{F}$,  
    any  vector  space  $\mathbb{F}^N$
    and  any  finite  set  $S\subseteq  \mathbb{F}^N$,  
    the  followings  are  equivalent
    \begin{enumerate}[(a)]
    \item
    there  exists  a  simplicial  embedding    of  $P_4$  into  $\mathcal{M}$;  
    \item
    there  exists  a  $2$-regular  embedding  $f$  of  $P_4$  into  $\mathbb{F}^N$ such that  $f(V)\subseteq  S$;
    \item
    there  exists  a  $P_4$-regular  embedding  $f$  of  $P_4$  into  $\mathbb{F}^N$
     such that  $f(V)\subseteq  S$.  
    \end{enumerate}
    By  Example~\ref{ex-25-12-16-1},  
        the  followings  are  equivalent
        \begin{enumerate}[(a)']
        \item
          there  exists  a  simplicial  embedding    of  $\bigsqcup_n   P_4$  into  $\mathcal{M}$; 
        \item
        there  exists  a  $2$-regular  embedding  $f$  of  $\tilde  *_n   P_4$  into  $\mathbb{F}^N$
     such that  $\bigsqcup_n  f(V(P_4))\subseteq  S$;
     \item
     there  exists  a  $(\tilde  *_n   P_4)$-regular  embedding  $f$  of  $\tilde  *_n   P_4$  into  $\mathbb{F}^N$
     such that  $\bigsqcup_n  f(V(P_4))\subseteq  S$.
             \end{enumerate}
        By  Example~\ref{ex-25-12-16-2},  
        the  followings  are  equivalent
        \begin{enumerate}[(a)"]
          \item
        there  exists  a  simplicial  embedding    of  $*_n  P_4$  into  $\mathcal{M}$;  
    \item
    there  exists  a  $(\bigsqcup  _n  P_4)$-regular  embedding  $f$  of  $\bigsqcup  _n  P_4$  into  $\mathbb{F}^N$
     such that  $\bigsqcup_n  f(V(P_4))\subseteq  S$. 
        \end{enumerate}
    \end{example}

For  any  field  $\mathbb{F}$,  any  vector  space  $\mathbb{F}^N$  
and  any  finite  set  $S\subseteq  \mathbb{F}^N$,  
let  $|\mathcal{M}|$  be  the  geometric  realization  of  $\mathcal{M}\setminus\{\emptyset\}$  as  a  simplicial  complex.  
Consider  the  homotopy  types of  the  independence  complexes  of cycles  and  paths.  
We  have  the  next  two   examples.   

 \begin{example}
 Let  $C_n$  be  the  $n$-cycle  with  vertices  $\{v_1,v_2,\ldots,v_n\}$  
 and  edges  $\{(v_i,v_{i+1})\mid  1\leq  i\leq  n \}$   where  $v_{n+1}=v_1$.   
 Let  $P_n$  be  the  $n$-path   
   with  vertices  $\{v_1,v_2,\ldots,v_n\}$  
 and  edges  $\{(v_i,v_{i+1})\mid  1\leq  i\leq  n-1\}$.  
In  \cite{kozlov},  it  is  proved  that  
\begin{eqnarray*}
{\rm  Ind}(C_n)=\Sigma {\rm  Ind}(C_{n-3}),~~~~~~~ 
{\rm  Ind}(P_n)=\Sigma {\rm  Ind}(P_{n-3}), 
\end{eqnarray*}
which  agrees  with  \cite[Theorem~1.1]{split-ind}  and   \cite[Theorem~1.3]{zhwt-whh}.   
 By  a  direct   calculation,   
 \begin{eqnarray*}
  &{\rm  Ind}(C_4)=\bigsqcup_2  P_2\simeq  S^0, ~~~ 
   {\rm  Ind}(C_5)=C_5\simeq   S^1,\\
 & {\rm  Ind}(C_6)={\rm  clique ~complex~of~}   C_3\square  P_2  
 \simeq   \bigvee_2  S^1, \\
& {\rm  Ind}(P_3)=P_2\simeq  *, ~~~  {\rm  Ind}(P_4)=P_4\simeq   *,\\
& {\rm  Ind}(P_5)=C_3\cup  C_4 (C_3{\rm~and~}  C_4{\rm  have~a~common~edge}) \simeq  \bigvee_2  S^1.  
 \end{eqnarray*}
By  induction,  for  any  $m\geq  0$,   
\begin{eqnarray*}
{\rm  Ind}(C_n)\simeq  
\begin{cases}
S^m,  & n=3m+4\\
 S^{m+1}, &   n=3m+5\\
\bigvee_2  S^{m+1},   &n=3m+6,  
\end{cases}~~~~~~
{\rm  Ind}(P_n)\simeq  
\begin{cases}
*, & n=3m{\rm~or~} 3m+1\\
\bigvee_2  S^{m+1},  & n=3m+2.  
\end{cases}
\end{eqnarray*}
Therefore,  if  one  of  the  followings  is  satisfied
\begin{enumerate}[(a)]
\item
for  $n=3m+4$,  $S^m$  cannot  be  continuously  embedded in  $|\mathcal{M}|$, 
\item
for  $n=3m+5$,  $S^{m+1}$  cannot  be  continuously  embedded in  $|\mathcal{M}|$,
\item
for  $n=3m+6$,  $\bigvee_2  S^{m+1}$  cannot  be  continuously  embedded in  $|\mathcal{M}|$,
\end{enumerate}
then   
there  does  not  exist  any  $C_n$-regular  embedding  
$f$  of  $C_n$  into  $\mathbb{F}^N$  such that  $f(V)\subseteq  S$.  
Moreover,   if  $n=3m+2$  
and   $\bigvee_2  S^{m+1}$  cannot  be  continuously  embedded in  $|\mathcal{M}|$, 
then   
there  does  not  exist  any  $P_n$-regular  embedding  
$f$  of  $P_n$  into  $\mathbb{F}^N$  such that  $f(V)\subseteq  S$. 
    \end{example}

\begin{example}
   For  any  graph  $G$,  let  $G^r$  denote  the  $r$-th  distance  power  of  $G$
    with  the  same  vertex  set 
  where  two  vertices   are  adjacent  in  $G^r$  iff  their  distances  in  $G$ 
  is  smaller  than  or  equal to  $r$.     It  is  proved  in  \cite[Theorem~1.1]{split-ind}
  that for  any  $r\geq  1$  and  any $n\geq  5r+4$,  
  \begin{eqnarray*}
  {\rm  Ind}(C_n^r)\simeq  \Sigma^2 {\rm  Ind}(C_{n-(3r+3)}^r) \bigvee 
  X_{n,r}, 
  \end{eqnarray*} 
  where  $X_{n,r}$  is  given  explicitly  in  terms  of  ${\rm  Ind}(P_{n'}^r)$  for  some  $n'\leq  n$
  in \cite[Corollary~6.5]{split-ind}.  
  Let  $r=2$.  By  induction,  for  $n=9m+h$  where  $3\leq  h\leq  11$,  
    \begin{eqnarray*}
  {\rm  Ind}(C_n^2)\simeq  \Sigma^{2m} {\rm  Ind}(C_{h}^2) \bigvee \cdots 
  \end{eqnarray*} 
  Since  
  \begin{eqnarray*}
  {\rm  Ind}(C_{h}^2)=
  \begin{cases}
  \bigsqcup_3   P_2\simeq  \bigvee_2   S^0,  &  h=6\\
  C_7\simeq   S^1,  & h=7\\
  \simeq  \bigvee_5   S^1,  &h=8,   
  \end{cases}
  \end{eqnarray*}
  we  have  
   \begin{eqnarray*}
  {\rm  Ind}(C_{n}^2)\simeq  
  \begin{cases}
  (\bigvee_2   S^{2m})  \bigvee  \cdots,  &  n=9m+6\\
  S^{2m+1}  \bigvee  \cdots,  & n=9m+7\\
    (\bigvee_5   S^{2m+1} ) \bigvee  \cdots,  &n=9m+8.    
  \end{cases}
  \end{eqnarray*}
Therefore,  if  one  of  the  followings  is  satisfied
\begin{enumerate}[(a)]
\item
for  $n=9m+6$,  $\bigvee_2   S^{2m}$  cannot  be  continuously  embedded in  $|\mathcal{M}|$, 
\item
for  $n=9m+7$,  $ S^{2m+1}$  cannot  be  continuously  embedded in  $|\mathcal{M}|$,
\item
for  $n=9m+8$,  $\bigvee_5   S^{2m+1} $  cannot  be  continuously  embedded in  $|\mathcal{M}|$,
\end{enumerate}
then   
there  does  not  exist  any  $C_n^2$-regular  embedding  
$f$  of  $C_n^2$  into  $\mathbb{F}^N$  such that  $f(V)\subseteq  S$.  
\end{example}

From  the     last two  examples,  
we  see  that  the  homotopy  types   of  the      independence  complexes 
can  give  obstructions  for  the  regular embedding  problems  of  graphs.  
In  the  remaining  sections,  
we  will  give  obstructions  for  the  existence  of  regular  embeddings  of  graphs  by  using  
the  homology  of  the  (directed)  independence  complexes,  the  embedded  homology  
of  sub-hyper(di)graphs  of  the  (directed)  independence  complexes,  
some  Mayer-Vietoris  sequences 
 and  some  K\"unneth-type  formulae.

\section{Hypergraphs  and  hyperdigraphs}\label{s3}

In  this  section,  we  review  the  definitions  of  hypergraphs  and  hyperdigraphs.  
We  give  some  preliminaries  on  the  canonical  projections  
from  hyperdigraphs  to  the  underlying  hypergraphs.

Let  $V$  be  a    set  of  vertices.  
A  {\it  directed  $k$-hyperedge}  $\vec\sigma$  on  $V$  
is  an  ordered   $k$-tuple    
$(v_1,\ldots, v_k) $      
such  that  $v_1,\ldots,v_k\in   V$  are  distinct.  
A  {\it  $k$-uniform  hyperdigraph}  $\vec{\mathcal{H}}_k$  on  $V$ 
is  a  collection  of  directed  $k$-hyperedges.   
A  {\it  hyperdigraph}   $\vec{\mathcal{H}}$  on  $V$  is  a  union  
$\bigcup_{k\geq  1}\vec{\mathcal{H}}_k$ 
where  $\vec{\mathcal{H}}_k$  is  a   $k$-uniform  hyperdigraph  on  $V$  for  each  $k\geq  1$.  
A  {\it   $k$-hyperedge}  $ \sigma$  on  $V$  
is  an  unordered   $k$-tuple    
$\{v_1,\ldots, v_k\}$  such  that   $v_1,\ldots, v_k\in   V$  are  distinct.  
A  {\it   $k$-uniform   hypergraph}  $\mathcal{H}$  on  $V$ 
is  a  collection  of   $k$-hyperedges.   
A  {\it  hypergraph}   $  \mathcal{H} $  on  $V$  is  a  union  
$\bigcup_{k\geq  1} \mathcal{H}_k$ 
where  $ \mathcal{H}_k$  is  a   $k$-uniform  hypergraph  on  $V$  for  each  $k\geq  1$.

 Let  $\vec\sigma=(v_1,\ldots,v_k)$  be  a  directed  $k$-hyperedge. 
The  $k$-th  symmetric  group  $\Sigma_k$  acts  on  $\vec\sigma$  by  permuting  the  
order  of  the  vertices  of  $\vec\sigma$.   
 For  any  $s\in\Sigma_k$,
 let   $s(\vec\sigma)= (v_{s(1)},\ldots,v_{s(k)})$. 
The  $\Sigma_k$-orbit  of  $\vec\sigma$  is  
$
\Sigma_k(\vec\sigma)= \{s(\vec\sigma)\mid s\in\Sigma_k\}$,
which  will be  identified  with  the  $k$-hyperedge  
$\sigma=\{v_1,\ldots, v_k\}$.

Let   $\vec{\mathcal{H}}_k$   be  a  $k$-uniform  hyperdigraph. 
We  say  that  $\vec{\mathcal{H}}_k$   is  {\it  $\Sigma_k$-invariant}   
 if   $\Sigma_k(\vec\sigma)\subseteq  \vec{\mathcal{H}}_k$ 
  for any  $\vec\sigma\in  \vec{\mathcal{H}}_k$.    
Equivalently,  $\vec{\mathcal{H}}_k$   is  $\Sigma_k$-invariant   if   
it  is  a  disjoint  union  of   $\Sigma_k$-orbits.   
For  any 
     $\Sigma_k$-invariant  hyperdigraph  $\vec{\mathcal{H}}_k$, 
     the    orbit  space  
  $\vec{\mathcal{H}}_k/\Sigma_k$ is  a $k$-uniform  hypergraph  ${\mathcal{H}}_k$  
  such  that  the   hyperedges  of  ${\mathcal{H}}_k$  are  $\sigma=\Sigma_k(\vec\sigma)$,
    where 
   $\vec\sigma\in  \vec{\mathcal{H}}_k$.  
 In  this  case,  the   canonical  projection 
 $\pi_k:  \vec{\mathcal{H}}_k\longrightarrow  {\mathcal{H}}_k$
 gives  a  principal  $\Sigma_k$-bundle.

 Let $\vec{\mathcal{H}}_k$  be  any  $k$-uniform  hyperdigraph
 which  is  not  supposed to be  $\Sigma_k$-invariant.   
 We     choose  the  smallest   ambient   
   $\Sigma_k$-invariant $k$-uniform  hyperdigraph  ${\rm  Sym}(\vec{\mathcal{H}}_k)$
   containing  $\vec{\mathcal{H}}_k$.  
 Then  ${\rm  Sym}(\vec{\mathcal{H}}_k)$  consists  of  the  
   $\Sigma_k$-orbits  of  the  directed  hyperedges  in  $\vec{\mathcal{H}}_k$. 
   We  have the  canonical  projection  
   \begin{eqnarray}\label{eq-25-10-13-1}
   \pi_k:  {\rm  Sym}(\vec{\mathcal{H}}_k)\longrightarrow 
    {\rm  Sym}(\vec{\mathcal{H}}_k)/\Sigma_k. 
    \end{eqnarray} 
    With the  help  of  \cite[Lemma~3.1]{hdg}, 
    the  
     restriction     of  (\ref{eq-25-10-13-1})    
  to  $\vec{\mathcal{H}}_k$  induces  a   projection 
       \begin{eqnarray}\label{eq-25-9.21.1}
       \pi_k:  \vec{\mathcal{H}}_k\longrightarrow  {\mathcal{H}}_k,
       \end{eqnarray}
       where  
       \begin{eqnarray}\label{eq-25-10-13-5}
       {\mathcal{H}}_k=\pi _k(\vec{\mathcal{H}}_k)=  {\rm  Sym}(\vec{\mathcal{H}}_k)/\Sigma_k.
       \end{eqnarray} 
       Note that  in       (\ref{eq-25-9.21.1}),  the  followings  are  equivalent:   
       (1)  $\pi_k$  is   a  principal  $\Sigma_k$-bundle;  
        (2)  $\vec{\mathcal{H}}_k$ 
   is  $\Sigma_k$-invariant;      (3) $\vec{\mathcal{H}}_k={\rm  Sym}(\vec{\mathcal{H}}_k)$.  
          We  say  that   the  $k$-uniform  hypergraph  $\mathcal{H}_k$  
   is  the  {\it  underlying  hypergraph}  of  $\vec{\mathcal{H}}_k$   
   (or  the  {\it  reduced  hypergraph}  according  to  \cite[Subsection~3.2]{hdg}).

   In   general,  for  any  hyperdigraph  $\vec{\mathcal{H}}=\bigcup_{k\geq  1} \vec{\mathcal{H}}_k$,  
   let  $\mathcal{H}= \bigcup_{k\geq  1} \mathcal{H}_k$, 
   where  $\mathcal{H}_k$  is  the underlying  hypergraph of  $\vec{\mathcal{H}}_k$
   given  by  (\ref{eq-25-10-13-5})  
   for  each  $k\geq  1$.  
   We  say that  $\mathcal{H}$  is  the  underlying  hypergraph  of  $\vec{\mathcal{H}}$
   and  write  $\mathcal{H}=\pi_\bullet(\vec{\mathcal{H}})$.  
   We  say  that   $\vec{\mathcal{H}}$  is  {\it  $\Sigma_\bullet$-invariant} 
   if    $\vec{\mathcal{H}}_k$  is  $\Sigma_k$-invariant  for  each  $k\geq  1$.

        Let  $V$  and   $V'$  be  disjoint  sets.  
        Suppose  both  $V$ and  $V'$  have  total  order  $\prec$.  
        Then  we  have an  induced   total  order  on  $V\sqcup  V'$  such  that 
         (1)  it  is    the  original  total  order    restricted to  $V$  or   $V'$,   and  (2)
        $v\prec  v'$  for  any  $v\in  V$  and  any  $v'\in  V'$.  
        Let  $\vec \sigma=(v_1,\ldots,v_k)$ be  a  directed $k$-hyperedge  on  $V$  
        and  let  $\vec\sigma'=(v'_1,\ldots,v'_l)$  be a  directed  $l$-hyperedge  on  $V'$.  
        Define  their  {\it  join}  to  be  a  directed  $(k+l)$-hyperedge  on  $V\sqcup  V'$  
        given  by  the   Cartesian  product   of  coordinates 
        \begin{eqnarray*}
\vec  {\sigma}* \vec{\sigma}' = (v_1,\ldots,v_k,  v'_1,\ldots,v'_l).  
\end{eqnarray*}        
 Let  $\vec { \mathcal{H}}$   and  $ \vec{\mathcal{H}}'$  be   
 hyperdigraphs  on  $V$  and   $V'$  respectively.  
Define  the  {\it  join}  of  $\vec { \mathcal{H}}$   and  $ \vec{\mathcal{H}}'$ to  be  
a  hyperdigraph  on  $V\sqcup  V'$  given  by  
\begin{eqnarray}\label{eq-25-10-91}
\vec { \mathcal{H}}  * \vec { \mathcal{H}'}  =  \vec { \mathcal{H}} \sqcup \vec { \mathcal{H}'}
\sqcup \{\vec \sigma* \vec\sigma'\mid  \sigma\in \vec{\mathcal{H}}, \sigma'\in \vec{\mathcal{H}'}\}.  
\end{eqnarray}
 Let  $  \sigma=\{v_1,\ldots,v_k\}$ be  a   $k$-hyperedge  on  $V$  
        and  let  $ \sigma'=\{v'_1,\ldots,v'_l\}$  be a    $l$-hyperedge  on  $V'$.  
        Their  {\it  join}  is  a    $(k+l)$-hyperedge  on  $V\sqcup  V'$  given  
        by  their  disjoint  union  
        \begin{eqnarray*}
\sigma* \sigma' = \{v_1,\ldots,v_k,  v'_1,\ldots,v'_l\}.  
\end{eqnarray*}        
 Let  $ \mathcal{H}$   and  $\mathcal{H}'$  be   
 hypergraphs  on  $V$  and   $V'$  respectively.  
Their  {\it  join}  is    
a  hypergraph  on  $V\sqcup  V'$  given  by  
\begin{eqnarray}\label{eq-25-10-92}
 { \mathcal{H}}  *   { \mathcal{H}'}  =   { \mathcal{H}} \sqcup   { \mathcal{H}'}
\sqcup \{  \sigma*  \sigma'\mid  \sigma\in  {\mathcal{H}}, \sigma'\in  {\mathcal{H}'}\}.  
\end{eqnarray}
It  is  clear  that  the  join  of  hyper(di)graphs  satisfies 
the  commutativity  law 
\begin{eqnarray}\label{eq-25-10-19-11}
\vec {\mathcal{H}}  *   \vec{\mathcal{H}}'  \cong 
\vec{ \mathcal{H}}'  *  \vec{\mathcal{H}},~~~~~~  
 \mathcal{H}  *  \mathcal{H}'  \cong  \mathcal{H}'  *  \mathcal{H}  
\end{eqnarray}
and 
the  associativity  law 
\begin{eqnarray}\label{eq-25-10-19-12}
(\vec{\mathcal{H}}  *  \vec{\mathcal{H}}' ) * \vec{\mathcal{H}}''=
\vec{ \mathcal{H}}  * ( \vec{\mathcal{H}}'   * \vec{\mathcal{H}}''),~~~~~~
(\mathcal{H}  *  \mathcal{H}' ) * \mathcal{H}''= \mathcal{H}  * ( \mathcal{H}'   * \mathcal{H}''). 
\end{eqnarray}

\begin{lemma}
\label{pr-25-10-11-1}
\begin{enumerate}[(1)]
\item
Let  $\vec{\mathcal{H}}$  and  $\vec{\mathcal{H}}'$  be  hyperdigraphs  on  $V$.  
Then  
\begin{eqnarray}
  \pi_\bullet(\vec{\mathcal{H}}\cup \vec{\mathcal{H}}')  &=& 
 \pi_\bullet(\vec{\mathcal{H}})\cup \pi_\bullet(\vec{\mathcal{H}}'),
 \label{eq-25-10-11-2}\\ 
 \pi_\bullet(\vec{\mathcal{H}}\cap \vec{\mathcal{H}}')  &\subseteq& 
 \pi_\bullet(\vec{\mathcal{H}})\cap \pi_\bullet(\vec{\mathcal{H}}')
  \label{eq-25-10-11-3}
\end{eqnarray}
and  the  equality in  (\ref{eq-25-10-11-3})     
is  satisfied  if  both $\vec{\mathcal{H}}$  and  $ \vec{\mathcal{H}}'$  are $\Sigma_\bullet$-invariant; 
\item
Let  $\vec { \mathcal{H}}$   and  $ \vec{\mathcal{H}}'$  be   
 hyperdigraphs  on  $V$  and   $V'$  respectively  where  $V$  and  $V'$ are  disjoint.  
 Then   
 \begin{eqnarray} \label{eq-25-10-11-5}
\pi_\bullet(\vec{\mathcal{H}}* \vec{\mathcal{H}}')  = 
 \pi_\bullet(\vec{\mathcal{H}})* \pi_\bullet(\vec{\mathcal{H}}').    
\end{eqnarray}
\end{enumerate}
\end{lemma}

\begin{proof}
(1)  
It  is  direct  to  have   (\ref{eq-25-10-11-2}) - (\ref{eq-25-10-11-5}). 
Suppose  in  addition  
that  both $\vec{\mathcal{H}}$  and  $ \vec{\mathcal{H}}'$  are $\Sigma_\bullet$-invariant. 
Then  both  $\vec{\mathcal{H}}\cup \vec{\mathcal{H}}'$   and  
$\vec{\mathcal{H}}\cap \vec{\mathcal{H}}'$  are  $\Sigma_\bullet$-invariant. 
Thus  we  have a  principal  $\Sigma_k$-bundle 
\begin{eqnarray*}
\Sigma_k\longrightarrow  \vec{\mathcal{H}}_k\cap \vec{\mathcal{H}}'_k\longrightarrow 
 \pi_k(\vec{\mathcal{H}}_k)\cap \pi_k(\vec{\mathcal{H}}'_k)
\end{eqnarray*}
for  any  $k\geq  1$.  
Hence  
$
\pi_k(\vec{\mathcal{H}}_k\cap \vec{\mathcal{H}}'_k)=  \pi_k(\vec{\mathcal{H}}_k)\cap \pi_k(\vec{\mathcal{H}}'_k)
$ 
for  any  $k\geq  1$,  which  implies  the  equality  in    (\ref{eq-25-10-11-3}).  
\end{proof}

Let  $\vec{\mathcal{H}}$  be  a  hyperdigraph on  $V$  and  
let   $\vec{\mathcal{H}}'$  be  a  hyperdigraph on  $V'$.   
We  define  a  {\it  morphism}  of  hyperdigraphs $\vec{\varphi}:  \vec{\mathcal{H}}\longrightarrow  \vec{\mathcal{H}'}$  
to  be   a  map  $\varphi:   V\longrightarrow  V'$  such  that  
for  any  directed  hyperedge  $\vec\sigma\in  \vec{\mathcal{H}}$,  
say  $\vec\sigma=(v_1,\ldots,  v_k)$,  
its  image  
\begin{eqnarray*}
\vec{\varphi}(\vec\sigma)=(\varphi(v_1),\ldots,  \varphi(v_k))
\end{eqnarray*}
 is  a directed  hyperedge  of  $ \vec{\mathcal{H}}'$
 (cf.  \cite[Subsection~3.1]{hdg}).  
 Let  $\mathcal{H}$  and  $\mathcal{H}'$   be  the  underlying  hypergraphs  
 of  $\vec{\mathcal{H}}$     and  
    $\vec{\mathcal{H}}'$   respectively.  
    A  {\it  morphism}  of  hypergraphs  
    ${\varphi}:   \mathcal{H} \longrightarrow   \mathcal{H}' $  
is  a  map  $\varphi:   V\longrightarrow  V'$  such  that  
for  any   hyperedge  $ \sigma\in   \mathcal{H} $,  
where  $ \sigma=\{v_1,\ldots,  v_k\}$,  
its  image  
\begin{eqnarray*}
{\varphi}( \sigma)=\{\varphi(v_1),\ldots,  \varphi(v_k)\}
\end{eqnarray*}
 is  a   hyperedge  of  $  \mathcal{H}'$.  
 Any   morphism  of  hyperdigraphs  
 $\vec{\varphi}:  \vec{\mathcal{H}}\longrightarrow  \vec{\mathcal{H}'}$
 induces  a  morphism  between   the   underlying  hypergraphs
  ${\varphi}:  \mathcal{H} \longrightarrow   \mathcal{H}' $  
  such  that  the  diagram  commutes 
  \begin{eqnarray}\label{eq-1103-1}
  \xymatrix{
  \vec{\mathcal{H}}\ar[r]^-{\vec{\varphi}}  \ar[d]_-{\pi}  &  \vec{\mathcal{H}'} \ar[d]^-{\pi}\\
  \mathcal{H} \ar[r]^-{{\varphi}}  &  \mathcal{H}'.  
  }
  \end{eqnarray}

   \section{Simplicial  complexes,  directed  simplicial  complexes  and  their  homology}\label{s4}

   In  this  section,   
   we  give  some  preliminaries  
   on  the  homology  of  simplicial  complexes  and  directed  simplicial  complexes.  
    We  discuss  about a  canonical  homomorphism from  the  homology  
    of  a directed  simplicial  complex to 
    the  homology  of   the  underlying  simplicial  complex  in  Proposition~\ref{le-25-10-5}.  
    We  formalize  some  Mayer-Vietoris  sequences  in  
   Proposition~\ref{pr-10-11-mv-1}
   and  some  K\"unneth-type  formulae  in  Proposition~\ref{pr-10-12-kf-1}.

A  {\it  directed  simplicial  complex}  $\vec{ \mathcal{K}}$  on   $V$  
is  a  hyperdigraph    such  that  
if   
$\vec\sigma\in\vec {\mathcal{K}}$  is  a  directed  hyperedge  and
  $\vec \tau$    is  a nonempty  subsequence  of  $\vec  \sigma$,  
then  $\vec \tau\in \vec{\mathcal{K}}$.  
We  call  a  directed  $k$-hyperedge  in   $\vec{ \mathcal{K}}$  a
{\it  directed    
$(k-1)$-simplex}.  
A  {\it   simplicial  complex}  $ \mathcal{K} $  on   $V$  
is  a  hypergraph    such  that  
 if   
$ \sigma\in \mathcal{K} $  is  a  hyperedge  and   $  \tau$   is a   nonempty  subset
of  $\sigma$,  
 then    $  \tau\in {\mathcal{K}}$  (cf. \cite[p. 107]{hatcher}).  
 We  call  a    $k$-hyperedge  in   $  \mathcal{K} $  a
{\it    
$(k-1)$-simplex}   (cf. \cite[p. 103]{hatcher}).

Let   $\vec{\mathcal{K}}$  and  
$\vec{\mathcal{K}}'$  be     directed  simplicial  complexes  on  $V$  and  $V'$  respectively.  
We  call  a  morphism  of  hyperdigraphs
  $\vec{\varphi}:  \vec{\mathcal{K}}\longrightarrow \vec{\mathcal{K}}'$  is  
a  {\it  directed  simplicial  map}.  
Let   $ \mathcal{K} $  and  
$  \mathcal{K}'$  be   the   underlying   simplicial  complexes  
of  $\vec{\mathcal{K}}$  and  
$\vec{\mathcal{K}}'$  respectively.  By (\ref{eq-1103-1}),  
 any   directed  simplicial  map    
 $\vec{\varphi}:  \vec{\mathcal{K}}\longrightarrow  \vec{\mathcal{K}'}$
 induces  a  simplicial  map   between   the  underlying  simplicial  complexes  
  ${\varphi}:  \mathcal{K} \longrightarrow   \mathcal{K}' $  
  such  that  the  diagram  commutes 
  \begin{eqnarray*}
  \xymatrix{
  \vec{\mathcal{K}}\ar[r]^-{\vec{\varphi}}  \ar[d]_-{\pi}  &  \vec{\mathcal{K}'} \ar[d]^-{\pi}\\
  \mathcal{K} \ar[r]^-{{\varphi}}  &  \mathcal{K}'.  
  }
  \end{eqnarray*}

\begin{lemma}\label{le-25.9.1}
Let  $\vec {\mathcal{H}}$  be  a  hyperdigraph and  let $\mathcal{H}=\pi_\bullet(\vec{\mathcal{H}})$. 
\begin{enumerate}[(1)]
\item
If  $\vec{\mathcal{H}}$  is a   directed  simplicial  complex, 
then  $\mathcal{H}$  is  a simplicial  complex;  
\item
If  $\vec{\mathcal{H}}$  is   $\Sigma_\bullet$-invariant  and  
 $\mathcal{H}$  is  a  simplicial  complex,  then 
$\vec{\mathcal{H}}$  is a   directed  simplicial  complex.  
\end{enumerate}
\end{lemma}

\begin{proof}
(1)  Suppose  $\vec{\mathcal{H}}$  is a   directed  simplicial  complex.  
 Let  $\sigma\in  \mathcal{H}$.  
 Choose   $\vec \sigma\in \vec{\mathcal{H}}$  such  that  
 $\pi(\vec\sigma)=\sigma$.  
 Then 
 for  any  nonempty  subset  $\tau$  of  $\sigma$,  
 there  exists a  nonempty  
 subsequence  $\vec  \tau$  of  $\vec \sigma$  such  that  
 $\pi(\vec\tau)=\tau$.  
  Thus 
 $\vec\tau\in \vec{\mathcal{H}}$.  
 Hence   $\tau\in \mathcal{H}$. 
   Consequently,  $\mathcal{H}$  is  a simplicial  complex.

(2)  Suppose $\vec{\mathcal{H}}$  is   $\Sigma_\bullet$-invariant.  
Then  for  any  $\eta\in \mathcal{H}$  and    
     any  directed  hyperedge  $\vec\eta$  on   $V$    
  such  that  $\pi(\vec\eta)=\eta$,  we  have  $\vec \eta\in \vec{\mathcal{H}}$.    
Suppose    
 $\mathcal{H}$  is  a  simplicial  complex.  
 Let  $\vec\sigma\in  \vec{\mathcal{H}}$  and  let  $\sigma=\pi(\vec\sigma)$. 
 Then  $\sigma\in \mathcal{H}$.   
For  any  nonempty  subsequence  $\vec\tau$  of  $\vec\sigma$,  
   $\tau=\pi(\vec\tau)$  is  a  nonempty  subset  of  $\sigma$.  
Thus    $\tau\in \mathcal{H}$,  which  implies  $\vec\tau\in \vec{\mathcal{H}}$.     
Consequently,   $\vec{\mathcal{H}}$  is a   directed  simplicial  complex.  
\end{proof}

Let  $R$  be  a   principal ideal domain.  
Let  $\vec {\mathcal{K}}$  be  a  directed  simplicial  complex  on  $V$.  
Let  $C_{k}(\vec {\mathcal{K}};R)$  be  the  free  $R$-module  spanned  by  all  the  
directed  $k$-simplices  of  $\vec {\mathcal{K}}$.  
Consider  the  $R$-linear  map      
 $\vec  {\partial}_k:   C_{k }(\vec {\mathcal{K}};R)\longrightarrow C_{k-1}(\vec {\mathcal{K}};R)$  
given  by  
\begin{eqnarray}\label{eq-25-10-3}
\vec  {\partial}_k(v_0,\ldots,v_k)= \sum_{i=0}^k (-1)^{i}  (v_0,\ldots,\widehat{v_i},\ldots, v_k). 
\end{eqnarray}
\begin{lemma}\label{le-25-10-13-7}
 $\vec{\partial}_{k-1}\circ\vec{\partial}_k=0$  thus  
 $C_\bullet(\vec {\mathcal{K}};R)=\{C_{k}(\vec {\mathcal{K}};R), \partial_k\}_{k\geq  0}$ 
is  a  chain  complex.  
\end{lemma}
\begin{proof}
Let  $(v_0,v_1,\ldots,v_k)$  be  a  generator  of $C_{k}(\vec {\mathcal{K}};R)$.   
It  is   direct    that 
\begin{eqnarray*}
\vec{\partial}_{k-1}\circ\vec{\partial}_k(v_0,v_1,\ldots,v_k)&=&
\sum_{i=0}^k (-1)^{i}  \vec{\partial}_{k-1} (v_0,\ldots,\widehat{v_i},\ldots, v_k)\\
&=&\sum_{i=0}^k (-1)^{i}  \sum_{j=0}^{i-1}(-1)^j
 (v_0,\ldots,\widehat{v_j},  \ldots,\widehat{v_i},\ldots, v_k)\\
&& + \sum_{i=0}^k (-1)^{i} \sum_{j=i+1}^{k}(-1)^{j-1}  (v_0,\ldots,\widehat{v_i},  \ldots,\widehat{v_j},\ldots, v_k)\\
&=&\sum_{0\leq  j<i\leq  k}(-1)^{i+j} (v_0,\ldots,\widehat{v_j},  \ldots,\widehat{v_i},\ldots, v_k)\\
&& + \sum_{0\leq  i<j\leq  k} (-1)^{i+j-1} (v_0,\ldots,\widehat{v_i},  \ldots,\widehat{v_j},\ldots, v_k)\\
&=&0. 
\end{eqnarray*} 
 With  the  help  of  the  $R$-linearity  of  $\vec{\partial}_{k-1}\circ\vec{\partial}_k$,
   the  proof  follows.   
\end{proof}
By  Lemma~\ref{le-25.9.1}~(1),  
the  underlying  hypergraph  of  $\vec{\mathcal{K}}$  is 
 a  simplicial  complex  $\mathcal{K}$  on   $V$.  
Let  $C_{k }( \mathcal{K};R)$  be  the  free  $R$-module  spanned  by  all  the  
  $k$-simplices  of  $ \mathcal{K}$.  
Let  $\prec$  be  a  total  order  on  $V$.  
We  have a   chain  complex 
 $C_\bullet(  {\mathcal{K}};R)=\{C_{k}( {\mathcal{K}};R), \partial_k\}_{k\geq  0}$
 with  the  boundary  map        
 $  {\partial}_k:   C_{k }( {\mathcal{K}};R)\longrightarrow C_{k-1}( {\mathcal{K}};R)$  
given  by  
\begin{eqnarray}\label{eq-25-10-4}
  {\partial}_k\{v_0,\ldots,v_k\}= \sum_{i=0}^k (-1)^{i}  \{v_0,\ldots,\widehat{v_i},\ldots, v_k\} 
\end{eqnarray}
for  any  $k$-simplices  $\{v_0,\ldots,v_k\}\in \mathcal{K}$  with  $v_0\prec \cdots\prec  v_k$.     
 The  projection 
 $\pi_k:  \vec{ \mathcal{K}}_k\longrightarrow  \mathcal{K}_k$  induces  
 an  $R$-linear  projection
 \begin{eqnarray}\label{eq-10-1-9}
 (\pi_k)_*:  C_{k-1}(\vec{ \mathcal{K}};R)\longrightarrow  C_{k-1}( \mathcal{K};R) 
 \end{eqnarray}
given  by 
\begin{eqnarray}\label{eq-25-10-5}
(\pi_k)_\#(v_{s(0)},\ldots,v_{s(k-1)})=
{\rm  sgn}(s)\{v_0,\ldots,v_{k-1}\}
\end{eqnarray}
 for  any  directed  $(k-1)$-simplices  $(v_0,\ldots,v_{k-1})\in \vec{ \mathcal{K}}_k$
 satisfying   $v_0\prec  \cdots\prec  v_{k-1}$  and  any  $s\in \Sigma_k$  permuting  
 the  order  of  
 $0,\ldots,k-1$.  
 Here ${\rm  sgn}(s)$  is  $ 1$  if  $s$  is  an  even  permutation  and  is  $-1$  if  $s$  is  an  odd  permutation.  
 The  next  proposition  is  extracted  from  \cite{hdg}.  
 
 \begin{proposition}\label{le-25-10-5}
 Let  $\vec{\mathcal{K}}$  be  a  directed  simplicial  complex  on  $V$
and  let  $\mathcal{K}$  be  its  underlying  simplicial  complex.  
 Then  we  have a    surjective  chain  map 
 \begin{eqnarray}\label{eq-25-10-7}
 (\pi_{\bullet+1})_\#:   C_\bullet(\vec {\mathcal{K}};R)\longrightarrow  C_\bullet({\mathcal{K}};R)
\end{eqnarray}
    and   consequently  an  induced      homomorphism  of  homology   
\begin{eqnarray}\label{eq-25-10-8}
 (\pi_{\bullet+1})_*:   H_\bullet(\vec {\mathcal{K}};R)\longrightarrow  H_\bullet({\mathcal{K}};R).  
\end{eqnarray}
Moreover,    (\ref{eq-25-10-7})   and   (\ref{eq-25-10-8})  are   functorial  with  respect  to 
directed  simplicial  maps  between  directed  simplicial  complexes   
 and  their  induced  simplicial  maps   between  the  underlying  simplicial  complexes.  
 \end{proposition}
 
 \begin{proof}
 The  proof  is  analogous  with  \cite[Lemma~3.8  and  Theorem~3.9]{hdg}.
 We  give  the  detailed  proof  for  completeness.   
 Let  $v_0 \prec  v_1\prec \cdots  \prec  v_k$  be  vertices  in  $V$.
 Let  $s\in \Sigma_{k+1}$   permuting  the  order  of  $v_0,v_1,\ldots,v_k$.  
 Suppose  $(v_{s(0)},v_{s(1)},\ldots,v_{s(k)}) \in  \vec{\mathcal{K}}$.  
 For  any  $0\leq  i\leq  k$,  
 by  counting  the  number  of  inversions  in 
 $(  {s(0)}, \ldots,  \widehat{ {s(i)}}, \ldots,  {s(k)})$,  we  have  
 \begin{eqnarray*}
 {\rm  sgn}  {{0,\ldots, \widehat{s(i)},\ldots,k}\choose{s(0),\ldots, \widehat{s(i)},\ldots,s(k)}}
 =(-1)^{2k-i-s(i)}  {\rm  sgn}(s).  
 \end{eqnarray*}
 Thus 
  \begin{eqnarray*} 
 (\pi_{k})_\#( v_{s(0)}, \ldots,  \widehat{v_{s(i)}}, \ldots, v_{s(k)}) 
&=& {\rm  sgn}  {{0,\ldots, \widehat{s(i)},\ldots,k}\choose{s(0),\ldots, \widehat{s(i)},\ldots,s(k)}}
\{v_0, \ldots,  \widehat{v_{s(i)}}, \ldots,  v_k\}\\
&=&(-1)^{(k-i)+(k-s(i))} {\rm  sgn}(s) \{v_0, \ldots,  \widehat{v_{s(i)}}, \ldots,  v_k\}\\
&=&(-1)^{i-s(i)} {\rm  sgn}(s) \{v_0, \ldots,  \widehat{v_{s(i)}}, \ldots,  v_k\}.  
 \end{eqnarray*}
 With  the  help  of   (\ref{eq-25-10-3}),   (\ref{eq-25-10-4})  and  (\ref{eq-25-10-5}),  
   the  following  diagram  commutes  
 \begin{eqnarray*}
 \xymatrix{
( v_{s(0)}, \ldots,  v_{s(k)})  \ar[r] ^-{\vec{\partial}_k}  \ar[d]_-{(\pi_{k+1})_\#}
&  \sum_{i=0}^k (-1)^i  ( v_{s(0)}, \ldots,  \widehat{v_{s(i)}}, \ldots, v_{s(k)})
 \ar[d]^-{(\pi_{k})_\#}\\
 {\rm  sgn}(s)\{v_0,\ldots,v_{k}\} \ar[r]^-{\partial_k}
 &  \sum_{i=0}^k (-1)^i  {\rm  sgn}(s) \{v_{0}, \ldots,  \widehat{v_{i}}, \ldots, v_{k}\}.  
 }
 \end{eqnarray*}
Hence 
 the  following  diagram  commutes 
 \begin{eqnarray}\label{diag-25-10-1}
 \xymatrix{
 C_{k}(\vec {\mathcal{K}}; R)\ar[r]^-{\vec  {\partial}_k}  \ar[d]_-{(\pi_{k+1})_\#}&
   C_{k -1}(\vec {\mathcal{K}}; R)\ar[d]^-{(\pi_k)_\#}\\
  C_{k }( \mathcal{K}; R)\ar[r]^-{   {\partial}_k}  &  C_{k -1}( \mathcal{K}; R).    
 }
 \end{eqnarray}
Therefore, 
(\ref{eq-25-10-7})  
 is  a  chain  map  and  consequently   it  induces  a  homomorphism  (\ref{eq-25-10-8})  
 of  homology.  
 Since  $\mathcal{K}= \pi_\bullet(\vec{\mathcal{K}})$,  it is  direct  that  
 (\ref{eq-25-10-7})   is  surjective.  
 
  Let  $\vec{\mathcal{K}}$   and   $\vec{\mathcal{K}}'$   
 be    directed  simplicial  complexes   on  $V$  and  $V'$  respectively.  
 Let    $\mathcal{K}$  and  $\mathcal{K}'$   be  the  underlying  simplicial  complexes 
 of   $\vec{\mathcal{K}}$   and   $\vec{\mathcal{K}}'$  respectively.  
 Suppose  
 $\vec\varphi:  \vec{\mathcal{K}}\longrightarrow \vec{\mathcal{K}}'$  
 is  a  directed  simplicial  map.  
 Let  $ \varphi:  {\mathcal{K}}\longrightarrow {\mathcal{K}}'$
 be  the  simplicial  map  induced  by  $\vec\varphi$.  
 By  (\ref{eq-25-10-3}),  
 we  have  an  induced  chain  map  
 \begin{eqnarray*}
 \vec\varphi_\#:    C_\bullet(\vec {\mathcal{K}};R)\longrightarrow   C_\bullet(\vec {\mathcal{K}}';R).  
 \end{eqnarray*}
 By  (\ref{eq-25-10-4}),  we  have an  induced  chain  map 
  \begin{eqnarray*}
 \varphi_\#:    C_\bullet( \mathcal{K};R)\longrightarrow   C_\bullet(\mathcal{K}';R).  
 \end{eqnarray*}
 By  (\ref{eq-25-10-5}),  
   we  have  a  commutative  diagram  of  chain  complexes 
 \begin{eqnarray}\label{eq-25-11-8}
 \xymatrix{
 C_\bullet(\vec {\mathcal{K}};R) \ar[d]_-{\pi_\#}  \ar[r]^-{\vec\varphi_\#}
  &    C_\bullet(\vec {\mathcal{K}}';R) \ar[d]^-{\pi_\#}     \\
 C_\bullet( \mathcal{K};R) \ar[r]^-{\varphi_\#}
  &   C_\bullet( \mathcal{K}';R) 
 }
 \end{eqnarray}
 such  that  all  the  arrows are  chain  maps.  
 Therefore,   (\ref{eq-25-10-7})  is   functorial. 
 Applying  the  homology  functor  to  (\ref{eq-25-11-8}),  
 we  have  an  induced    commutative  diagram   of  homology   groups 
\begin{eqnarray}\label{eq-25-11-9}
 \xymatrix{
 H_\bullet(\vec {\mathcal{K}};R) \ar[d]_-{\pi_*}  \ar[r]^-{\vec\varphi_*}
  &  H_\bullet(\vec {\mathcal{K}}';R) \ar[d]^-{\pi_*}   \\
    H_\bullet( \mathcal{K};R) \ar[r]^-{\varphi_*}
  &   H_\bullet( \mathcal{K}';R).  
 }
 \end{eqnarray}
Therefore,   (\ref{eq-25-10-8})  is   functorial. 
 \end{proof}

\begin{corollary}\label{pr-2.1-a} 
Let  $\vec{\mathcal{K}}$  be  a  directed  simplicial  complex  on  $V$
and  let  $\mathcal{K}$  be  its  underlying  simplicial  complex.  
If  $\vec{\mathcal{K}}_k$  is  $\Sigma_k$-invariant  for  all  $k\geq  1$,  
then  
\begin{eqnarray}\label{eq-25-10-1}
C_{k-1}(\vec {\mathcal{K}};R)\cong  C_{k-1}(  {\mathcal{K}};R)^{\bigoplus  k !}
\end{eqnarray}
 and  consequently  
\begin{eqnarray}\label{eq-25-10-2}
H_{k-1}(\vec{\mathcal{K}};  R) \cong   H_{k-1}( \mathcal{K};  R)  ^{\bigoplus  k ! }  
\end{eqnarray}
for  all  $k\geq  1$.  
\end{corollary}

\begin{proof}

For  each  $k\geq  1$,  
let  $\pi_k:  \vec{\mathcal{K}}_k\longrightarrow \mathcal{K}_k$  be  the  canonical  projection. 
Since  $\vec{\mathcal{K}}_k$  is  $\Sigma_k$-invariant,  
$\pi_k$  is  a  principal  $\Sigma_k$-bundle    thus  it   is    $k!$-sheeted  covering.  
Taking  the  chain  module  $C_{k-1}(\vec {\mathcal{K}};R)$
   generated  by  the  directed  $(k-1)$-simplexes 
 in  $\vec{\mathcal{K}}_k$  and   
 the  chain  module  $C_{k-1}( \mathcal{K};R)$   generated  by  the  $(k-1)$-simplexes 
  in  $\mathcal{K}_k$  respectively, 
 we  obtain  (\ref{eq-25-10-1})   from  (\ref{eq-25-10-7}).  
 With  the help  of (\ref{diag-25-10-1}), 
 \begin{eqnarray*}
  {\rm  Ker}(\vec  {\partial}_{k-1})\cong  {\rm  Ker}(   {\partial}_{k-1})^{\bigoplus  k !}, ~~~~~~
    {\rm  Im}(\vec  {\partial}_{k })\cong  {\rm  Im}(   {\partial}_{k })^{\bigoplus  k !}.  
 \end{eqnarray*}
 Therefore,  we  obtain  (\ref{eq-25-10-2})   from  (\ref{eq-25-10-8}).  
\end{proof}

The  next  example  is  a  partial  illustration  for  \cite[Lemma~3.8  and  Theorem~3.9]{hdg}  and  Proposition~\ref{le-25-10-5}.  

 \begin{example}\label{ex-25-10-1}
 Let  $V=\{v_0,v_1,v_2\}$  such  that  $v_0\prec  v_1\prec   v_2$.  
 Let 
 \begin{eqnarray*}
 \vec{\mathcal{K}}&=& \{(v_1,v_2,v_0),  (v_1,v_0),  (v_2,v_0),  (v_1,v_2),  (v_0),  (v_1), (v_2)\},\\
 \vec{\mathcal{K}}'&=& \{(v_1,v_2,v_0), (v_2,v_0,v_1), (v_0,v_1,v_2), \\
 && (v_1,v_0),  (v_0,v_1),(v_2,v_0), (v_0,v_2), (v_1,v_2),  (v_2,v_1),  (v_0),  (v_1), (v_2)\},  \\
 \mathcal{K}&=&\{\{v_0,v_1,v_2\},  \{v_0,v_1\},  \{v_0,v_2\},  \{v_1,v_2\},  \{v_0\},  \{v_1\}, \{v_2\}\}.
 \end{eqnarray*}
 Then  $\pi_\bullet( \vec{\mathcal{K}})=\pi_\bullet( \vec{\mathcal{K}}')=\mathcal{K}$.  
  The  boundary  map  $\vec \partial_\bullet$  of  $C_\bullet( \vec{\mathcal{K}}; R)$  satisfies 
   \begin{eqnarray*}
  & \vec \partial_2(v_1,v_2,v_0) = (v_2,v_0)-(v_1,v_0)+(v_1,v_2),\\
  & \vec \partial_1(v_2,v_0)= (v_0)-(v_2),  ~~~\vec \partial_1(v_1,v_0)= (v_0)-(v_1),~~~
   \vec \partial_1(v_1,v_2)= (v_2)-(v_1); 
   \end{eqnarray*}
   the  boundary  map  $\vec \partial'_\bullet$  of  $C_\bullet( \vec{\mathcal{K}}'; R)$  satisfies 
   \begin{eqnarray*}
  & \vec \partial'_2(v_1,v_2,v_0) = (v_2,v_0)-(v_1,v_0)+(v_1,v_2),\\
    & \vec \partial'_2(v_2,v_0,v_1) = (v_0,v_1)-(v_2,v_1)+(v_2,v_0),\\
  & \vec \partial'_2(v_0,v_1,v_2) = (v_1,v_2)-(v_0,v_2)+(v_0,v_1),\\
  & \vec \partial'_1(v_2,v_0)= (v_0)-(v_2),  ~~~\vec \partial'_1(v_0,v_2)= (v_2)-(v_0),\\
  & \vec \partial'_1(v_1,v_0)= (v_0)-(v_1),~~~ \vec \partial'_1(v_0,v_1)= (v_1)-(v_0),\\
  & \vec \partial'_1(v_1,v_2)= (v_2)-(v_1),~~~\vec \partial'_1(v_2,v_1)= (v_1)-(v_2); 
   \end{eqnarray*}
 and   the  boundary  map  $  \partial_\bullet$  of  $C_\bullet( \mathcal{K}; R)$  satisfies
    \begin{eqnarray*}
    &\partial_2\{v_0,v_1,v_2\}=\{v_1,v_2\}-\{v_0,v_2\}+\{v_0,v_1\},\\
    &  \partial_1\{v_0,v_2\}= \{v_2\}-\{v_0\},  ~~~  \partial_1\{v_0,v_1\}= \{v_0\}-\{v_1\},~~~
    \partial_1\{v_1,v_2\}= \{v_2\}-\{v_1\}. 
   \end{eqnarray*}
   Note  that  on  the  chain  level,  
   \begin{eqnarray*}
  & (\pi_3)_\#(v_1,v_2,v_0)=(\pi_3)_\#(v_2,v_0,v_1)=(\pi_3)_\#(v_0,v_1,v_2)=\{v_0,v_1,v_2\},\\
  &(\pi_2)_\#(v_1,v_2)=-(\pi_2)_\#(v_2,v_1)=\{v_1,v_2\},\\
  & (\pi_2)_\#(v_0,v_2)=-(\pi_2)_\#(v_2,v_0)= \{v_0,v_2\},\\
 & (\pi_2)_\#(v_0,v_1)=-(\pi_2)_\#(v_1,v_0)= \{v_0,v_1\},\\
 &\pi_1(v_0)=\{v_0\}, ~~~  \pi_1(v_1)=\{v_1\},~~~ \pi_1(v_2)=\{v_2\}.  
   \end{eqnarray*}
 \end{example}
 
 \begin{example}
 \label{ex-25-10-19-1}
 A  {\it  digraph}  is  a  $1$-dimensional  directed  simplicial  complex
  $\vec  G=  V\cup  \vec  E$,  where  
  $V$  is  the  set    of    vertices  and  $\vec   E$  is  the   set  of   directed  edges.  
  A  {\it  graph}  is  a  $1$-dimensional   simplicial  complex
  $G=  V\cup  E$,  where  
    $E$  is  the  set  of   edges.    
  Let  $V$  and  $V'$  be  two  disjoint  sets  of  vertices.  
  Let  $\vec  G$  and  $\vec  G'$  be  digraphs  on  $V$  and  $V'$  respectively. 
  Let  $G$  and  $G'$  be  the  underlying  graphs  of  $\vec  G$  and  $\vec  G'$  respectively.   
  We define  the  {\it  reduced  join}  of  $\vec  G$  and  $\vec  G'$  
  to  be  a  digraph  on  $V\sqcup  V'$  given  by 
  \begin{eqnarray*}
 \vec { G}\tilde { *}  \vec {G}'= {\rm  sk}^1( \vec { G}    *   \vec {G}')
 =( V\sqcup  V') \cup ( \vec  E\sqcup  \vec  E'  \sqcup  \{(v,v')\mid  v\in  V,  v'\in  V'\})
  \end{eqnarray*}
  and  define  the  {\it  reduced  join}  of  $  G$  and  $  G'$  
  to  be  a   graph  on  $V\sqcup  V'$  given  by   (cf.  Example~\ref{ex-25-12-16-1})
  \begin{eqnarray*}
   { G}\tilde { *}   {G}'= {\rm  sk}^1(  { G}    *     {G}')
 =( V\sqcup  V') \cup ( \vec  E\sqcup  \vec  E'  \sqcup  \{\{v,v'\}\mid  v\in  V,  v'\in  V'\}).  
  \end{eqnarray*}
With  the  help  of  (\ref{eq-25-10-11-5}),     
 ${ G}\tilde { *}   {G}'$  is   the  underlying  graph  of  $ \vec { G}\tilde { *}  \vec {G}'$.  
 With  the  help  of  (\ref{eq-25-10-19-11}),
 the  reduced  join  of  (di)graphs  satisfies the  commutativity  law.   
  With  the  help  of  (\ref{eq-25-10-19-12}),
 the  reduced  join  of  (di)graphs  satisfies the  associativity  law.   
 \end{example}
 
 \subsection{The  Mayer-Vietoris  sequences }\label{ss-3.1}

 \begin{proposition}\label{pr-10-11-mv-1}
 For  any  directed  simplicial  complexes $\vec{\mathcal{K}}$  and  $\vec{\mathcal{K}}'$
 on  $V$  with  their  underlying   simplicial  complexes 
 $ {\mathcal{K}}$  and  $ {\mathcal{K}}'$   respectively,  
 we  have  a  commutative   diagram  of  homology  groups  
 \begin{eqnarray}\label{diag-10-12-2}
 \xymatrix{
\cdots\ar[r]  
&H_n (\vec{\mathcal{K}}\cap  \vec{\mathcal{K}}';R) \ar[r]\ar[d]_-{(\pi_{n+1})_*}
&H_n(\vec{\mathcal{K}};R)\oplus  H_n(  \vec{\mathcal{K}}';R)\ar[r]\ar[d]_-{(\pi_{n+1})_*}
 &\\
\cdots\ar[r]  
 &H_n ( {\mathcal{K}}\cap  {\mathcal{K}}';R) \ar[r] 
&H_n( {\mathcal{K}};R)\oplus  H_n(  {\mathcal{K}}';R)\ar[r] 
 &\\
\ar[r]&H_n(\vec{\mathcal{K}}\cup  \vec{\mathcal{K}}';R)\ar[r]\ar[d]_-{(\pi_{n+1})_*}
&H_{n-1}(\vec{\mathcal{K}}\cap  \vec{\mathcal{K}}';R)\ar[r]\ar[d]_-{(\pi_{n})_*}
&\cdots\\
\ar[r]&H_n( {\mathcal{K}}\cup   {\mathcal{K}}';R)\ar[r] 
&H_{n-1}( {\mathcal{K}}\cap   {\mathcal{K}}';R)\ar[r] 
&\cdots
} 
 \end{eqnarray}
 such that  the  two  rows  are  long  exact sequences.  
 Moreover,  the  diagram  (\ref{diag-10-12-2})  is  natural  with  respect to  
 directed  simplicial  maps  between  directed  simplicial  complexes   
 and  their  induced  simplicial  maps   between  the  underlying  simplicial  complexes.  
 \end{proposition}
 
 \begin{proof}
 Let $\vec{\mathcal{K}}$  and  $\vec{\mathcal{K}}'$  be  directed  simplicial  complexes  on  $V$.  
 Let  $\mathcal{K}=\pi_\bullet(\vec{\mathcal{K}})$  and 
 $\mathcal{K}'=\pi_\bullet(\vec{\mathcal{K}}')$  be  the  underlying  simplicial  complexes   
 of  $\vec{\mathcal{K}}$  and  $\vec{\mathcal{K}}'$  respectively.  
 With  the  help  of Proposition~\ref{le-25-10-5},  
 we  have  a  commutative  diagram  of  chain  complexes  
  \footnote[4]{With  the  help  of     Lemma~\ref{pr-25-10-11-1}~(1),  
 the  last  two  vertical  maps  in    (\ref{diag-10-12-1})  are  surjective  and  if  in  addition that 
 both  $\vec{\mathcal{K}}$  and  $\vec{\mathcal{K}}'$ 
 are  $\Sigma_\bullet$-invariant,  
 then  the  first  vertical  map  in  (\ref{diag-10-12-1})  is  surjective  as  well.}  
 \begin{eqnarray}\label{diag-10-12-1}
 \xymatrix{
0\ar[r] 
&  C_\bullet(\vec{\mathcal{K}}\cap  \vec{\mathcal{K}}';R)\ar[r] \ar[d]_-{(\pi_{\bullet+1})_\#}
& C_\bullet(\vec{\mathcal{K}};R)\oplus  C_\bullet(  \vec{\mathcal{K}}';R)\ar[r] 
\ar[d]_-{(\pi_{\bullet+1})_\#}
& C_\bullet(\vec{\mathcal{K}}\cup  \vec{\mathcal{K}}';R)\ar[r]\ar[d]_-{(\pi_{\bullet+1})_\#}
& 0\\
0\ar[r] 
&  C_\bullet( {\mathcal{K}}\cap   {\mathcal{K}}';R)\ar[r] 
& C_\bullet( {\mathcal{K}};R)\oplus  C_\bullet(  {\mathcal{K}}';R)\ar[r] 
& C_\bullet( {\mathcal{K}}\cup   {\mathcal{K}}';R)\ar[r] 
& 0
 }
 \end{eqnarray}
 such  that  all  the  arrows are  chain  maps  and  the  two  rows  are  short  exact  sequences.  
 Applying   the  homology  functor  to
    (\ref{diag-10-12-1}),   
 we  obtain  the  commutative  diagram (\ref{diag-10-12-2})  of homology  groups.  
 The  short  exact  sequences  of  chain  complexes  in  (\ref{diag-10-12-1})
 induce  the  long  exact  sequences  of  homology  groups  in   (\ref{diag-10-12-2}).  
 \end{proof}
 
For  simplicity,  we  denote  the  long  exact  sequence  in  the  first  row  of  (\ref{diag-10-12-2})  by  
${\bf {\rm   MV}}(\vec{\mathcal{K}},  \vec{\mathcal{K}}')$  and  
denote  the  long  exact  sequence  in the  second  row  of    (\ref{diag-10-12-2})  
by  ${\bf {\rm   MV}}( {\mathcal{K}},   {\mathcal{K}}')$.  
 We  simply  denote    (\ref{diag-10-12-2})  by    
  $\pi_*:  {\bf {\rm   MV}}(\vec{\mathcal{K}},  \vec{\mathcal{K}}')\longrightarrow 
 {\bf {\rm   MV}}( {\mathcal{K}},   {\mathcal{K}}')$,
 which  is  a  morphism  of  long  exact  sequences  (cf.  \cite[Section~7]{comalg}).  
 
  \subsection{The  K\"{u}nneth-type  formulae }\label{ss-3.2}

\begin{proposition}\label{pr-10-12-kf-1}
For  any  directed  simplicial  complex  $\vec{\mathcal{K}}$  on  $V$   and
any  directed  simplicial  complex   $\vec{\mathcal{K}}'$
 on  $V'$
 with  their  underlying   simplicial  complexes 
 $ {\mathcal{K}}$  and  $ {\mathcal{K}}'$   respectively,  
 we  have  a  commutative   diagram    
 \begin{eqnarray}\label{diag-10-12-7}
\xymatrix{
0\ar[r]& \bigoplus_{p+q+1=n} H_{p+1}(\vec{\mathcal{K}};R)\otimes H_{q+1}(\vec{\mathcal{K}}';R)
\ar[r]\ar[d] _-{ \bigoplus_{p+q+1=n}(\pi_{p+2})_*\otimes (\pi_{q+2})_*}
&H_{n+1}(\vec{\mathcal{K}} *\vec{\mathcal{K}}';R) \ar[r]\ar[d]_-{(\pi_{n+2})_*} & \\
0\ar[r]& \bigoplus_{p+q+1=n} H_{p+1}({\mathcal{K}};R)\otimes H_{q+1}({\mathcal{K}}';R)
\ar[r]
&H_{n+1}(\mathcal{K} *\mathcal{K}';R)\ar[r] &
}\\
~~~\nonumber\\
\xymatrix{
\ar[r] &\bigoplus_{p+q+1=n} {\rm  Tor}_R(H_{p+1}(\vec{\mathcal{K}};R), H_{q}(\vec{\mathcal{K}}';R))\ar[r]\ar[d] &0 \\
\ar[r]  &\bigoplus_{p+q+1=n} {\rm  Tor}_R(H_{p+1}(\mathcal{K};R), H_{q}(\mathcal{K}';R))\ar[r]  &0
}
\nonumber
 \end{eqnarray}
 such that  the  two  rows  are  short  exact sequences.  
 Moreover,  the  diagram  (\ref{diag-10-12-7})  is  natural  with  respect to  
 directed  simplicial  maps  between  directed  simplicial  complexes   
 and  their  induced  simplicial  maps   between  the  underlying  simplicial  complexes.  
\end{proposition}

\begin{proof}
Let  $V$  and   $V'$  be  disjoint  sets. 
 Let $\vec{\mathcal{K}}$  and  $\vec{\mathcal{K}}'$  
 be  directed  simplicial  complexes  on  $V$  and   $V'$  respectively.  
  Let  ${\mathcal{K}}$  and  $ {\mathcal{K}}'$ 
  be  the  underlying  simplicial  complexes  
  of  $\vec{\mathcal{K}}$  and  $\vec{\mathcal{K}}'$ 
  respectively.  With  the  help  of    Lemma~\ref{pr-25-10-11-1}~(2)
  and    Proposition~\ref{le-25-10-5},   
  we  have a   commutative  diagram  of   $R$-modules   
  \begin{eqnarray}\label{diag-10-12-3}
  \xymatrix{
     C_{k}(\vec{\mathcal{K}};R)\otimes C_l(\vec{\mathcal{K}}';R) \ar[r]^-{\cong } 
  \ar[d]_-{(\pi_{k+1})_\#\otimes (\pi_{l+1})_\#}
  &C_{k+l+1}(\vec{\mathcal{K}}*\vec{\mathcal{K}}';R)\ar[d]^-{(\pi_{k+l+2})_\#}
  \\
  C_{k}( {\mathcal{K}};R)\otimes C_l( {\mathcal{K}}';R)\ar[r]^-{\cong }
     &C_{k+l+1}( {\mathcal{K}}* {\mathcal{K}}';R)
  }
  \end{eqnarray}
 for  any  $k,l\geq  -1$  by  setting  $C_{-1}(\vec{\mathcal{K}};R)=C_{-1}(\vec{\mathcal{K}}';R)=R$. 
 It  is  direct  to  verify  that  all  the  maps  in  (\ref{diag-10-12-3}) 
 commute with the boundary  maps.  
 Hence  we  have a  commutative  diagram  of  chain  complexes 
\begin{eqnarray}\label{diag-10-12-5}
\xymatrix{
C_{\bullet}(\vec{\mathcal{K}};R)\otimes  C_\bullet(\vec{\mathcal{K}}';R)
\ar[d]_-{(\pi_{\bullet+1})_\#\otimes (\pi_{\bullet+1})_\#} \ar[r]^-{\cong }
& C_{\bullet+1}(\vec{\mathcal{K}}*\vec{\mathcal{K}}';R)\ar[d]^-{(\pi_{\bullet+2})_\#}
\\
 C_{\bullet}( {\mathcal{K}};R)\otimes  C_\bullet( {\mathcal{K}}';R) 
  \ar[r]^-{\cong }
&
C_{\bullet+1}( {\mathcal{K}}* {\mathcal{K}}';R).   
}
\end{eqnarray}
Apply \cite[Theorem~3B.5]{hatcher} to  the  commutative diagram  (\ref{diag-10-12-5}).  
The two  rows  of  (\ref{diag-10-12-5})  imply  the  two short  exact  sequences  of  
(\ref{diag-10-12-7}).  
By  the  naturality  of  \cite[Theorem~3B.5]{hatcher},  the  vertical  chain  maps  of   
 (\ref{diag-10-12-5})  imply  the  vertical  homomorphisms  of  (\ref{diag-10-12-7}).  
\end{proof}

For  simplicity,  we  denote  the  short  exact  sequence  in  the  first  row  of  (\ref{diag-10-12-7})  by  
${\bf {\rm   KU}}(\vec{\mathcal{K}},  \vec{\mathcal{K}}')$  and  
denote the  short  exact  sequence  in  the  second  row    of    (\ref{diag-10-12-7})  
by  ${\bf {\rm   KU}}( {\mathcal{K}},   {\mathcal{K}}')$.  
 We  simply  denote    (\ref{diag-10-12-7})  by    
  $\pi_*:  {\bf {\rm   KU}}(\vec{\mathcal{K}},  \vec{\mathcal{K}}')\longrightarrow 
 {\bf {\rm   KU}}( {\mathcal{K}},   {\mathcal{K}}')$, which  is  a  morphism  of  short  exact  sequences.

\section{Hyper(di)graphs,    their    embedded  homology  and
associated   (directed)  simplicial  complexes  
 }\label{s5}

 In  this  section,  we  give  some  preliminaries  on  the  embedded  homology 
 of  hypergraphs  and  hyperdigraphs  as  well  as  the  homology  of  the  
 associated  (directed)  simplicial  complexes  and  the  lower-associated  (directed)
 simplicial  complexes.  
 We  discuss  about  a  canonical  homomorphism  from  the  embedded  homology
 of  a  hyperdigraph  to  the embedded  homology  of  the  underlying  hypergraph  in  
 Theorem~\ref{th-251127-hh1}. 
 We  give  some  commutative  diagrams  of  Mayer-Vietoris  sequences 
  in  Theorem~\ref{th-251118-mv1}
 and  some  commutative  diagrams  of  K\"unneth-type  short  exact sequences  
 in  Theorem~\ref{th-251127-hku5}.

   Let  $\vec{\mathcal{H}}$  be  a  hyperdigraph  on  $V$.  
   The  {\it  associated  directed  simplicial  complex}  
   $\Delta\vec{\mathcal{H}}$  is  the  smallest  directed  simplicial  complex
   containing  $\vec{\mathcal{H}}$  
   and  the  {\it  lower-associated  directed  simplicial  complex}  
   $\delta\vec{\mathcal{H}}$  is  the  largest  directed  simplicial  complex
   contained  in  $\vec{\mathcal{H}}$  (cf.  \cite{jgp}).  
   Let  $\mathcal{H}$  be  the  underlying  hypergraph  of  $\mathcal{H}$.  
   The  {\it  associated     simplicial  complex}  
   $\Delta{\mathcal{H}}$  is  the  smallest    simplicial  complex
   containing  $ {\mathcal{H}}$  
   and  the  {\it  lower-associated     simplicial  complex}  
   $\delta {\mathcal{H}}$  is  the  largest    simplicial  complex
   contained  in  $ {\mathcal{H}}$  (cf.  \cite{jktr2,jktr2023,jktr2022-2}).  
   The  canonical  projection  $\pi:  \vec{\mathcal{H}}\longrightarrow   {\mathcal{H}}$
   induces  projections     
   $\Delta\pi:  \Delta\vec{\mathcal{H}}\longrightarrow  \Delta\mathcal{H}$  and 
    $\delta\pi:  \delta\vec{\mathcal{H}}\longrightarrow  \delta\mathcal{H}$
    sending  the  directed  simplicial  complexes  to  the underlying  simplicial  complexes
        such  that  the  following    diagram    commutes 
   \begin{eqnarray}\label{diag-1116-1}
   \xymatrix{
   \delta\vec{\mathcal{H}}\ar[r] \ar[d]_-{\delta\pi}
   & \vec{\mathcal{H}}\ar[r]\ar[d]^-{\pi}
   & \Delta\vec{\mathcal{H}}\ar[d]^-{\Delta\pi}\\
   \delta {\mathcal{H}}\ar[r] 
   &  {\mathcal{H}}\ar[r]
   & \Delta {\mathcal{H}},  
   }
   \end{eqnarray}
where  all  the  horizontal  maps  are  canonical  inclusions.

For  each  $k\geq  1$,  let 
$R(\vec{\mathcal{H}}_{k})$  be  the free  $R$-module  spanned  by  
$\vec{\mathcal{H}}_{k}$.  
With the  help  of  \cite[Section~2]{h1}, 
the  infimum   chain  complex 
\begin{eqnarray}\label{eq-251117-a1}
{\rm  Inf}_{k-1}(\vec{\mathcal{H}};R)= R(\vec{\mathcal{H}}_{k})
 \cap( \vec{\partial}_{k-1})^{-1} R(\vec{\mathcal{H}}_{k-1})
\end{eqnarray}
is  the  largest  sub-chain  complex  of  $C_\bullet(\Delta\vec{\mathcal{H}};R)$  
contained  in   $R(\vec{\mathcal{H}}_\bullet)$
and  the  supremum   chain  complex 
\begin{eqnarray}\label{eq-251117-a2}
{\rm  Sup}_{k-1}(\vec{\mathcal{H}};R)= R(\vec{\mathcal{H}}_{k})
 +  \vec{\partial}_{k }   R(\vec{\mathcal{H}}_{k+1})
\end{eqnarray}
is  the  smallest  sub-chain  complex  of  $C_\bullet(\Delta\vec{\mathcal{H}};R)$  
containing    $R(\vec{\mathcal{H}}_\bullet)$.  
By  \cite[Section~3]{h1},  the  infimum   chain  complex 
\begin{eqnarray}\label{eq-251117-a3}
{\rm  Inf}_{k-1}( {\mathcal{H}};R)= R( {\mathcal{H}}_{k})
 \cap(  {\partial}_{k-1})^{-1} R( {\mathcal{H}}_{k-1})
\end{eqnarray}
is  the  largest  sub-chain  complex  of  $C_\bullet(\Delta {\mathcal{H}};R)$  
contained  in   $R( {\mathcal{H}}_\bullet)$
and  the  supremum   chain  complex 
\begin{eqnarray}\label{eq-251117-a4}
{\rm  Sup}_{k-1}( {\mathcal{H}};R)= R( {\mathcal{H}}_{k})
 +   {\partial}_{k }   R( {\mathcal{H}}_{k+1})
\end{eqnarray}
is  the  smallest  sub-chain  complex  of  $C_\bullet(\Delta {\mathcal{H}};R)$  
containing    $R( {\mathcal{H}}_\bullet)$.  
Let  $\vec{\mathcal{K}}$  be  $\Delta{\vec{\mathcal{H}}}$
and  consequently   $\mathcal{K}$  be   
$\Delta\mathcal{H}$   (cf.   (\ref{diag-1116-1}))  in  (\ref{diag-25-10-1}). 
Then   
\begin{eqnarray}\label{eq-251117-a7}
\pi_\#\circ  \vec{\partial}= \partial\circ\pi_\#.
\end{eqnarray}
It  follows  from      (\ref{eq-251117-a1}),  (\ref{eq-251117-a3})  and  
(\ref{eq-251117-a7})  that 
there  is  an  induced  projection  
\begin{eqnarray}\label{eq-251117-a5}
{\rm  Inf}(\pi): {\rm  Inf}_\bullet(\vec{\mathcal{H}};R)\longrightarrow  
{\rm  Inf}_\bullet({\mathcal{H}};R).  
\end{eqnarray}
Similarly,  it  follows  from       (\ref{eq-251117-a2}),   (\ref{eq-251117-a4})  and    
(\ref{eq-251117-a7})  that 
there  is  an  induced  projection  
\begin{eqnarray}\label{eq-251117-a6}
{\rm  Sup}(\pi): {\rm  Sup}_\bullet(\vec{\mathcal{H}};R)\longrightarrow  
{\rm  Sup}_\bullet({\mathcal{H}};R). 
\end{eqnarray}
 Consequently,  with  the  help  of  (\ref{eq-251117-a5})  
 and  (\ref{eq-251117-a6}),   
 the  diagram  (\ref{diag-1116-1})  induces  a  commutative  diagram  of   chain  complexes 
   \begin{eqnarray}\label{diag-1116-2}
   \xymatrix{
   C_\bullet(\delta\vec{\mathcal{H}};R)\ar[r] \ar[d]_-{(\delta\pi)_\#}
   &  {\rm  Inf}_\bullet(\vec{\mathcal{H}};R)\ar[r]^-{\vec\iota}\ar[d]^-{{\rm  Inf}(\pi)}
      &  {\rm  Sup}_\bullet(\vec{\mathcal{H}};R)\ar[r]\ar[d]^-{{\rm  Sup}(\pi)}
   &  C_\bullet(\Delta\vec{\mathcal{H}};R)\ar[d]^-{(\Delta\pi)_\#}\\
   C_\bullet( \delta{\mathcal{H}};R)\ar[r] 
   &  {\rm  Inf}_\bullet( {\mathcal{H}};R)\ar[r]^-{\iota} 
      &  {\rm  Sup}_\bullet( {\mathcal{H}};R)\ar[r] 
   &  C_\bullet(\Delta {\mathcal{H}};R)
   }
   \end{eqnarray}
   where  all the  maps  are  chain  maps, all  the  horizontal  maps 
    are  canonical  inclusions 
    and  all  the  vertical  maps  are  projections.   
    It  follows  from  
   \cite[Section~3]{h1}  and  \cite[Theorem~3.9]{hdg} 
   that  both  $\vec\iota$  and  $\iota$  in     (\ref{diag-1116-2})  
    are  quasi-isomorphisms.  
    The  homology  of  ${\rm  Inf}_\bullet(\vec{\mathcal{H}};R)$
    as  well  as   ${\rm  Sup}_\bullet(\vec{\mathcal{H}};R)$  is  denoted  as  
    $H_\bullet(\vec{\mathcal{H}};R)$  
    and  is 
    called  the  embedded  homology  of  the  hyperdigraph  $\vec{\mathcal{H}}$.     
    The   homology  of  ${\rm  Inf}_\bullet( {\mathcal{H}};R)$
    as  well  as   ${\rm  Sup}_\bullet( {\mathcal{H}};R)$  is  
    denoted  as  
    $H_\bullet({\mathcal{H}};R)$  
    and  is
    called  the  embedded  homology  of  the  hypergraph  ${\mathcal{H}}$.

    Let  $\vec{\mathcal{H}}$  and  $\vec{\mathcal{H}}'$  be  hyperdigraphs 
 on  $V$  and  $V'$  respectively.  
 Let  $\vec{\varphi}:  \vec{\mathcal{H}}\longrightarrow \vec{\mathcal{H}}'$  
 be  a  morphism  of  hyperdigraphs.  
 By  (\ref{eq-1103-1}),  
 we  have  an  induced  morphism  $\varphi: \mathcal{H}\longrightarrow \mathcal{H}'$
 between  the  underlying  hypergraphs. 
 Moreover,  we  have  induced  directed  simplicial  maps  
 $\Delta\vec{\varphi}:  \Delta\vec{\mathcal{H}}\longrightarrow   \Delta\vec{\mathcal{H}}'$  
 and  $\delta\vec{\varphi}:   \delta\vec{\mathcal{H}}\longrightarrow   \delta\vec{\mathcal{H}}'$  such  that  the  following    diagram    commutes 
   \begin{eqnarray}\label{diag-1116-12}
   \xymatrix{
   \delta\vec{\mathcal{H}}\ar[r] \ar[d]_-{\delta\vec{\varphi}}
   & \vec{\mathcal{H}}\ar[r]\ar[d]^-{\vec\varphi}
   & \Delta\vec{\mathcal{H}}\ar[d]^-{\Delta\vec{\varphi}}\\
   \delta \vec{\mathcal{H}}'\ar[r] 
   &  \vec{\mathcal{H}}'\ar[r]
   & \Delta \vec{\mathcal{H}}'.  
   }
   \end{eqnarray}
  Furthermore,  we  have   induced   simplicial  maps  
  $\Delta {\varphi}:  \Delta{\mathcal{H}}\longrightarrow    \Delta {\mathcal{H}}'$  
 and  $\delta {\varphi}:  \delta{\mathcal{H}}\longrightarrow   \delta {\mathcal{H}}'$  
 such  that  the  following    diagram    commutes 
   \begin{eqnarray}\label{diag-1116-13}
   \xymatrix{
   \delta {\mathcal{H}}\ar[r] \ar[d]_-{\delta {\varphi}}
   &  {\mathcal{H}}\ar[r]\ar[d]^-{ \varphi}
   & \Delta {\mathcal{H}}\ar[d]^-{\Delta {\varphi}}\\
   \delta {\mathcal{H}}'\ar[r] 
   &   {\mathcal{H}}'\ar[r]
   & \Delta  {\mathcal{H}}'.  
   }
   \end{eqnarray} 
The    commutative  diagram  (\ref{diag-1116-1})  
 is  functorial  with  respect  to  morphisms  of  hyper(di)graphs,  i.e.  
 the  vertical  maps  $\delta\pi$,  
 $\pi$  and  $\Delta\pi$  in (\ref{diag-1116-1})  
 as  well  as     $\delta\pi'$,  
 $\pi'$  and  $\Delta\pi'$  obtained  by  substituting  $\vec{\mathcal{H}}$  with  $\vec{\mathcal{H}}'$  
 in   (\ref{diag-1116-1})    commute  with  the  maps  in  the  
 diagrams  (\ref{diag-1116-12})  and  (\ref{diag-1116-13}). 
 Consequently,   by  an  analogous  argument  of  \cite[Proposition~3.7]{h1}
 and  \cite[Lemma~3.8]{hdg},   
 the  commutative  diagram (\ref{diag-1116-2})  
 is  functorial  with  respect  to  morphisms  of  hyper(di)graphs.  
    The  next theorem  is  a  straightforward  improvement  of  \cite[Proposition~3.7]{h1}  and 
    \cite[Proposition~3.7  and  Theorem~3.9]{hdg}. 
    
    \begin{theorem}\label{th-251127-hh1}
      We  have  a  commutative  diagram  of  homology  groups
    \begin{eqnarray}\label{diag-1116-3}
   \xymatrix{
   H_\bullet(\delta\vec{\mathcal{H}};R)\ar[r] \ar[d]_-{(\delta\pi)_*}
   & H_\bullet(\vec{\mathcal{H}};R) \ar[d]^-{\pi_*}\ar[r]
   &  H_\bullet(\Delta\vec{\mathcal{H}};R)\ar[d]^-{(\Delta\pi)_*}\\
   H_\bullet( \delta{\mathcal{H}};R)\ar[r] 
    & H_\bullet( {\mathcal{H}};R) \ar[r] 
   &  H_\bullet(\Delta {\mathcal{H}};R)
   }
   \end{eqnarray}
   where  $H_\bullet(\vec{\mathcal{H}};R)$  is  the  embedded  homology  of  
   the  hyperdigraph $\vec{\mathcal{H}}$  and  
   $H_\bullet( {\mathcal{H}};R)$  is  the  embedded  homology  of  
   the  hypergraph ${\mathcal{H}}$.  
   Moreover,  
   the  diagram  (\ref{diag-1116-3})  
    is  functorial  with  respect  to  morphisms  of  hyperdigraphs
    and  the   induced  morphisms  between  the  underlying hypergraphs.  
    \end{theorem}
    
   \begin{proof}
    Apply  the  homology  functor  to (\ref{diag-1116-2}).  
    We  obtain the  commutative  diagram  (\ref{diag-1116-3}).   
    Since  (\ref{diag-1116-1})  and   (\ref{diag-1116-2})    
 are   functorial  with  respect  to  morphisms  of  hyperdigraphs
 and  the   induced  morphisms  between  the  underlying hypergraphs,   
 (\ref{diag-1116-3})  
    is  functorial.  
   \end{proof}

    \subsection{The  Mayer-Vietoris  sequences }\label{ss-5.1}

Let  $\vec{\mathcal{H}}$  and  $\vec{\mathcal{H}}'$  be  hyperdigraphs  on  $V$. 
 Let   $ {\mathcal{H}}$  and  $ {\mathcal{H}}'$  be their  underlying   hypergraphs.  
 Suppose 
 \begin{enumerate}[(I)]
 \item
  both $\vec{\mathcal{H}}$  and  $\vec{\mathcal{H}}'$  are  $\Sigma_\bullet$-invariant. 
  \end{enumerate}
 By an  analogous    argument  of  (\ref{eq-25-10-1}),  
 we   have  
 \begin{eqnarray*} 
{\rm  Inf}_{k-1}(\vec {\mathcal{H}};R)&\cong&  {\rm  Inf}_{k-1}(  {\mathcal{H}};R)^{\bigoplus  k !},\\
{\rm  Sup}_{k-1}(\vec {\mathcal{H}};R)&\cong&  {\rm  Sup}_{k-1}(  {\mathcal{H}};R)^{\bigoplus  k !}
\end{eqnarray*}
 and  consequently 
  \begin{eqnarray} \label{eq-251120-6}
   H_{k-1}(\vec {\mathcal{H}};R) \cong    H_{k-1}(  {\mathcal{H}};R)^{\bigoplus  k !}.  
\end{eqnarray}
Suppose 
\begin{enumerate}[(II)]
\item
 for  any  $\sigma\in \mathcal{H}$  and  any  $\sigma'\in\mathcal{H}'$,  
  either  $\sigma\cap\sigma'$  is  the  empty-set  or 
  $\sigma\cap\sigma'\in \mathcal{H}\cap\mathcal{H}'$. 
  \end{enumerate}
By    \cite[Theorem~3.10]{h1},  we  have a  long   exact  sequence 
\begin{eqnarray}\label{eq-251120-5}
\xymatrix{
  \cdots\ar[r]
  & H_n(\mathcal{H}\cap\mathcal{H}';R)\ar[r]
  &   H_n(\mathcal{H};R)\oplus  H_n(\mathcal{H}';R)\ar[r]
  &\\
   \ar[r]& H_n(\mathcal{H}\cup\mathcal{H}';R)\ar[r]
  &  H_{n-1}(\mathcal{H}\cap\mathcal{H}';R)\ar[r]
  &\cdots
}
\end{eqnarray}  
which  will  be  denoted  by  ${\bf {\rm  MV}}(\mathcal{H}, \mathcal{H}')$.  
The  next  theorem  is  a  straightforward 
  improvement  of  \cite[Section~5.3]{stab-hg}   and  
   \cite[Proposition~3.7  and  Theorem~3.9]{hdg}.   

\begin{theorem}\label{th-251118-mv1}
Let  $\vec{\mathcal{H}}$  and  $\vec{\mathcal{H}}'$  be  
hyperdigraphs.  
\begin{enumerate}[(1)]
\item
 We  have  a  commutative  diagram 
 \begin{eqnarray}\label{diag-251118-01}
 \xymatrix{
{\bf {\rm  MV}}(\delta \vec{ \mathcal{H}},\delta \vec {\mathcal{H}}') \ar[r] \ar[d]_-{(\delta\pi)_*} 
& {\bf {\rm  MV}}(\Delta \vec{ \mathcal{H}},\Delta \vec {\mathcal{H}}')  \ar[d]^-{(\Delta\pi)_*} \\
{\bf {\rm  MV}}(\delta { \mathcal{H}},\delta   {\mathcal{H}}')  \ar[r]  
 &  {\bf {\rm  MV}}(\Delta  { \mathcal{H}},\Delta   {\mathcal{H}}')   
 }
 \end{eqnarray}
 where  the  arrows are  morphisms  of  long  exact  sequences.  
 Moreover,   (\ref{th-251118-mv1})  is  functorial  with respect to  morphisms  of  hyperdigraphs
 and  the induced  morphisms  of  the  underlying  hypergraphs.
 \item
Suppose (I)  and  (II).  
  Then 
 we  have  a  commutative  diagram 
  \begin{eqnarray}\label{diag-251119-02}
 \xymatrix{
{\bf {\rm  MV}}(\delta \vec{ \mathcal{H}},\delta \vec {\mathcal{H}}') \ar[r] \ar[d]_-{(\delta\pi)_*}  
&{\bf {\rm  MV}}(  \vec{ \mathcal{H}},  \vec {\mathcal{H}}') \ar[r] \ar[d]_-{ \pi _*} 
& {\bf {\rm  MV}}(\Delta \vec{ \mathcal{H}},\Delta \vec {\mathcal{H}}')  \ar[d]^-{(\Delta\pi)_*} \\
{\bf {\rm  MV}}(\delta { \mathcal{H}},\delta   {\mathcal{H}}')  \ar[r]  
&{\bf {\rm  MV}}(    \mathcal{H},  \mathcal{H}') \ar[r] 
 &  {\bf {\rm  MV}}(\Delta  { \mathcal{H}},\Delta   {\mathcal{H}}')   
 }
 \end{eqnarray}
  where  the  arrows are  morphisms  of  long  exact  sequences.  
Moreover,   (\ref{diag-251119-02})  is  functorial  with  respect  to  morphisms  of  $\Sigma_\bullet$-invariant  
  hyperdigraphs and  the induced  morphisms  of  the  underlying  hypergraphs.  
  \end{enumerate}
\end{theorem}

\begin{proof}
Applying  Proposition~\ref{pr-10-11-mv-1}  to  the  first  and  the  third  columns  of  
(\ref{diag-1116-1}),  
we  obtain  (\ref{diag-251118-01}).  
 Suppose  (I)  and  (II).  
  By   (\ref{eq-251120-6})  and  (\ref{eq-251120-5}),  
  we  obtain  $\pi_*$  in  the  middle  column  of  (\ref{diag-251119-02}).  
  Applying   the  embedded  homology  functor to (\ref{diag-1116-1})  and  taking  the  long exact sequence (\ref{eq-251120-5}),  we  obtain    (\ref{diag-251119-02}).  
\end{proof}

  \subsection{The  K\"{u}nneth-type  formulae }

Let  $V$  and   $V'$  be  disjoint  sets. 
Let  $\vec{\mathcal{H}}$  and  $\vec{\mathcal{H}}'$  be  hyperdigraphs
  on  $V$  and  $V'$  respectively
  with  their  underlying  hypergraphs $ {\mathcal{H}}$  and  $ {\mathcal{H}}'$. 
  Then   with  the  help  of  \cite[Section~3]{kun}, 
  \begin{eqnarray}\label{eq-251120-k1}
 {\rm  Inf}_{\bullet+1}(\vec{\mathcal{H}}*\vec{\mathcal{H}}';R)&\cong & 
  {\rm  Inf}_\bullet(\vec{\mathcal{H}} ;R)\otimes  {\rm  Inf}_\bullet( \vec{\mathcal{H}}';R),\\
  {\rm  Inf}_{\bullet+1}( {\mathcal{H}}* {\mathcal{H}}';R)&\cong & 
  {\rm  Inf}_\bullet( {\mathcal{H}} ;R)\otimes  {\rm  Inf}_\bullet( {\mathcal{H}}';R). 
  \label{eq-251120-k2}
  \end{eqnarray}
  Consequently,  
 by  (\ref{eq-251120-k1})  we  have  a  short  exact  sequence 
  \begin{eqnarray}\label{eq-251121-a1}
 \xymatrix{
0\ar[r]
& \bigoplus_{p+q+1=n} H_{p+1}(\vec{\mathcal{H}};R)\otimes H_{q+1}(\vec{\mathcal{H}}';R)
\ar[r] 
&H_{n+1}(\vec{\mathcal{H}} *\vec{\mathcal{H}}';R) \ar[r] 
& \\
 \ar[r]
 &\bigoplus_{p+q+1=n} {\rm  Tor}_R(H_{p+1}(\vec{\mathcal{H}};R), H_{q}(\vec{\mathcal{H}}';R))\ar[r] &0,  
 &
}
  \end{eqnarray}
  denoted  by  ${\bf{\rm  KU}}(\vec{\mathcal{H}},  \vec{\mathcal{H}}')$;  
  and  by  (\ref{eq-251120-k2})  we  have  a  short  exact  sequence 
  \begin{eqnarray}\label{eq-251121-a2}
 \xymatrix{
0\ar[r]
& \bigoplus_{p+q+1=n} H_{p+1}( {\mathcal{H}};R)\otimes H_{q+1}( {\mathcal{H}}';R)
\ar[r] 
&H_{n+1}( {\mathcal{H}} * {\mathcal{H}}';R) \ar[r] 
& \\
 \ar[r]
 &\bigoplus_{p+q+1=n} {\rm  Tor}_R(H_{p+1}( {\mathcal{H}};R), H_{q}( {\mathcal{H}}';R))\ar[r] &0,  
 &
}
  \end{eqnarray}
  denoted  by  ${\bf{\rm  KU}}( {\mathcal{H}},   {\mathcal{H}}')$.   
  The  projections  $ \pi:  \vec{\mathcal{H}}\longrightarrow  \mathcal{H}$
  and    $ \pi:  \vec{\mathcal{H}}'\longrightarrow  \mathcal{H}'$
  induce  a   morphism  of  short  exact  sequences  
  \begin{eqnarray}\label{eq-251121-a3}
  \pi_*: {\bf{\rm  KU}}(\vec{\mathcal{H}},  \vec{\mathcal{H}}')
  \longrightarrow  {\bf{\rm  KU}}( {\mathcal{H}},   {\mathcal{H}}').  
  \end{eqnarray}
  The  next  theorem  is  a  straightforward 
  improvement  of  \cite[Section~5.4]{stab-hg}   and  
   \cite[Proposition~3.7  and  Theorem~3.9]{hdg}.   
  
  \begin{theorem}\label{th-251127-hku5}
  For  any  hyperdigraphs   $\vec{\mathcal{H}}$  and   $\vec{\mathcal{H}}'$  on  $V$  and  $V'$  
  respectively  where   $V$ and  $V'$  are  disjoint,  
  we  have  a  commutative  diagram 
  \begin{eqnarray}
  \label{diag-251120-ku3}
  \xymatrix{
  {\bf{\rm  KU}}(\delta\vec{\mathcal{H}},  \delta\vec{\mathcal{H}}')\ar[r]\ar[d]_-{\pi_*}
  & {\bf{\rm  KU}}( \vec{\mathcal{H}},   \vec{\mathcal{H}}')\ar[r]\ar[d]  ^-{\pi_*}
 &  {\bf{\rm  KU}}(\Delta\vec{\mathcal{H}},  \Delta\vec{\mathcal{H}}') \ar[d]^-{\pi_*}\\
   {\bf{\rm  KU}}(\delta {\mathcal{H}},  \delta {\mathcal{H}}')\ar[r] 
  & {\bf{\rm  KU}}( {\mathcal{H}},    {\mathcal{H}}')\ar[r] 
  &  {\bf{\rm  KU}}(\Delta {\mathcal{H}},  \Delta {\mathcal{H}}') 
  }
  \end{eqnarray}
  where  the  arrows are  morphisms  of  short  exact  sequences.  
Moreover,  (\ref{diag-251120-ku3})   is  functorial  with  respect  to  morphisms  of   
  hyperdigraphs and  the induced  morphisms  of  the  underlying  hypergraphs. 
  \end{theorem}
  
  \begin{proof}
  Applying  (\ref{diag-1116-1})  to  $\vec{\mathcal{H}}$  and  $\vec{\mathcal{H}}'$  respectively 
  and  taking  the  joins,  we  obtain  a  commutative  diagram 
  \begin{eqnarray}\label{diag-260107-1}
 \xymatrix{
   \delta\vec{\mathcal{H}}*   \delta\vec{\mathcal{H}}'\ar[r] \ar[d] 
   & \vec{\mathcal{H}}*  \vec{\mathcal{H}}'\ar[r]\ar[d] 
   & \Delta\vec{\mathcal{H}}* \Delta\vec{\mathcal{H}}'\ar[d] \\
   \delta {\mathcal{H}}* \delta {\mathcal{H}}'\ar[r] 
   &  {\mathcal{H}}* {\mathcal{H}}'\ar[r]
   & \Delta {\mathcal{H}}* \Delta {\mathcal{H}}'. 
   }
  \end{eqnarray}
  Applying  (\ref{diag-10-12-7})  to  the  first and  the  third  columns  
  and  applying (\ref{eq-251121-a1})  - (\ref{eq-251121-a3})  to the  middle  column,  
  we  obtain  (\ref{diag-251120-ku3}).  
 With  the  help  of    \cite[Theorem~3B.5]{hatcher}, 
 we  obtain  the  functoriality  of   (\ref{diag-251120-ku3}).  
  \end{proof}

\section{Independence  complexes,  directed  independence  complexes  
and  their  homology}\label{s6}

In  this  section,  we  apply      Section~\ref{s4} 
to  the  independence  complexes  as  well  as   the  directed  independence  complexes    
 and   study  the  homology.  
In  Theorem~\ref{th-25-10-15-1},  we  give  the  canonical  homomorphism  
from  the  homology  of  the  directed  independence  complexes  to  the  
homology  of  the  independence  complexes  and  prove  the  functoriality  
with  respect  to  filtrations  of  the  vertices.  
In  Theorem~\ref{th-10-19-1},  
we  prove  some  Mayer-Vietoris  sequences  for  the  homology 
 of  (directed)  independence complexes  of  joins  of  graphs.  
In  Theorem~\ref{th-10-15-kf-1},  
we  prove  some  K\"unneth-type  formulae  for  the  homology   
 of  (directed)  independence complexes  of  disjoint  unions   of  graphs. 
 Moreover,  we  apply  Section~\ref{s5}
 to  sub-hyper(di)graphs  of   the  (directed)  independence  complexes.  
 We   generalize     Theorem~\ref{th-25-10-15-1},  
  Theorem~\ref{th-10-19-1}  and    Theorem~\ref{th-10-15-kf-1} 
  to  the  sub-hyper(di)graph   context  in  Theorem~\ref{th-25-11-26-1},  
  Theorem~\ref{th-251126-m5}  and  Theorem~\ref{co-251127-ku1}  respectively.

Let  $G=(V,E)$  be  a  graph.  
Let  the  ordered  configuration  space  ${\rm   Conf}_k(G)$
  be  the   collection  of  all the 
ordered  $k$-tuples  $(v_1,\ldots,v_k)$  such  that  
$v_1,\ldots,v_k\in  V$,  $v_i\neq  v_j$  and  $v_i,v_j$ are  non-adjacent in  $G$
 for  any  $1\leq  i<j\leq  k$.  
 The symmetric  group  $\Sigma_k$  acts on    ${\rm   Conf}_k(G)$  freely  from  the  left 
 by  permuting  the  coordinates. 
 As  a  $k$-uniform  hyperdigraph,  
 ${\rm   Conf}_k(G)$  is  $\Sigma_k$-invariant.  
 Let  the  unordered  configuration  space  
 be  the  orbit  space  $ {\rm   Conf}_k(G)/\Sigma_k$,  which  is the   collection  of  all the 
unordered  $k$-tuples  $\{v_1,\ldots,v_k\}$  such  that  
$v_1,\ldots,v_k\in  V$,  $v_i\neq  v_j$  and  $v_i,v_j$ are  non-adjacent in  $G$
 for  any  $1\leq  i<j\leq  k$.  
 We  have  a  principal  $\Sigma_k$-bundle  
\begin{eqnarray}\label{eq-canonical-covering}
\xymatrix{
 \Sigma_k\ar[r]  &  {\rm   Conf}_k(G)\ar[r]^-{\pi_k(G) }&{\rm   Conf}_k(G)/\Sigma_k  
 }
\end{eqnarray}
where  $\pi_k(G)$   is  a $k!$-sheeted  covering  map.  
 Let  
 \begin{eqnarray*}
 \overrightarrow{\rm  Ind}(G)  =   \bigcup_{k\geq  1}  {\rm   Conf}_k(G), ~~~~~~~~~  
 {\rm  Ind}(G)  =   \bigcup_{k\geq  1}  {\rm   Conf}_k(G)/\Sigma_k. 
 \end{eqnarray*}
 Then $ \overrightarrow{\rm  Ind}(G)$  
 is  a  directed  simplicial  complex  on  $V$  and  
 $ {\rm  Ind}(G) $  is  the  underlying   simplicial  complex  of   $ \overrightarrow{\rm  Ind}(G)$.  
 The  simplicial  complex 
  ${\rm  Ind}(G)$ is  called  the {\it  independence  complex}  of  $G$,  
 which     consists  of  all  the  finite independent sets of  $G$.  
 We  call  $\overrightarrow{\rm  Ind}(G)$  the  {\it  directed  independence  complex}  of  $G$
 and  call  each  element  in  
 $ \overrightarrow{\rm  Ind}(G)$   an  {\it  independent  sequence}  on  $G$.    
 Let  $\prec$  be  a  total  order  on  $V$.  
  Note  that   an  element  in   ${\rm   Conf}_k(G)/\Sigma_k$  
  can  be  uniquely  written  in  the  form  $\{v_1,\ldots,v_k\}$
  where  $v_1\prec  \cdots \prec  v_k$  in   $V$.

  Let  $\{V_t\mid  t\in \mathbb{Z}\}$  be  a  filtration  of  $V$  such  that  
  $V_s\subseteq   V_t$  for  any  $s\leq  t$    and   $\bigcup_{t\in \mathbb{Z}}  V_t=V$.  
  For  each  $t\in \mathbb{Z}$,  let  
  \begin{eqnarray*}
  E_t=\{(u,v)\in  E\mid  u,v\in  V_t\}. 
  \end{eqnarray*}
  Then       
  $E_s\subseteq   E_t$  for  any  $s\leq  t$    and   $\bigcup_{t\in \mathbb{Z}}  E_t=E$.  
The    family  of  graphs
  \begin{eqnarray}\label{eq-251103-f1}
 \{ G_t=(V_t, E_t)\mid  t\in  \mathbb{Z}\}
  \end{eqnarray}
  gives  a   filtration  of  $G$  with  the  canonical  inclusions    
  $i_{s,t}:  G_s\longrightarrow  G_t$  for  any  $s\leq  t$  
  and  $i_{t,\infty}:  G_t\longrightarrow  G$  for  any  $t\in \mathbb{Z}$.  
  
 \begin{lemma}\label{le-251103-9}
 The  family  of  directed  simplicial  complexes 
 \begin{eqnarray}\label{eq-25-11-b1}
 \{\overrightarrow{\rm  Ind}(G_t)\mid  t\in \mathbb{Z}\}
 \end{eqnarray}
 gives  a  filtration  of  $\overrightarrow{\rm  Ind}(G)$  and  the  family  of  
 simplicial  complexes 
 \begin{eqnarray}\label{eq-25-11-b2}
 \{{\rm  Ind}(G_t)\mid  t\in \mathbb{Z}\}
 \end{eqnarray}
gives  a  filtration  of  ${\rm  Ind}(G)$
such  that  the  following  diagram  commutes 
\begin{eqnarray}\label{diag-251103-a}
\xymatrix{
\cdots\ar[r]  
& \overrightarrow{\rm  Ind}(G_s)\ar[r]\ar[d]^-{\pi} 
&\cdots\ar[r]
&\overrightarrow{\rm  Ind}(G_t) \ar[r]\ar[d]^-{\pi} 
 &\cdots \ar[r] 
&\overrightarrow{\rm  Ind}(G)\ar[d]^-{\pi} \\
\cdots\ar[r]  
& {\rm  Ind}(G_s)\ar[r]
&\cdots\ar[r]
& {\rm  Ind}(G_t) \ar[r] &\cdots \ar[r]
& {\rm  Ind}(G)
}
\end{eqnarray}
 for  any  $s\leq  t$.  
 \end{lemma}
 \begin{proof}
 Let  $u,v\in  V_t$.  
 Then  $u$  and  $v$ are  adjacent  in  $G_t$  iff  they  are  adjacent  in  $G$.  
 Thus  for any  distinct  vertices   $v_1,\ldots,  v_k\in  V_t$,   
 $(v_1,\ldots,  v_k)$  is  a  directed  simplex  in  $\overrightarrow{\rm  Ind}(G_t)$  
 iff  it  is  a  directed  simplex  in  $ \overrightarrow{\rm  Ind}(G)$;
 and    $\{v_1,\ldots,  v_k\}$  is  a  simplex  in  ${\rm  Ind}(G_t)$  
 iff  it  is  a    simplex  in  ${\rm  Ind}(G)$.  
 Therefore,  
 (\ref{eq-25-11-b1})   is   a   filtration  of  $ \overrightarrow{\rm  Ind}(G)$
 and   (\ref{eq-25-11-b2})   is   a   filtration  of  ${\rm  Ind}(G)$
 such  that  the  diagram  (\ref{diag-251103-a})   commutes.  
 \end{proof}
 
 \begin{theorem}\label{th-25-10-15-1}
  We  have a    surjective  chain  map 
 \begin{eqnarray}\label{eq-25-10-17}
 \pi_\#:   C_\bullet(\overrightarrow{\rm  Ind}(G);R)
 \longrightarrow  C_\bullet({\rm  Ind}(G);R)
\end{eqnarray}
    and   consequently  an  induced      homomorphism  of  homology   
\begin{eqnarray}\label{eq-25-10-18}
 \pi_*:   H_\bullet(\overrightarrow{\rm  Ind}(G);R)\longrightarrow  H_\bullet({\rm  Ind}(G);R).  
\end{eqnarray}
Moreover,    (\ref{eq-25-10-17})   and   (\ref{eq-25-10-18})  are   functorial  with  respect  to 
 directed  simplicial  maps  between  directed  independence  complexes   
 and  their  induced  simplicial  maps   between  the  underlying  independence  complexes,
  both  of  which  are  
 induced  by  filtrations  of  the  vertices  of  $G$.  
 \end{theorem}
 
 \begin{proof}
 The  proof  follows  from  Proposition~\ref{le-25-10-5}  and   Lemma~\ref{le-251103-9}.  
Let   $\vec{\mathcal{K}}$  be  $\overrightarrow{\rm  Ind}(G)$  in    
 (\ref{eq-25-10-7})  and  (\ref{eq-25-10-8}).  
 We  obtain   (\ref{eq-25-10-17})  and  (\ref{eq-25-10-18})  respectively.
 Let    $\{V_t\mid  t\in \mathbb{Z}\}$  be  a  filtration  of  $V$.   
 By  Lemma~\ref{le-251103-9},
  there  is  an  induced  filtration  
  (\ref{eq-25-11-b1})  of  $\overrightarrow{\rm  Ind}(G)$
  as  well  as  an  induced  filtration (\ref{eq-25-11-b2})  of    
  ${\rm  Ind}(G)$.    
 It  follows from   the  commutativity  of  (\ref{diag-251103-a})  that   
  (\ref{eq-25-10-17})   and   (\ref{eq-25-10-18})  are   functorial 
 with  respect  to  the  canonical  inclusions  in  (\ref{eq-25-11-b1})  and  (\ref{eq-25-11-b2})
 respectively.  
 \end{proof}
 
 \begin{theorem}\label{th-25-11-26-1}
  For  any   hyperdigraph  $\vec{\mathcal{H}} $  consisting  of  certain  independent  sequences  
  on  $G$,  let  $\mathcal{H}$  be  its  underlying  hypergraph. 
  Then   we  have  a  commutative  diagram  of  chain  complexes 
(\ref{diag-1116-2})  
  such  that  the vertical  maps  are  surjective  chain  maps  
  and  the  horizontal  maps  are  canonical  inclusions.  
  Consequently,  we  have  an  induced  commutative  diagram  of  homology groups 
 (\ref{diag-1116-3}).  
  Moreover,    (\ref{diag-1116-2})     and  (\ref{diag-1116-3})    are   functorial  in  the  
  sense  of  Theorem~\ref{th-25-10-15-1}.  
 \end{theorem}

\begin{proof}
Since  $\Delta\vec{\mathcal{H}}$  is  a  directed  simplicial  sub-complex  of  
 $\overrightarrow{\rm  Ind}(G)$,  
we  restrict  (\ref{eq-25-10-17})  to  
the  sub-chain  complexes  of  $ C_\bullet(\overrightarrow{\rm  Ind}(G);R)$  in  the  first  row  of   (\ref{diag-1116-2}).  
We  obtain   (\ref{diag-1116-2})     and  (\ref{diag-1116-3}). 
The  filtration  (\ref{eq-25-11-b1})   induces  a  filtration  $\{\vec{\mathcal{H}}_t\mid  t\in \mathbb{Z}\}$ 
 where  $\vec{\mathcal{H}}_t=\vec{\mathcal{H}}\cap  \overrightarrow{\rm  Ind}(G_t)$
  and  the  filtration  (\ref{eq-25-11-b2})   induces 
   a  filtration  $\{\mathcal{H}_t\mid  t\in \mathbb{Z}\}$ 
 where  $\mathcal{H}_t=\mathcal{H}\cap   {\rm  Ind}(G_t)$.  
The  functoriality  of  (\ref{diag-1116-2})     and  (\ref{diag-1116-3})
  follows  from  Theorem~\ref{th-251127-hh1}  and  Theorem~\ref{th-25-10-15-1}.  
\end{proof}

 \subsection{The  Mayer-Vietoris  sequences }

Let $V'$,  $V''$  and  $V'''$  be  mutually  disjoint 
sets  of  vertices.  
Let  $G'$,  $G''$  and  $G'''$  be  graphs  on  $V'$,  $V''$  and  $V'''$  
respectively.  
With  the  help  of  Example~\ref{ex-25-10-19-1}, 
  their   reduced  join    is  a  graph 
\begin{eqnarray}\label{eq-10-19-31}
G=G' \tilde{*}  G''  \tilde{*}  G''' .  
\end{eqnarray} 
Consider the  graphs  
\begin{eqnarray}\label{eq-25-11-11-1}
L'= G'  \tilde{*}  G''',~~~~~~  L''=G''  \tilde{*}  G'''. 
\end{eqnarray}
Then  
\begin{eqnarray*}
L'\cap  L''= G'''. 
\end{eqnarray*}

\begin{lemma}\label{le-10-19-1}
We  have  
\begin{eqnarray}
\label{eq-25-10-21}
\overrightarrow {\rm  Ind}(G''') & = &\overrightarrow  {\rm  Ind}(L' ) \cap \overrightarrow {\rm  Ind}(   L''),\\
\label{eq-25-10-22}
\overrightarrow {\rm  Ind}(G) & = & \overrightarrow {\rm  Ind}(L' ) \cup \overrightarrow {\rm  Ind}(   L'').  
\end{eqnarray}
\end{lemma}
\begin{proof}
Since  $L'$  is  obtained  from  the  disjoint  union  of  $G'$  and  $G'''$  
by  connecting  each  pair  of  vertices  $(v',v''')$  where  $v'$  is  a  
vertex  of  $G'$  and  $v'''$  is  a  vertex  of  $G'''$,  we  have  
a  disjoint  union  
\begin{eqnarray}\label{eq-10-19-51}
\overrightarrow {\rm  Ind}(L')=  \overrightarrow {\rm  Ind}(G')  \sqcup\overrightarrow  {\rm  Ind}(G'''). 
\end{eqnarray}  
Similarly,  
\begin{eqnarray}\label{eq-10-19-52}
\overrightarrow {\rm  Ind}(L'')=  \overrightarrow {\rm  Ind}(G'')  \sqcup \overrightarrow {\rm  Ind}(G'''). 
\end{eqnarray}  
By  (\ref{eq-10-19-51})  and  (\ref{eq-10-19-52}),  
we  obtain  (\ref{eq-25-10-21}).  
By  (\ref{eq-10-19-31}),  we  have  
a  disjoint  union  
   \begin{eqnarray}\label{eq-10-19-53}
\overrightarrow {\rm  Ind}(G)=  \overrightarrow {\rm  Ind}(G')  \sqcup\overrightarrow  {\rm  Ind}(G'') \overrightarrow \sqcup {\rm  Ind}(G'''). 
\end{eqnarray}  
By  (\ref{eq-10-19-51})  -  (\ref{eq-10-19-53}),  
we  obtain  (\ref{eq-25-10-22}).  
\end{proof}

\begin{theorem}\label{th-10-19-1}
For  any  graph  $G$  given  as a  reduced  join  (\ref{eq-10-19-31}),  
 we  have  a  commutative   diagram  of  homology  groups  
 \begin{eqnarray}\label{diag-10-19-2}
 \xymatrix{
\cdots\ar[r]  
&H_n (\overrightarrow{\rm  Ind}(G''') ;R) \ar[r]\ar[d]_-{(\pi_{n+1})_*}
&H_n(\overrightarrow{\rm  Ind}(G'  \tilde{*}  G''');R)\oplus 
 H_n( \overrightarrow{\rm  Ind}(G''  \tilde{*}  G''');R)\ar[r]\ar[d]_-{(\pi_{n+1})_*}
 &\\
\cdots\ar[r]  
 &H_n ( {\rm  Ind}(G''');R) \ar[r] 
&H_n({\rm  Ind}(G'  \tilde{*}  G''');R)\oplus  H_n( {\rm  Ind}(G''  \tilde{*}  G''');R)\ar[r] 
 &\\
\ar[r]&H_n(\overrightarrow{\rm  Ind}(G );R)\ar[r]\ar[d]_-{(\pi_{n+1})_*}
&H_{n-1}(\overrightarrow{\rm  Ind}(G''');R)\ar[r]\ar[d]_-{(\pi_{n})_*}
&\cdots\\
\ar[r]&H_n( {\rm  Ind}(G );R)\ar[r] 
&H_{n-1}( {\rm  Ind}(G''');R)\ar[r] 
&\cdots
} 
 \end{eqnarray}
 such that  the  two  rows  are  long  exact sequences. 
 Moreover,    the  diagram   (\ref{diag-10-19-2})   is    functorial  with  respect  to 
 directed  simplicial  maps  between  directed  independence  complexes   
 and  their  induced  simplicial  maps   between  the  underlying  independence  complexes, 
 both  of  which  are  
 induced  by  filtrations  of  the  vertices  of  $G'$,  $G''$  and  $G'''$.   
\end{theorem}

\begin{proof}
Let   $\vec{\mathcal{K}}$  be  $\overrightarrow{\rm  Ind}(L')$ 
and  let  $\vec{\mathcal{K}}'$  be  $\overrightarrow{\rm  Ind}(L'')$ in  
Proposition~\ref{pr-10-11-mv-1}.  
By  (\ref{eq-25-10-11-2})  and  (\ref{eq-25-10-22}),  
\begin{eqnarray}
\label{eq-25-10-82}
 {\rm  Ind}(G)  =    {\rm  Ind}(L' ) \cup   {\rm  Ind}(   L'').  
\end{eqnarray}
By  (\ref{eq-25-10-11-3})  and  (\ref{eq-10-19-51})  -  (\ref{eq-10-19-53}),
\begin{eqnarray}
\label{eq-25-10-83}
 {\rm  Ind}(G''')   =  {\rm  Ind}(L'\cap  L'') =  {\rm  Ind}(L' ) \cap   {\rm  Ind}(   L'').  
\end{eqnarray}
With  the  help  of  Lemma~\ref{le-10-19-1},  (\ref{eq-25-10-82})  and  
(\ref{eq-25-10-83}),   
we  obtain  the  commutative  diagram (\ref{diag-10-19-2})  from 
 the  commutative  diagram  (\ref{diag-10-12-2}).

 Let  $\{V'_t\mid  t\in \mathbb{Z}\}$,  
 $\{V''_t\mid  t\in \mathbb{Z}\}$  
 and  $\{V'''_t\mid  t\in \mathbb{Z}\}$  be  filtrations  of  
 $V'$,  $V''$  and  $V'''$  respectively.  
 By  (\ref{eq-251103-f1}),  we  have  filtrations  
 $\{G'_t,  G''_t,  G'''_t\mid  t\in \mathbb{Z}\}$  of  $G'$,  $G''$   and  $G'''$
 respectively 
 such  that  the  diagram  commutes 
 \begin{eqnarray*}
 \xymatrix{
 \cdots\ar[r]  
 & (G'_s,  G''_s,  G'''_s)\ar[d]^-{\tilde{*}}\ar[r] 
 &\cdots\ar[r]
 &(G'_t,  G''_t,  G'''_t)\ar[d]^-{\tilde{*}}\ar[r] 
  &\cdots\ar[r]
   &(G',  G'',  G''')\ar[d]^-{\tilde{*}}\\
   \cdots\ar[r]  
 & G'_s  \tilde{*}  G''_s  \tilde{*}  G'''_s  \ar[r] 
 &\cdots\ar[r]
 & G'_t  \tilde{*}  G''_t  \tilde{*} G'''_t \ar[r] 
  &\cdots\ar[r]
   &G'  \tilde{*} G''  \tilde{*} G'''  
 }
 \end{eqnarray*} 
 for  any  $s\leq  t$.  
 Let 
 \begin{eqnarray*} 
 G_t=G'_t \tilde{*}  G''_t  \tilde{*}  G'''_t,  ~~~~~~ 
L'_t= G'_t  \tilde{*}  G'''_t,~~~~~~  L''_t=G''_t  \tilde{*}  G'''_t 
\end{eqnarray*}
for  any  $t\in \mathbb{Z}$,  which  give  
   filtrations  of  $G$,  $L'$  and  $L''$  respectively.  
The     commutative  diagram  
\begin{eqnarray*}
\xymatrix{
 & L'_t \ar[rd] &\\
G'''_t \ar[ru]\ar[rd] &&  G_t\\
& L''_t \ar[ru] &
}
\end{eqnarray*}
is  functorial  with  respect to 
 the  canonical  inclusions  of  graphs  induced  by the  filtrations
of  $G'''$,  $G$,  $L'$  and  $L''$.  
 Consequently,    
we  have  a   commutative   diagram   
\begin{eqnarray}\label{diag-251103-double}
\xymatrix{
&&\overrightarrow  {\rm  Ind}(L'_t)\ar[d]^-{\pi}\ar[rrdd]&&\\
&&   {\rm  Ind}(L' _t)\ar[rd]&\\
\overrightarrow  {\rm  Ind}(G'''_t)\ar[r]^-{\pi}\ar[rrdd]  \ar[rruu]
&  {\rm  Ind}(G'''_t)  \ar[ru]\ar[rd] &&   {\rm  Ind}( G_t)
&\overrightarrow  {\rm  Ind}( G_t)\ar[l]_-{\pi}\\
&&  {\rm  Ind}(L''_t)\ar[ru] &&\\
&& \overrightarrow  {\rm  Ind}(L''_t) \ar[u]_-{\pi}\ar[rruu]&&
}
\end{eqnarray}
where  all  the  unlabeled  arrows  are  canonical  inclusions  of  directed  simplicial  complexes 
and   canonical  inclusions  of simplicial  complexes,  such  that  the  diagram 
is  functorial  with  respect to 
 the  canonical  inclusions  of  (directed)  independence  complexes   induced  by the  filtrations
of  $G'''$,  $G$,  $L'$  and  $L''$.  
Therefore,    we  obtain  the  functoriality  of    (\ref{diag-10-19-2}).   
\end{proof}

By  Subsection~\ref{ss-3.1},  
the  long exact  sequence    in  the first  row  of  (\ref{diag-10-19-2}) 
 is  denoted  by  $
 {\bf  \rm  MV}(\overrightarrow{\rm  Ind}(L'), \overrightarrow{\rm  Ind}(L'')) 
$ 
and      the  long exact  sequence     in  the second  row  of  (\ref{diag-10-19-2}) 
is  denoted  by  $
 {\bf  \rm  MV}( {\rm  Ind}(L'),  {\rm  Ind}(L'')) 
$.  
The  diagram  (\ref{diag-10-19-2})  is  denoted  by  a  morphism  of  long  exact  sequences 
\begin{eqnarray}\label{eq-251121-mor1}
\pi_*:   {\bf  \rm  MV}(\overrightarrow{\rm  Ind}(L'), \overrightarrow{\rm  Ind}(L'')) 
\longrightarrow   {\bf  \rm  MV}( {\rm  Ind}(L'),  {\rm  Ind}(L'')) .  
\end{eqnarray}

\begin{corollary}
\label{pr-251125-mv1}
For  any  $k\geq  0$  and  any  graph  $G$  given  as a  reduced  join  (\ref{eq-10-19-31}),  
 we  have  a
 morphism  of  long  exact sequences  
\begin{eqnarray}\label{eq-251125-mor21}
\pi_*:   {\bf  \rm  MV}({\rm  sk}^{k}(\overrightarrow{\rm  Ind}(L')), {\rm  sk}^{k}(\overrightarrow{   \rm  Ind}(L''))) 
\longrightarrow   {\bf  \rm    MV}({\rm  sk}^{k} ({\rm  Ind}(L')),  ({\rm  sk}^{k}{\rm  Ind}(L'')))    
\end{eqnarray}
which   is   functorial  in  the  sense  of  Theorem~\ref{th-10-19-1}. 
\end{corollary}

\begin{proof}
The  proof  is  an  analog  of  Theorem~\ref{th-10-19-1}.  
\end{proof}

\begin{theorem}\label{th-251126-m5}
For  any   hyperdigraphs  $\vec{\mathcal{H}}' $,
$\vec{\mathcal{H}}'' $  and
$\vec{\mathcal{H}}'''$  
consisting  of  certain  independent  sequences  
  on  $G'$,  $G'' $  and  $G'''$ given  in   (\ref{eq-10-19-31}) respectively,  let  
  $\mathcal{H}'$,  $\mathcal{H}''$  and  $\mathcal{H}'''$  be  the  underlying  hypergraphs.  
  Then  we  have  a  commutative  diagram  of  long   exact  sequences
   \begin{eqnarray}\label{eq-251127-h1a}
  \xymatrix{
  {\bf  \rm  MV}(\delta\vec{\mathcal{H}}'\sqcup \delta\vec{\mathcal{H}}''', \delta\vec{\mathcal{H}}''\sqcup \delta\vec{\mathcal{H}}''')  \ar[r]\ar[d]  
  &  {\bf  \rm  MV}(\Delta\vec{\mathcal{H}}'\sqcup \Delta\vec{\mathcal{H}}''', \Delta\vec{\mathcal{H}}''\sqcup \Delta\vec{\mathcal{H}}''')   \ar[d]  \\
  {\bf  \rm  MV}(\delta {\mathcal{H}}'\sqcup \delta {\mathcal{H}}''', \delta {\mathcal{H}}''\sqcup \delta {\mathcal{H}}''')  \ar[r]   
  &  {\bf  \rm  MV}(\Delta {\mathcal{H}}'\sqcup \Delta {\mathcal{H}}''', \Delta {\mathcal{H}}''\sqcup \Delta {\mathcal{H}}''').  
  }
  \end{eqnarray}
  In  addition,    
    if  both   
  (I)  and  (II)  in  Subsection~\ref{ss-5.1}  are  satisfied
  for  the  pair  of  hyperdigraphs  $ \vec{\mathcal{H}}'\sqcup  \vec{\mathcal{H}}''' $
  and  $ \vec{\mathcal{H}}''\sqcup  \vec{\mathcal{H}}''' $,   
  then  we have  a  commutative  diagram  of  long  exact  sequences 
  \begin{eqnarray}\label{eq-251127-h2a}
  \xymatrix{
  {\bf  \rm  MV}(\delta\vec{\mathcal{H}}'\sqcup \delta\vec{\mathcal{H}}''', \delta\vec{\mathcal{H}}''\sqcup \delta\vec{\mathcal{H}}''')  \ar[r]\ar[d]  
  & {\bf  \rm  MV}(\vec{\mathcal{H}}'\sqcup \vec{\mathcal{H}}''', \vec{\mathcal{H}}''\sqcup \vec{\mathcal{H}}''')  \ar[r]\ar[d]  
  &  {\bf  \rm  MV}(\Delta\vec{\mathcal{H}}'\sqcup \Delta\vec{\mathcal{H}}''', \Delta\vec{\mathcal{H}}''\sqcup \Delta\vec{\mathcal{H}}''')   \ar[d]  \\
  {\bf  \rm  MV}(\delta {\mathcal{H}}'\sqcup \delta {\mathcal{H}}''', \delta {\mathcal{H}}''\sqcup \delta {\mathcal{H}}''')  \ar[r]   
  & {\bf  \rm  MV}( {\mathcal{H}}'\sqcup  {\mathcal{H}}''',  {\mathcal{H}}''\sqcup  {\mathcal{H}}''')  \ar[r]
  &  {\bf  \rm  MV}(\Delta {\mathcal{H}}'\sqcup \Delta {\mathcal{H}}''', \Delta {\mathcal{H}}''\sqcup \Delta {\mathcal{H}}''').  
  }
  \end{eqnarray}
Both  (\ref{eq-251127-h1a})
and (\ref{eq-251127-h2a})  are    functorial  in  the  sense  of  Theorem~\ref{th-10-19-1}. 
\end{theorem}

\begin{proof}
The  proof  follows  from  Theorem~\ref{th-251118-mv1},
 Theorem~\ref{th-25-11-26-1}   and  Theorem~\ref{th-10-19-1}.  
\end{proof}

\subsection{The  K\"unneth-type   formulae}

 Let  $V$  and  $V'$  be  disjoint  sets  of  vertices.  
 Let  $G$  be  a  graph  with  vertices  in  $V$  and
 let  $G'$  be  a  graph  with  vertices  in $V'$.   
 The  disjoint  union  of  $G$  and  $G'$  is  the graph     
 \begin{eqnarray*}
 G\sqcup  G'= (V\sqcup  V',  E_{G}\sqcup  E_{G'}). 
 \end{eqnarray*} 
It  is  direct  that  
 \begin{eqnarray}
  \overrightarrow{{\rm  Ind}}(G\sqcup G') &=&{\rm  Sym}( \overrightarrow{ {\rm  Ind}}(G) *
  \overrightarrow{  {\rm  Ind}}(G')),
  \label{eq-25-10-15-a1}\\
 {\rm  Ind}(G\sqcup G') &=&  {\rm  Ind}(G) *  {\rm  Ind}(G').  
   \label{eq-25-10-15-a2}
 \end{eqnarray}
 Hence  by  (\ref{eq-25-10-2}), (\ref{eq-25-10-15-a1})  and  (\ref{eq-25-10-15-a2}),  
 for  any  $k\geq  1$,  
 \begin{eqnarray*}
 H_{k-1}(\overrightarrow{{\rm  Ind}}(G\sqcup G');R) 
    = H_{k-1}(  {\rm  Ind}(G\sqcup G');R)^{\oplus  k!}=  H_{k-1}( {\rm  Ind}(G) *
   {\rm  Ind}(G');R)^{\oplus  k!}  .  
 \end{eqnarray*}

 \begin{theorem}
\label{th-10-15-kf-1}
For  any  graph  $G$  with  vertices in $ V$  
and  any  graph  $G'$  with  vertices  in  $V'$  
such that  $V$  and  $V'$  are  disjoint,  
 we  have  a  commutative   diagram    
 \begin{eqnarray}\label{diag-10-15-11}
\xymatrix{
0\ar[r]& \bigoplus_{p+q+1=n} H_{p+1}(\overrightarrow{ {\rm  Ind}}(G);R)\otimes H_{q+1}(\overrightarrow{ {\rm  Ind}}(G');R)
\ar[r]\ar[d] _-{ \bigoplus_{p+q+1=n}(\pi_{p+2})_*\otimes (\pi_{q+2})_*}
&H_{n+1}(\overrightarrow{ {\rm  Ind}}(G) *
  \overrightarrow{  {\rm  Ind}}(G');R) \ar[r]\ar[d]_-{(\pi_{n+2})_*} & \\
0\ar[r]& \bigoplus_{p+q+1=n} H_{p+1}(  {\rm  Ind}(G);R)\otimes H_{q+1}({\rm  Ind}(G');R)
\ar[r]
&H_{n+1}({\rm  Ind}(G\sqcup G') ;R)\ar[r] &
}\\
~~~\nonumber\\
\xymatrix{
\ar[r] &\bigoplus_{p+q+1=n} {\rm  Tor}_R(H_{p+1}(\overrightarrow{ {\rm  Ind}}(G);R), H_{q}(\overrightarrow{  {\rm  Ind}}(G');R))\ar[r]\ar[d] &0 \\
\ar[r]  &\bigoplus_{p+q+1=n} {\rm  Tor}_R(H_{p+1}({\rm  Ind}(G);R), H_{q}( {\rm  Ind}(G');R))\ar[r]  &0
}
\nonumber
 \end{eqnarray}
 such that  the  two  rows  are  short  exact sequences.  
 Moreover,    the  diagram   (\ref{diag-10-15-11})   is    functorial  with  respect  to 
 directed  simplicial  maps  between  directed  independence  complexes   
 and  their  induced  simplicial  maps   between  the  underlying  independence  complexes, 
 both  of  which  are  
 induced  by  filtrations  of  the  vertices  of  $G$  and  $G'$.   
 \end{theorem}
 
 \begin{proof}
Let   $\vec{\mathcal{K}}$  be  $\overrightarrow{\rm  Ind}(G)$ 
and  let  $\vec{\mathcal{K}}'$  be  $\overrightarrow{\rm  Ind}(G')$ in  
Proposition~\ref{pr-10-12-kf-1}.  
With  the  help  of  (\ref{eq-25-10-15-a2}),  
we  obtain  the  commutative  diagram (\ref{diag-10-15-11})  from 
 the  commutative  diagram  (\ref{diag-10-12-7}).

  Let  $\{V_t\mid  t\in \mathbb{Z}\}$  and   
 $\{V'_t\mid  t\in \mathbb{Z}\}$    be  filtrations  of  
 $V$   and   $V'$    respectively.  
 By  (\ref{eq-251103-f1}), 
 we  have  induced  filtrations  $\{G_t,  G'_t\mid  t\in \mathbb{Z}\}$  
 of  $G $  and  $G'$  respectively.  
 The  commutative  diagram 
 \begin{eqnarray*}
 \xymatrix{
 (\overrightarrow{ {\rm  Ind}}(G_t), \overrightarrow{ {\rm  Ind}}(G'_t)) \ar[r]^-{*}\ar[d]_-{\pi}
 & \overrightarrow{ {\rm  Ind}}(G_t)  * \overrightarrow{ {\rm  Ind}}(G'_t)\ar[d]^-{\pi}\\
  (  {\rm  Ind}(G_t),   {\rm  Ind}(G'_t))\ar[r]^-{*}
  &   {\rm  Ind}(G_t\sqcup  G'_t)
 }
 \end{eqnarray*}
 is  functorial  with  respect to 
 the  canonical  inclusions  of  (directed)  independence  complexes   induced  by the  filtrations
of  $G$  and  $G'$.  
Therefore,     we  obtain  the  functoriality  of    (\ref{diag-10-15-11}).   
 \end{proof}

 By  Subsection~\ref{ss-3.2},  
 the  short  exact  sequence    in  the first  row  of  (\ref{diag-10-15-11}) 
 is  denoted  by  $
 {\bf  \rm  KU}(\overrightarrow{\rm  Ind}(G), \overrightarrow{\rm  Ind}(G')) 
$ 
and      the  short  exact  sequence     in  the second  row  of  (\ref{diag-10-15-11}) 
is  denoted  by  $
 {\bf  \rm   KU}( {\rm  Ind}(G),  {\rm  Ind}(G')) 
$.  
The  diagram  (\ref{diag-10-15-11})  is  denoted  by  a  morphism  of  short  exact  sequences 
\begin{eqnarray}\label{eq-251121-mor2}
\pi_*:   {\bf  \rm  KU}(\overrightarrow{\rm  Ind}(G), \overrightarrow{   \rm  Ind}(G')) 
\longrightarrow   {\bf  \rm   KU}( {\rm  Ind}(G),  {\rm  Ind}(G')) .  
\end{eqnarray}

\begin{corollary}
\label{pr-251112-ku}
Let  $G$  and  $G'$  be  the  graphs  in  Theorem~\ref{th-10-15-kf-1}.  
For  any  $k\geq  0$,    
We   have  a  morphism  of  short  exact sequences  
\begin{eqnarray}\label{eq-251125-mor1}
\pi_*:   {\bf  \rm  KU}({\rm  sk}^{k}(\overrightarrow{\rm  Ind}(G)), {\rm  sk}^{k}(\overrightarrow{   \rm  Ind}(G'))) 
\longrightarrow   {\bf  \rm   KU}({\rm  sk}^{k} ({\rm  Ind}(G)),  ({\rm  sk}^{k}{\rm  Ind}(G')))   
\end{eqnarray}
which   is    functorial  in  the  sense  of  Theorem~\ref{th-10-15-kf-1}.   
\end{corollary}

\begin{proof}
The  proof  is  an  analog  of  Theorem~\ref{th-10-15-kf-1}.  
\end{proof}

\begin{theorem}\label{co-251127-ku1}
Let  $G$  and  $G'$  be  the  graphs  in  Theorem~\ref{th-10-15-kf-1}.  
For  any   hyperdigraphs  $\vec{\mathcal{H}} $  and   $\vec{\mathcal{H}} '$  consisting  of  certain  independent  sequences  
  on  $G$  and  $G'$  respectively,  let  $\mathcal{H}$  
  and  $\mathcal{H}'$  be  their  underlying  hypergraphs. 
Then  
we  have  a  commutative  diagram  of  short   exact  sequences
(\ref{diag-251120-ku3})  
which   is    functorial  in  the  sense  of  Theorem~\ref{th-10-15-kf-1}.   
\end{theorem}

\begin{proof}
Since  $\vec{\mathcal{H}} $   is  a  sub-hyperdigraph  of  
$\overrightarrow{\rm  Ind}(G)$  and  $\vec{\mathcal{H}} '$   is  a  sub-hyperdigraph  of  
$\overrightarrow{\rm  Ind}(G')$,
the  first row  of   (\ref{diag-260107-1})  are   sub-hyperdigraphs  of  
$\overrightarrow{\rm  Ind}(G) *  \overrightarrow{\rm  Ind}(G')$.  
 By   Theorem~\ref{th-251127-hku5},  
 we  obtain  the  commutative  diagram (\ref{diag-251120-ku3}). 
 By    
Theorem~\ref{th-25-11-26-1}    
and  Theorem~\ref{th-10-15-kf-1}, 
we  obtain  the  functoriality.   
\end{proof}
  
\section{Matroids,  directed  matroids     and  their  homology}
\label{s7}

In  this  section,  we  apply      Section~\ref{s4} 
to  matroids  as  well  as      directed  matroids    
 and   study  the  homology.  
In  Theorem~\ref{co-25-10-27},  we  give  the  canonical  homomorphism  
from  the  homology  of     directed  matroids  to  the  
homology  of  matroids.  
In  Theorem~\ref{co-10-27-mv-1},  
we  prove  some  Mayer-Vietoris  sequences  for  the  homology 
 of  (directed)  matroids.  
In  Theorem~\ref{co-10-27-kf-1},  
we  prove  some  K\"unneth-type  formulae  for  the  homology   
 of  (directed)  matroids. 
 Moreover,  we  apply  Section~\ref{s5}
 to  sub-hyper(di)graphs  of      (directed)  matroids.  
 We   generalize     Theorem~\ref{co-25-10-27},  
  Theorem~\ref{co-10-27-mv-1}  and    Theorem~\ref{co-10-27-kf-1} 
  to  the  sub-hyper(di)graph   context  in  Theorem~\ref{co-251127-mat-1},  
  Theorem~\ref{co-251127-mat-mv-a1}  and  Theorem~\ref{co-251127-mat-ku-a1}
    respectively.

A {\it  matroid}  $\mathcal{M}$  is  a  finite  set $S$ 
  and  a  collection $\mathcal{I}$ of  
    subsets  of  $S$ 
(called  {\it  independent  sets})   such  that  
(I1)  $\emptyset\in \mathcal{I}$;  (I2)  If  $\sigma\in  \mathcal{I}$ and  $\tau\subseteq \sigma$
then  $\tau\in \mathcal{I}$;  (I3)  If  $\sigma_1,\sigma_2\in \mathcal{I}$  with  
$|\sigma_2|=|\sigma_1|+1$,  then there  exists  $x\in \sigma_2-\sigma_1$  such that 
$\sigma_1\cup   x\in  \mathcal{I}$  (\cite[Section~1.2]{matroid}).  
We  say  that  a  matroid   $\mathcal{M}'=(S',\mathcal{I}')$  is  
a  submatroid   of  $\mathcal{M}$  
if   $S'\subseteq  S$  and  $\mathcal{I}'\subseteq \mathcal{I}$. 
 In  particular,  if  we  let    
  $\mathcal{I}=[2^S]$  be    the  collection  of  all  the   subsets  of  $S$,  then 
 $[S]=(S, [2^S])$  is  a  matroid.  
 For  any  matroid  $\mathcal{M}=(S, \mathcal{I})$, 
   $\mathcal{I}$  must be  a  subset  of  $[2^S]$  
 thus  $\mathcal{M}$  must  be  a  submatroid  of  $[S]$.

A   maximal  independent  set  in  $\mathcal{I}$  
is  called  a  {\it  base}  of  $\mathcal{M}$  (\cite[p. 7]{matroid}).  
Let  $\mathbb{B}(\mathcal{M})$  be  the  set  of  bases  of  $\mathcal{M}$.  
The  {\it  dual  matroid}  $\mathcal{M}^*$  of  $\mathcal{M}$  
is  the matroid  with  the  set  of  bases   (\cite[Section~2.1]{matroid})
\begin{eqnarray}\label{eq-25-10-30-1}
 \mathbb{B}(\mathcal{M}^*)=\{S-B\mid  B\in \mathbb{B}(\mathcal{M})\}.  
\end{eqnarray} 
The  {\it  rank}  of  $\mathcal{M}$   is  
 $r(\mathcal{M})=\max\{|\sigma|\mid  \sigma\in \mathcal{I}\}$.   
 Note  that  $r(\mathcal{M})+  r(\mathcal{M}^*)=|S|$  (cf.  \cite[Section~2.1]{matroid}).

 Similar  to  the  definition  of  a  matroid,  
 we define  a  {\it  directed  matroid}  $\vec {\mathcal{ M}}$  to  be  
 a   finite  set $S$  and  a  collection $\vec {\mathcal{I}}$ of  finite  
 sequences  of  distinct  elements  in  $S$ 
(we  call these  sequences     {\it  independent  sequences})   such  that  
(I1)'  the  empty  sequence  $\vec\emptyset\in \vec{\mathcal{I}}$; 
 (I2)'  If  $\vec \sigma\in  \vec{\mathcal{I}}$ and     $\vec\tau$  is  a  subsequence  of  $\vec\sigma$,  
then  $\vec\tau\in \vec{\mathcal{I}}$;  (I3)'  If  $\vec{\sigma}_1,\vec{\sigma}_2\in \vec{\mathcal{I}}$  with  
$|\vec\sigma_2|=|\vec\sigma_1|+1$,  then there  exist   $x\in \vec\sigma_2$     
and  $\vec\sigma\in \vec{\mathcal{I}}$  such  that  $x\notin \vec\sigma_1$ 
and  $\vec\sigma$  is  obtained  by  adding  $x$  into  $\vec\sigma_1$  \footnote[5]{  Here  
   $\vec\sigma$  is  obtained  by  adding  $x$  into  $\vec\sigma_1$  
means  that   if  we  write  
$\vec\sigma_1=  v_1v_2\ldots  v_k$,  
where  $v_1,v_2,\ldots,v_k\in  S$  are  distinct,  then  there  exists  $1\leq  i\leq  k+1$  such  that  
$\vec \sigma= v_1\ldots  v_{i-1} x v_i \ldots  v_k$.  }.    
We  say  that  a  directed  matroid   $\vec{\mathcal{M}}'=(S', \vec{\mathcal{I}}')$  is  
a  directed  submatroid   of  $\vec{\mathcal{M}}$  
if   $S'\subseteq  S$  and  $\vec{\mathcal{I}}'\subseteq\vec{ \mathcal{I}}$. 
 In  particular,  if     
 we  let  $\vec{\mathcal{I}}=\vec 2^S$    be  set  of  all  the 
  finite  sequences  of   distinct  elements  in  $S$, 
 then  $(S)=(S,  \vec 2^S)$  is  a  directed  matroid.  
  For  any  directed  matroid  $\vec {\mathcal{ M}}=(S, \vec{\mathcal{I}})$, 
 we  have  that  $\vec{\mathcal{I}}$  must be  a  subset  of  $\vec{2}^S$  
 thus  $\vec {\mathcal{ M}}$  must  be  a  directed  submatroid  of  $(S)$.

We  have  a    canonical  projection  
$
 \pi:  (S)\longrightarrow  [S]
$
sending  each    sequence  of  distinct   elements  in  $S$  to  the  
   set  of  the  elements  in  the  sequence. 
Let  $\vec {\mathcal{ M}}=(S, \vec{\mathcal{I}})$  be  any   directed  matroid.  
We  define  the  {\it  underlying  matroid}  
 $ \mathcal{ M}=(S,  \mathcal{I})$  of  $\vec {\mathcal{ M}}$  by  
 letting  
 $\mathcal{I}=\pi(\vec{\mathcal{I}})$.  
 It  is  direct  to  verify  that  $\mathcal{M}$  is  a  matroid. 
 Restricted  to $\vec {\mathcal{ M}}$,   $\pi$  induces  a   canonical  projection  
 \begin{eqnarray}\label{eq-25-10-27-1}
 \pi:   \vec {\mathcal{ M}}\longrightarrow  \mathcal{M}. 
 \end{eqnarray}  
 We  say  that   $\vec  {\mathcal{M}}$  is  {\it  $\Sigma_\bullet$-invariant}  if  
 $\vec{\mathcal{I}}=\pi^{-1}(\mathcal{I})$.  
 We  write   
 \begin{eqnarray}\label{eq-25-10-25-2}
 \vec {\mathcal{M}}= \{\emptyset\}  \bigcup   \bigcup_{k\geq  1}  \vec {\mathcal{ M}}_k  
 \end{eqnarray}
 where  $\vec {\mathcal{ M}}_k $ 
 is  the   $k$-uniform  hyperdigraph  on  $S$   consisting  of  all  the  sequences 
 with     $k$-distinct  elements  in  $S$   and  write     
  \begin{eqnarray}\label{eq-25-10-25-1}
  \mathcal{M} = \{\emptyset\}  \bigcup   \bigcup_{k\geq  1}    \mathcal{ M} _k  
 \end{eqnarray}
 where $\mathcal{M}_k$   is  the  underlying  $k$-uniform  hypergraph 
  of  $ \vec {\mathcal{ M}}_k $.  
 Then     $\vec{\mathcal{M}} $  is  $\Sigma_\bullet$-invariant  iff  
  $\vec {\mathcal{ M}}_k $  is  $\Sigma_k$-invariant  for  each  $k\geq  1$.  
 By  (I2)  and    (I2)',  
 $\vec {\mathcal{ M}}-\{\emptyset\}=\bigcup_{k\geq  1}\vec{\mathcal{M}}_k$  is  a  
 directed  simplicial  complex   on  $S$  and 
 ${\mathcal{ M}}-\{\emptyset\}=\bigcup_{k\geq  1} \mathcal{M}_k$  is  the  underlying  simplicial  complex.

\begin{example}
\label{ex-25-10-88}
Let  $S$  be  any  finite  set  in  $\mathbb{F}^N$. 
\begin{enumerate}[(1)]
\item 
For  any     subset $\sigma\subseteq   S$,  we  define  $\sigma\in  \mathcal{I}$  
iff  the  vectors  $\{v\mid   v\in \sigma\}$  are  linearly  independent  in  $\mathbb{F}^N$. 
It  is  known that  $\mathcal{M}=(S,\mathcal{I})$  is  a  matroid,   called  the  {\it  vectorial  matroid}
  (cf. \cite[Section~1.3]{matroid}).  

\item
For  any   finite sequence  $\vec\sigma$  whose  elements  are  distinct  in   $S$,  
we  define  $\vec\sigma\in \vec{ \mathcal{I}}$  
iff  the  vectors  $\{v\in S  \mid  v  {\rm ~is~an~element~of~}\vec\sigma\}$  are  linearly  independent  in  $\mathbb{F}^N$. 
Then  $\vec  {\mathcal{M}}=(S,\vec{\mathcal{I}})$  is  a  directed  matroid. 
In  fact,  if  $\vec{\sigma}_1,\vec{\sigma}_2\in \vec{\mathcal{I}}$  with  
$|\vec\sigma_2|=|\vec\sigma_1|+1$,  then there  exist   $x\in \vec\sigma_2$     
 such  that  $x\notin \vec\sigma_1$ 
and  $\vec\sigma\in \vec{\mathcal{I}}$  is  obtained  by  adding  $x$  into  $\vec\sigma_1$  at  any  places  in  the  sequence.  
It  is  direct that  $\vec{\mathcal{M}}$  is  $\Sigma_\bullet$-invariant.  
We  call   $\vec{\mathcal{M}}$  the  {\it  vectorial}  directed  matroid.  
\item
Suppose  there  is     a  total  order   $\prec$  on  $\mathbb{F}$.  
For  any   finite sequence  $\vec\sigma$  whose  elements  are  distinct  in   $S$,  
we  define  $\vec\sigma\in \vec{ \mathcal{I}}'$  
iff 
\begin{eqnarray*}
\vec\sigma=(v_1,\ldots,v_k)=\begin{pmatrix}
v_1^1  &\cdots   &  v_k^1\\
 \vdots &    &  \vdots \\
v_1^N  &   \cdots  &  v_k^N
\end{pmatrix}
\end{eqnarray*}
has  rank  $k$  (i.e.  $v_1,\ldots,v_k\in  \mathbb{F}^N$  are  linearly  independent)  and  $v_1^1\preceq   v_2^1\preceq  \cdots\preceq  v_k^1$. 
Then  $\vec { \mathcal{M}}'=(S,\vec{\mathcal{I}}')$  is  a  directed  matroid. 
In  fact,  if  $\vec{\sigma}_1,\vec{\sigma}_2\in \vec{\mathcal{I}}'$  with  
$|\vec\sigma_2|=|\vec\sigma_1|+1$,  then there  exist   $x\in \vec\sigma_2$     
 such  that  $x\notin \vec\sigma_1$ 
and  $\vec\sigma\in \vec{\mathcal{I}}'$  is  obtained  by  adding  $x$  into  $\vec\sigma_1$  at  
certain  places  in  the  sequence  with  respect to the  total  order  $\prec$  of  the  first  coordinates
of  the  vectors  in  $\mathbb{F}^N$.  
If  the  first coordinates  of  the  elements  in  $S$  have  at  least  two  distinct  values, 
then  
  $\vec{\mathcal{M}}$  is  not  $\Sigma_\bullet$-invariant.  
\end{enumerate}
\end{example}

    Let  $\mathcal{M}=(S, \mathcal{I})$  and  $\mathcal{M}'=(S', \mathcal{I}')$  be  matroids.  
    We  call  a  map  $\varphi:  S\longrightarrow  S'$  
    a  {\it  simplicial  map}  and  denote  it  by  
    $\varphi: \mathcal{M}\longrightarrow \mathcal{M}'$  
     if  $\varphi$  induces  a  simplicial  map   between  simplicial  complexes  
     $\varphi:  \mathcal{M}\setminus \{\emptyset\}
     \longrightarrow  \mathcal{M}'\setminus \{\emptyset\}$.  
     Let  $\vec{\mathcal{M}}=(S, \vec{\mathcal{I}})$  and
       $\vec{\mathcal{M}}'=(S', \vec{\mathcal{I}}')$  be   
       directed  matroids.  
    We  call  a  map  $\vec\varphi:  S\longrightarrow  S'$  
    a  {\it  directed  simplicial  map}  and  denote  it  by  
    $\vec\varphi: \vec{\mathcal{M}}\longrightarrow \vec{\mathcal{M}}'$  
     if  $\vec\varphi$  induces  a  directed  simplicial  map  between  directed  simplicial  complexes     
    $\vec\varphi: \vec{\mathcal{M}}\setminus \{\emptyset\}\longrightarrow \vec{\mathcal{M}}'\setminus \{\emptyset\}$.  
 For  any  
 directed  simplicial  map  $\vec\varphi: \vec{\mathcal{M}}\longrightarrow \vec{\mathcal{M}}'$,  
 we  have  an  induced  simplicial  map  $\varphi: {\mathcal{M}}\longrightarrow {\mathcal{M}}'$  
 between  the  underlying  matroids  such  that  
 the  following  diagram  commutes 
 \begin{eqnarray}\label{eq-diag-251106-1}
 \xymatrix{
 \vec{\mathcal{M}}\ar[r]^-{\vec\varphi}\ar[d]_-{\pi}  & \vec{\mathcal{M}}'\ar[d]^-{\pi} \\
 {\mathcal{M}}\ar[r]^-{\varphi} &{\mathcal{M}}'. 
} 
 \end{eqnarray}

\begin{example}\label{ex-260108-2}
Let  $S$  be  any  finite  set  in  $\mathbb{F}^{N}$  and  
let  $S'$  be  any  finite  set  in  $\mathbb{F}^{N'}$.   
Suppose   $N\leq   N'$  and  $\varphi:  \mathbb{F}^{N}\longrightarrow \mathbb{F}^{N'}$  
and  an  injective   linear  map  such  that  $\varphi(S)\subseteq   S'$. 
\begin{enumerate}[(1)]
\item
Let  $\mathcal{M}=(S,\mathcal{I})$  and  $\mathcal{M}'=(S',\mathcal{I}')$ 
be  the  vectorial  matroids  in  Example~\ref{ex-25-10-88}~(1).  
Then  $\varphi$  induces  a  simplicial  map  
$\varphi:  \mathcal{M}\longrightarrow \mathcal{M}'$.  

\item
Let  $\vec{\mathcal{M}}=(S,\vec{\mathcal{I}})$
  and  $\vec{\mathcal{M}}'=(S',\vec{\mathcal{I}}')$ 
be  the  vectorial   directed  matroids  in  Example~\ref{ex-25-10-88}~(2).  
Then  $\varphi$  induces  a  directed   simplicial  map  
$\vec{\varphi}:  \vec{\mathcal{M}}\longrightarrow \vec{\mathcal{M}}'$ 
such  that  $\varphi$  in  (1)  and  
$\vec{\varphi}$  in  (2)  satisfy   the   commutative  diagram   (\ref{eq-diag-251106-1}).

\item
Let  $\vec{\mathcal{M}}=(S,\vec{\mathcal{I}})$
  and  $\vec{\mathcal{M}}'=(S',\vec{\mathcal{I}}')$ 
be  the   directed  matroids  in  Example~\ref{ex-25-10-88}~(3).  
Suppose  in  addition  that  
for  any  $(v^1,\ldots,v^N)$  and  $  (u^1,\ldots,  u^N)$  in   $ S$
 such  that  $v^1\preceq  u^1$, 
 their  images  
 $(v'^1,\ldots, v'^N)= \varphi(v^1,\ldots,v^N)$  and  
 $(u'^1,\ldots, u'^N)= \varphi(u^1,\ldots,u^N)$   in  $S'$   satisfy  
 $v'^1\preceq   u'^1$.  
Then  $\varphi$  induces  a  directed   simplicial  map  
$\vec{\varphi}:  \vec{\mathcal{M}}\longrightarrow \vec{\mathcal{M}}'$ 
such  that  $\varphi$  in  (1)  and  
$\vec{\varphi}$  in  (3)  satisfy   the   commutative  diagram   (\ref{eq-diag-251106-1}).

\item
Let  $\vec{\mathcal{M}}=(S,\vec{\mathcal{I}})$
  and  $\vec{\mathcal{M}}'=(S',\vec{\mathcal{I}}')$ 
be  the   directed  matroids  in  Example~\ref{ex-25-10-88}~(3).  
For   any  $(v_1,\ldots,v_k)\in \mathcal{H}$,  
define  $\vec\varphi(v_1,\ldots,v_k)= (\varphi(v_1),\ldots,\varphi(v_k))$ 
if  $\varphi(v_1)^1\preceq  \cdots \preceq \varphi(v_k)^1$ 
and  define  $\vec\varphi(v_1,\ldots,v_k)= \emptyset$
otherwise. 
Then  $\vec\varphi$  is  not   necessarily  a  directed  simplicial  map  since 
 $\emptyset\notin  \mathcal{M}\setminus\{\emptyset\}$.  
\end{enumerate}
\end{example}

 Let  $C_{k-1}(\vec{\mathcal{M}};R)$  be  the  free  $R$-module  generated  by  the  elements  of  
$\vec{\mathcal{M}}_k$   and  
let  $C_{k-1}(\mathcal{M};R)$  be  the  free  $R$-module  generated  by  the  elements  of  
$\mathcal{M}_k$,  for  any  $k\geq  1$.   
 We  have  a  chain  complex   
 \begin{eqnarray*}
 C_\bullet(\vec{\mathcal{M}};R) = \{C_{k}(\vec{\mathcal{M}};R)\mid  k\geq   0\}  
 \end{eqnarray*}
 whose  boundary  map 
 \begin{eqnarray*}
 \vec\partial_k:  C_{k}(\vec{\mathcal{M}};R)\longrightarrow  C_{k-1}(\vec{\mathcal{M}};R)
 \end{eqnarray*}
 is  given  by  (\ref{eq-25-10-3})  for  any  $k\geq  1$  and  by  
 $ \vec\partial_0(v)= 0$  for  any  $v\in  \mathcal{M}_0$.  
 We  also      have  a  chain  complex   
 \begin{eqnarray*}
 C_\bullet( \mathcal{M};R) = \{C_{k}( \mathcal{M};R)\mid  k\geq   0\}  
 \end{eqnarray*}
 whose  boundary  map 
 \begin{eqnarray*}
 \partial_k:  C_{k}(\mathcal{M};R)\longrightarrow  C_{k-1}(\mathcal{M};R)
 \end{eqnarray*}
 is  given  by  (\ref{eq-25-10-4})  for  any  $k\geq  1$  and  by  
 $ \partial_0(v)= 0$  for  any  $v\in  \mathcal{M}_0$.  
 
 \begin{example}
 \label{ex-260108-1}
 \begin{enumerate}[(1)]
 \item
 The  simplicial map   $\varphi$  in  Example~\ref{ex-260108-2}~(1)
 induces  a  chain  map  $\varphi_\#:  C_\bullet( \mathcal{M};R) \longrightarrow 
 C_\bullet( \mathcal{M}';R)$. 
  \item
   The  directed  simplicial  maps     $\vec\varphi$  in 
 Example~\ref{ex-260108-2}~(2)  and  (3)   induce  chain  maps
$\vec\varphi_\#:  C_\bullet( \vec{\mathcal{M}};R) \longrightarrow 
 C_\bullet( \vec{\mathcal{M}}';R)$.
 \item 
 The    map      $\vec\varphi$  in 
 Example~\ref{ex-260108-2}~(4)      induces  a    chain  map 
$\vec\varphi_\#:  C_\bullet( \vec{\mathcal{M}};R) \longrightarrow 
 C_\bullet( \vec{\mathcal{M}}';R)$.  
 \end{enumerate}
 \end{example}
 
 \begin{theorem}\label{co-25-10-27}
 The  canonical  projection  (\ref{eq-25-10-27-1}) 
  induces  
 a  surjective  chain  map  
  \begin{eqnarray}\label{eq-25-10-27-2}
 \pi_\#:   C_\bullet(\vec {\mathcal{ M}};R)\longrightarrow  C_\bullet(\mathcal{M};  R) 
  \end{eqnarray}  
 and   consequently    induces  a     homomorphism  of  homology   
\begin{eqnarray}\label{eq-25-10-27-3}
 \pi_*:   H_\bullet(\vec {\mathcal{M}};R)\longrightarrow  H_\bullet({\mathcal{M}};R).  
\end{eqnarray}
Moreover,    (\ref{eq-25-10-27-2})   and   (\ref{eq-25-10-27-3})  are   functorial  with  respect  to 
directed  simplicial  maps  between  directed  matroids   
 and  their  induced  simplicial  maps   between  the  underlying  matroids.  
 \end{theorem}
 
\begin{proof}
Let  $\vec{\mathcal{K}}=      \vec{\mathcal{ M}}  \setminus \{\emptyset\} $
 and  let  
$\mathcal{K}=     \mathcal{ M}  \setminus \{\emptyset\}  $  in   Proposition~\ref{le-25-10-5}.    
Then   (\ref{eq-25-10-27-2}) and  (\ref{eq-25-10-27-3})  follow  from  
 (\ref{eq-25-10-7})  and  (\ref{eq-25-10-8})  respectively.  
\end{proof}
  
  \begin{theorem}\label{co-251127-mat-1}
  For  any  sub-hyperdigraph  $\vec{\mathcal{H}}$  of  $ \vec{\mathcal{ M}}  \setminus \{\emptyset\}$,
  let  $\mathcal{H}$  be  the  underlying  hypergraph  of   $\vec{\mathcal{H}}$.  
  Then  we  have  a  commutative  diagram  of  chain  complexes  (\ref{diag-1116-2})  
  and  a  commutative  diagram  of  homology  groups
   (\ref{diag-1116-3})  such  that  they are  functorial 
   with respect  to  morphisms  of the sub-hyperdigraphs   induced  by the  directed  simplicial  maps  in
   Theorem~\ref{co-25-10-27}
   and  morphisms  of  sub-hypergraphs  induced  by the    simplicial  maps  in
   Theorem~\ref{co-25-10-27}.  
  \end{theorem}
  
  \begin{proof}
  The  diagrams  (\ref{diag-1116-2})   and    (\ref{diag-1116-3})  follow   from  
  Theorem~\ref{th-251127-hh1}.  
    Restricting the  directed  simplicial  map   in
   Theorem~\ref{co-25-10-27}  to  a  sub-hyperdigraph  $\vec{\mathcal{H}}$ of 
    $   \vec{\mathcal{ M}}  \setminus \{\emptyset\} $,
    we  obtain  a  morphism   
    $\vec\varphi:  \vec {\mathcal{H}} \longrightarrow  \vec{\mathcal{H}}'$ 
    of  hyperdigraphs.  
    The  morphism  $ \varphi:   {\mathcal{H}} \longrightarrow   {\mathcal{H}}'$  
    between  the  underlying  hypergraphs  is  the restriction  of  the  simplicial  map  in   
  Theorem~\ref{co-25-10-27}  to  the  sub-hypergraph  $\mathcal{H}$  of 
   ${\mathcal{ M}}  \setminus \{\emptyset\} $.    
   Therefore,  we  obtain  the  functoriality.  
  \end{proof}
  
  \subsection{The  Mayer-Vietoris  sequences}

Let   $\vec  {\mathcal{M}}$  and  $\vec{\mathcal{M}}'$  be  directed  matroids on  $S$
with  their  underlying  matroids  $\mathcal{M}$  and  $\mathcal{M}'$  respectively.  
Then  the  intersection  $ \vec  {\mathcal{M}}\cap  \vec{\mathcal{M}}'$
is  a  directed  matroid  on  $S$  with  its  underlying  matroid  $\mathcal{M}\cap\mathcal{M}'$.  
However,  the  union  $ \vec  {\mathcal{M}}\cup  \vec{\mathcal{M}}'$
may  not  be  a   directed  matroid  on  $S$  
and  the  union  $\mathcal{M}\cup\mathcal{M}'$
  may  not  be  a    matroid  on  $S$.  
  Nevertheless,  
  $ \vec  {\mathcal{M}}\cup  \vec{\mathcal{M}}'$  is  an   augmented    directed  
simplicial  complex  (i.e.  the  union  of  a  directed  simplicial  complex  and  
the  empty-set  $\emptyset$  assigned  with  dimension $-1$)
  and  $\mathcal{M}\cup\mathcal{M}'$  is  an  augmented  simplicial  complex
  (i.e.  the  union  of  a   simplicial  complex  and  
the  empty-set  $\emptyset$  assigned  with  dimension $-1$,  cf.  \cite[Section~5]{comalg}).  
  Thus  the  chain  complex   $ C_*(\vec{\mathcal{M}}\cup \vec{\mathcal{M}}';R)$
  as  well  as    the  chain  complex 
     $ C_*( \mathcal{M}\cup  \mathcal{M}';R)$  and   the  chain  map  
    \begin{eqnarray*} 
 \pi_\#:   C_*(\vec {\mathcal{ M}}\cup  \vec{\mathcal{M}}';R)\longrightarrow  C_*(\mathcal{M}\cup  \mathcal{M}';  R) 
  \end{eqnarray*}  
  are   defined  as   well.  
 
 \begin{theorem}\label{co-10-27-mv-1}
 For  any  directed  matroids   $\vec{\mathcal{M}}$  and  $\vec{\mathcal{M}}'$
 on  $S$  with  their  underlying  matroids 
 $ {\mathcal{M}}$  and  $ {\mathcal{M}}'$   respectively,  
 we  have  a  commutative   diagram  of  homology  groups  
 \begin{eqnarray}\label{diag-11-11-2}  
 \xymatrix{
\cdots\ar[r]  
&H_n (\vec{\mathcal{M}}\cap  \vec{\mathcal{M}}';R) \ar[r]\ar[d]_-{\pi_*}
&H_n(\vec{\mathcal{M}};R)\oplus  H_n(  \vec{\mathcal{M}}';R)\ar[r]\ar[d]_-{\pi_*}
 &\\
\cdots\ar[r]  
 &H_n ( {\mathcal{M}}\cap  {\mathcal{M}}';R) \ar[r] 
&H_n( {\mathcal{M}};R)\oplus  H_n(  {\mathcal{M}}';R)\ar[r] 
 &\\
\ar[r]&H_n(\vec{\mathcal{M}}\cup  \vec{\mathcal{M}}';R)\ar[r]\ar[d]_-{\pi_*}
&H_{n-1}(\vec{\mathcal{M}}\cap  \vec{\mathcal{M}}';R)\ar[r]\ar[d]_-{\pi_*}
&\cdots\\
\ar[r]&H_n( {\mathcal{M}}\cup   {\mathcal{M}}';R)\ar[r] 
&H_{n-1}( {\mathcal{M}}\cap   {\mathcal{M}}';R)\ar[r] 
&\cdots
} 
 \end{eqnarray}
 such that  the  two  rows  are  long  exact sequences.  
  Moreover,  the  diagram  is  natural  with  respect to  
 directed  simplicial  maps  between  directed  matroids   
 and  their  induced  simplicial  maps   between  the  underlying  matroids. 
 \end{theorem}

\begin{proof}
Let  $\vec{\mathcal{K}}=  \vec{\mathcal{ M}}  \setminus \{\emptyset\} $
 and  let  
$\vec{\mathcal{K}}'=   \vec{\mathcal{ M}}' \setminus\{\emptyset\}$ 
in  Proposition~\ref{pr-10-11-mv-1}.  
Their  underlying  simplicial  complexes  are 
 ${\mathcal{K}}=       {\mathcal{ M}}  \setminus \{\emptyset\} $
 and    
$\mathcal{K}'=     \mathcal{ M}'  \setminus\{\emptyset\} $
 respectively.   
The  proof  of  Theorem~\ref{co-10-27-mv-1} follows  from  Proposition~\ref{pr-10-11-mv-1}.  
\end{proof}

Similar  with   Subsection~\ref{ss-3.1},  
the  long exact  sequence    in  the first  row  of  (\ref{diag-11-11-2}) 
 is  denoted  by  $
 {\bf  \rm  MV}(\vec{\mathcal{M}}, \vec{\mathcal{M}}') 
$ 
and      the  long exact  sequence     in  the second  row  of   (\ref{diag-11-11-2})
is  denoted  by  $
 {\bf  \rm  MV}( \mathcal{M},  \mathcal{M}') 
$.  
The  diagram   (\ref{diag-11-11-2})  is  denoted  by  a  morphism  of  long  exact  sequences 
\begin{eqnarray}\label{eq-251213-mor1}
\pi_*:   {\bf  \rm  MV}(\vec{\mathcal{M}}, \vec{\mathcal{M}}') 
\longrightarrow   {\bf  \rm  MV}(  \mathcal{M},  \mathcal{M}') .  
\end{eqnarray}

\begin{corollary}\label{co-10-27-mv-2}
For  any  matroid  $\mathcal{M}$  on  $S$,  let  
$n=\min\{r(\mathcal{M}), r(\mathcal{M}^*)\}-1$.   Then  
\begin{eqnarray}
H_{n}(\mathcal{M} \cup \mathcal{M}^*;R)   \cong    H_{n}(\mathcal{M};R) \oplus  H_{n}( \mathcal{M}^*;R).  
\end{eqnarray} 
\end{corollary}

\begin{proof}
 Since  each  $\sigma\in  \mathcal{M}\cap \mathcal{M}^*$  is  not    
 a  base  of  $\mathcal{M}$  or    a   base  of  $\mathcal{M}^*$,  we  obtain  
 \begin{eqnarray*}
r( \mathcal{M}\cap \mathcal{M}^*)  <  \min\{r(\mathcal{M}),     r(\mathcal{M}^*)\}.  
 \end{eqnarray*}
 Thus  
 $H_n( \mathcal{M}\cap \mathcal{M}^*)=0$  
for  any  $n\geq   \min\{r(\mathcal{M}), r(\mathcal{M}^*)\}-1$.  
Apply  the  second  long  exact  sequence  in  Corollary~\ref{co-10-27-mv-1}
to  the  pair  of  matroids  $(\mathcal{M}, \mathcal{M}^*)$.  
We  obtain  (\ref{co-10-27-mv-2}).  
\end{proof}

\begin{theorem}
\label{co-251127-mat-mv-a1}
For  any  sub-hyperdigraphs  $\vec{\mathcal{H}}$  and  $\vec{\mathcal{H}}'$ 
of  $\bigcup_{k\geq  1} \vec{\mathcal{M}}_k$  and  $\bigcup_{k\geq  1} \vec{\mathcal{M}}'_k$ 
respectively,
let  $\mathcal{H}$  and  $\mathcal{H}'$   be  the  underlying  hypergraphs.  
Then  we  have  a  commutative  diagram  of  long  exact sequences  
(\ref{diag-251118-01}).  
Moreover,  
   (\ref{diag-251118-01}) is   functorial  in  the  sense  of  
   Theorem~\ref{co-251127-mat-1}  and  Theorem~\ref{co-10-27-mv-1}.  
Furthermore,   
 if  both   
  (I)  and  (II)  in  Subsection~\ref{ss-5.1}  are  satisfied,  
  then  we have  a  commutative  diagram  of  long  exact  sequences 
  (\ref{diag-251119-02}). 
\end{theorem}

\begin{proof}
The  proof  follows  from  Theorem~\ref{th-251118-mv1},
   Theorem~\ref{co-251127-mat-1}    and  Theorem~\ref{co-10-27-mv-1}. 
\end{proof}

\subsection{The  K\"unneth-type  formulae}

Let  $S$  and  $S'$  be  disjoint  finite  sets.  
Let  $\mathcal{M}=(S, \mathcal{I})$  and  $\mathcal{M}'=(S', \mathcal{I}')$  be  matroids  on 
$S$  and  $S'$  respectively.  
We  define  their  {\it  join}  
$\mathcal{M}*\mathcal{M}'$  to  be  a  matroid  on  $S\sqcup  S'$  whose  independent  sets  are 
 given  by   
\begin{eqnarray*}
\mathcal{I}* \mathcal{I}'= \{\sigma\sqcup\sigma'\mid  \sigma\in \mathcal{I}{\rm~and~}
\sigma'\in \mathcal{I}'\}.  
\end{eqnarray*}
It  follows  from  (I1) -  (I3)  directly  that  
$\mathcal{M}*\mathcal{M}'= (S\sqcup  S',  \mathcal{I}  * \mathcal{I}')$  is  a  matroid
 such  that  
 \begin{eqnarray*}
 r(\mathcal{M}*\mathcal{M}')=  r(\mathcal{M})+r(\mathcal{M}'). 
 \end{eqnarray*}    
Moreover,  if  $\mathbb{B}(\mathcal{M})$  be  the  set  of  bases  of  $\mathcal{M}$
and  $\mathbb{B}(\mathcal{M}')$  be  the  set  of  bases  of  $\mathcal{M}'$,  
then  
\begin{eqnarray}\label{eq-25-10-30-2}
 \mathbb{B}(\mathcal{M})  *   \mathbb{B}(\mathcal{M}')
 =\{\sigma\sqcup\sigma'\mid  \sigma\in \mathbb{B}(\mathcal{M})  {\rm~and~}
\sigma'\in \mathbb{B}(\mathcal{M}')  \}  
\end{eqnarray}
  is   the  set  of    bases    of  $\mathcal{M}*\mathcal{M}'$.  
  It  follows  from  (\ref{eq-25-10-30-1})  and  (\ref{eq-25-10-30-2})  that  
  \begin{eqnarray*}
  (\mathcal{M}  *\mathcal{M}')^*=  \mathcal{M}^*  *\mathcal{M}'^*. 
  \end{eqnarray*}
Let  $\vec{\mathcal{M}}=(S, \vec{\mathcal{I}})$  and  $\vec{\mathcal{M}}'=(S', \vec{\mathcal{I}}')$
  be  directed   matroids  on 
$S$  and  $S'$  respectively.  
We  define  their  {\it  join}  
$\vec{\mathcal{M}}*\vec{\mathcal{M}}'$  to  be  a  directed  matroid  on  $S\sqcup  S'$  whose  independent  sequences  are 
 given  by   
\begin{eqnarray*}
\vec{\mathcal{I}}* \vec{\mathcal{I}}'= \{\vec {\sigma} * \vec {\sigma}'\mid  \vec{\sigma}\in \vec{\mathcal{I}}{\rm~and~}
\vec{\sigma}'\in \vec{\mathcal{I}}'\}.  
\end{eqnarray*}
Here  if  $\vec\sigma =  v_1v_2\ldots  v_k$
and  $\vec\sigma'=  v'_1v'_2\ldots  v'_l$
where   $ v_1, \ldots  v_k $  are  distinct  in  $S$  and  
$ v'_1, \ldots  v'_l $  are  distinct  in  $S'$,  then  
$\vec {\sigma} * \vec {\sigma}'=   v_1v_2\ldots  v_k  v'_1v'_2\ldots  v'_l$.    
It  follows  from  (I1)' -  (I3)'  directly  that  
$\vec{\mathcal{M}}*\vec{\mathcal{M}}'= (S\sqcup  S',  \vec{\mathcal{I}}  * \vec{\mathcal{I}}')$  is  a  
directed  matroid.    
 Suppose 
 $\mathcal{M}$  and  $\mathcal{M}'$  are  the  underlying  matroids
 of    $\vec{\mathcal{M}}$   and   
 $\vec{\mathcal{M}}'$  respectively.  
 Let  $\pi:  \vec{\mathcal{M}}\longrightarrow  \mathcal{M} $  and  
  $\pi':  \vec{\mathcal{M}}'\longrightarrow  \mathcal{M}' $   be  the canonical  projections.  
  Then  we  have  an  induced  projection 
  \begin{eqnarray*}
  \pi*\pi':  \vec{\mathcal{M}} * \vec{\mathcal{M}}'\longrightarrow  \mathcal{M}* \mathcal{M}'
  \end{eqnarray*}
  sending  $\vec\sigma * \vec{\sigma}'$  to  $\pi(\sigma)*\pi'(\sigma')$  for  any  
  $\vec\sigma\in  \vec{\mathcal{M}}$  and  any    $\vec{\sigma}'\in  \vec{\mathcal{M}}'$.    
   The  diagram  commutes 
   \begin{eqnarray*}
   \xymatrix{
   (\vec{\mathcal{M}},  \vec{\mathcal{M}}')\ar[r]^-{*}\ar[d]_-{(\pi,\pi')}  &\vec{\mathcal{M}} * \vec{\mathcal{M}}' \ar[d]^-{\pi*\pi'}\\
    (  \mathcal{M}, \mathcal{M}')\ar[r]^-{*} &\mathcal{M}* \mathcal{M}'.  
   }
   \end{eqnarray*}

\begin{theorem}\label{co-10-27-kf-1}
For  any  directed  matroid  $\vec{\mathcal{M}}$  on  $S$   and
any  directed  matroid   $\vec{\mathcal{M}}'$
 on  $S'$
 with  their  underlying   matroids  
 $ {\mathcal{M}}$  and  $ {\mathcal{M}}'$   respectively,  
 we  have  a  commutative   diagram    
 \begin{eqnarray} \label{diag-251112-ku2}
\xymatrix{
0\ar[r]& \bigoplus_{p+q+1=n} H_{p+1}(\vec{\mathcal{M}};R)\otimes H_{q+1}(\vec{\mathcal{M}}';R)
\ar[r]\ar[d] _-{ \bigoplus_{k+l+1=n}\pi_*\otimes \pi'_*}
&H_{n+1}(\vec{\mathcal{M}} *\vec{\mathcal{M}}';R) \ar[r]\ar[d]_-{(\pi  * \pi')_*} & \\
0\ar[r]& \bigoplus_{p+q+1=n} H_{p+1}({\mathcal{M}};R)\otimes H_{q+1}({\mathcal{M}}';R)
\ar[r]
&H_{n+1}(\mathcal{M} *\mathcal{M}';R)\ar[r] &
}\\
~~~\nonumber\\
\xymatrix{
\ar[r] &\bigoplus_{p+q+1=n} {\rm  Tor}_R(H_{p+1}(\vec{\mathcal{M}};R), H_{q}(\vec{\mathcal{M}}';R))\ar[r]\ar[d] &0 \\
\ar[r]  &\bigoplus_{p+q+1=n} {\rm  Tor}_R(H_{p+1}(\mathcal{M};R), H_{q}(\mathcal{M}';R))\ar[r]  &0
}
\nonumber
 \end{eqnarray}
 such that  the  two  rows  are  short  exact sequences. 
 Moreover,  the  diagram  (\ref{diag-251112-ku2})  is  natural  with  respect to  
 directed  simplicial  maps  between  directed  matroids   
 and  their  induced  simplicial  maps   between  the  underlying  matroids.   
\end{theorem}

\begin{proof}
Let  $\vec{\mathcal{K}}=  \vec{\mathcal{ M}}  \setminus \{\emptyset\} $
 and  let  
$\vec{\mathcal{K}}'=   \vec{\mathcal{ M}}' \setminus\{\emptyset\}$ 
 in  Proposition~\ref{pr-10-12-kf-1}.  
 Then  $\vec{\mathcal{K}} *\vec{\mathcal{K}}'=
 \big( \vec{\mathcal{ M}} *  \vec{\mathcal{ M}}' \big)\setminus \{\emptyset\}$.
 Let   
 ${\mathcal{K}}=       {\mathcal{ M}}  \setminus \{\emptyset\} $
 and    
$\mathcal{K}'=     \mathcal{ M}'  \setminus\{\emptyset\} $
  be  the  underlying  simplicial  complexes  of 
  $\vec{\mathcal{K}} $  and  $\vec{\mathcal{K}}'$    respectively.
  Then  
$ {\mathcal{K}} * {\mathcal{K}}'=
 \big(  {\mathcal{ M}} *   {\mathcal{ M}}' \big)\setminus \{\emptyset\}$.      
The  proof  of  Theorem~\ref{co-10-27-kf-1}   follows  from  Proposition~\ref{pr-10-12-kf-1}.  
\end{proof}

 Similar  with   Subsection~\ref{ss-3.2},  
 the  short  exact  sequence    in  the first  row  of  (\ref{diag-251112-ku2}) 
 is  denoted  by  $
 {\bf  \rm  KU}(\vec {\mathcal{M}}, \vec{\mathcal{M}}') 
$ 
and      the  short  exact  sequence     in  the second  row  of  (\ref{diag-251112-ku2}) 
is  denoted  by  $
 {\bf  \rm  KU}(  \mathcal{M},  \mathcal{M}') 
$.  
The  diagram  (\ref{diag-251112-ku2})  is  denoted  by  a  morphism  of  short  exact  sequences 
\begin{eqnarray}\label{eq-251121-mor5}
\pi_*:   {\bf  \rm  KU}(\vec {\mathcal{M}}, \vec{\mathcal{M}}')
\longrightarrow   {\bf  \rm   KU}(  \mathcal{M},  \mathcal{M}').  
\end{eqnarray}

\begin{theorem}
\label{co-251127-mat-ku-a1}
For  any  sub-hyperdigraphs  $\vec{\mathcal{H}}$  and  $\vec{\mathcal{H}}'$ 
of  $\vec{\mathcal{ M}}  \setminus \{\emptyset\}$  
and  $\vec{\mathcal{ M}}'  \setminus \{\emptyset\}$ 
respectively,
let  $\mathcal{H}$  and  $\mathcal{H}'$   be  the  underlying  hypergraphs.  
Then  we  have  a  commutative  diagram  of  short    exact sequences  
(\ref{diag-251120-ku3})  
which  is   functorial  in  the  sense  of 
Theorem~\ref{co-251127-mat-1} and  Theorem~\ref{co-10-27-kf-1}.  
\end{theorem}

\begin{proof}
The  proof  follows  from  Theorem~\ref{th-251127-hku5},
Theorem~\ref{co-251127-mat-1}   and  Theorem~\ref{co-10-27-kf-1}. 
\end{proof}

  \section{Topological  obstructions  for  regular  maps  on  graphs}\label{s8}
  
  In  this  section,  we  give  some  topological  obstructions  for  
  the  existence  of  regular  embeddings  of  graphs  
  by  using  the  homology  of  the  (directed)  independence  complexes  
  as  well  as  the  embedded  homology  of  sub-hyper(di)graphs  of 
   the  (directed)  independence  complexes.   
   We  give  some  commutative  diagrams  of  the  Mayer-Vietoris  sequences 
   as  well  as  some  commutative  diagrams  of  the  short   exact  sequences  of  K\"unneth-type 
      as  obstructions  for  the  existence  of  regular  embeddings
   of  graphs.   
   We  prove 
   Theorem~\ref{th-251107-1}  (Main  Result  I),  
                    Theorem~\ref{th-mv-reg}  -  
                   Theorem~\ref{th-mv-reg-hg}  (Main  Result  II)   
                and      Theorem~\ref{th-251111-ku}
                -  Theorem~\ref{th-251111-ku-hg}  (Main  Result  III).

  Let   $S$  be  a  fixed  finite  subset  of  $\mathbb{F}^N$.   
Let   $\vec{\mathcal{M}}=(S,\vec{\mathcal{I}})$  be  the  vectorial  directed  matroid  given  in 
Example~\ref{ex-25-10-88}~(2).  
Let  $\mathcal{M}=(S, \mathcal{I})$  be  the  underlying  matroid  of   $\vec{\mathcal{M}}$,  
i.e.,  the  vectorial     matroid  given  in 
Example~\ref{ex-25-10-88}~(1).
  Let  $G=(V,E)$  be  any  graph.  

\begin{proposition}\label{pr-251107-1}
There  is  a  $k$-regular  map  
$f:  G\longrightarrow  \mathbb{F}^N$  such  that  $f(V)\subseteq  S$  if  and  only  if  
there  is   a  commutative  diagram  
\begin{eqnarray}\label{diag-251107-a1}
\xymatrix{
{\rm  sk}^{k-1}(\overrightarrow    {\rm  Ind}(G)) \ar[rr]^-{\overrightarrow    {\rm  Ind}( f)}  \ar[d]_-{\pi}
&&
 \vec{\mathcal{M}}\ar[d]^-{\pi}\\
 {\rm  sk}^{k-1}({\rm  Ind}(G)) \ar[rr]^-{{\rm  Ind}(f)}    &&\mathcal{M}
}
\end{eqnarray}
such  that  $\overrightarrow {\rm  Ind}( f)$  is   
a  directed  simplicial  map    
sending  a  directed  simplex 
 $(v_1,\ldots,v_l)\in  {\rm  sk}^{k-1}(\overrightarrow    {\rm  Ind}(G))$,  
 where  $l\leq  k$,     to  an  independent   sequence  
 $(f(v_1),\ldots,f(v_l))\in \vec{\mathcal{I}}$
and    ${\rm  Ind}( f)$   is  a  simplicial  map     
 sending    a   simplex 
 $\{v_1,\ldots,v_l\}\in  {\rm  sk}^{k-1}({\rm  Ind}(G))$,  
 where  $l\leq  k$,   to  an  independent   set     
 $\{f(v_1),\ldots,f(v_l)\}\in \mathcal{I}$. 
 Consequently,  
 there  is  a  $k$-regular  map  
$f:  G\longrightarrow  \mathbb{F}^N$  such  that  $f(V)\subseteq  S$  if  and  only  if  
there  is  a  commutative  diagram  
\begin{eqnarray}\label{diag-251121-h1}
\xymatrix{
\vec{\mathcal{H}} \ar[rr]^-{\overrightarrow    {\rm  Ind}( f)\mid_{\vec{\mathcal{H}}}}  \ar[d]_-{\pi}
&&
 \vec{\mathcal{M}}\ar[d]^-{\pi}\\
\mathcal{H} \ar[rr]^-{{\rm  Ind}(f)\mid_{\mathcal{H}}}    &&\mathcal{M}
}
\end{eqnarray}
for  any  hyperdigraph  $\vec{\mathcal{H}}\subseteq  {\rm  sk}^{k-1}(\overrightarrow    {\rm  Ind}(G))$
with  its  underlying  hypergraph  $\mathcal{H}\subseteq  {\rm  sk}^{k-1}(   {\rm  Ind}(G))$.  
\end{proposition}

\begin{proof}
($\Longrightarrow$):  
Suppose  $f:  G\longrightarrow  \mathbb{F}^N$  is  
a  $k$-regular  map  such  that  $f(V)\subseteq  S$. 
Let  $1\leq  l\leq  k$.  
Then  for  any  distinct  vertices  $v_1,\ldots,  v_l$  such  that  
$\{v_i,v_j\}\notin  E$  for  any  $1\leq  i<j\leq  l$,  
their  images $f(v_1),\ldots,f(v_l) $  are  linearly  independent  in  $\mathbb{F}^N$.   
Consequently,  $(f(v_1),\ldots,f(v_l))\in \vec{\mathcal{I}}$  
and  $\{f(v_1),\ldots,f(v_l)\}\in \mathcal{I}$.  
 Hence     $\overrightarrow {\rm  Ind}( f)$  is  a  directed  simplicial  map 
   and    ${\rm  Ind}( f)$   is   a   simplicial  map.  
      Let   $(v_1,\ldots,v_l)\in  {\rm  sk}^{k-1}(\overrightarrow    {\rm  Ind}(G))$.   
   Then  
   \begin{eqnarray*}
   \pi\circ  \overrightarrow {\rm  Ind}( f)(v_1,\ldots,v_l)
   &=& \pi (f(v_1),\ldots,f(v_l))\\
   &=&\{f(v_1),\ldots,f(v_l)\}\\
   &=&{\rm  Ind}( f)(\{v_1,\ldots,v_l\})\\
   &=& {\rm  Ind}( f)\circ \pi \{v_1,\ldots,v_l\}.  
   \end{eqnarray*}
   Thus  $\pi\circ  \overrightarrow {\rm  Ind}( f)= {\rm  Ind}( f)\circ \pi$,    
    i.e.   the  diagram  (\ref{diag-251107-a1})   commutes.   
    Let  $\vec{\mathcal{H}}$  be  a  sub-hyperdigraph  of 
     ${\rm  sk}^{k-1}(\overrightarrow    {\rm  Ind}(G))$.  
     Then  its  underlying  hypergraph  $\mathcal{H}$  is  a  sub-hypergraph of  
      ${\rm  sk}^{k-1}(   {\rm  Ind}(G))$.  
      The  commutativity  of  (\ref{diag-251107-a1})  implies  
      the  commutativity  of  (\ref{diag-251121-h1}).  
      The  converse  is  straightforward  since  (\ref{diag-251121-h1})  implies  
      that  $f$  is  $k$-regular.  
      
      ($\Longleftarrow$):  Suppose  there  is  a  commutative  diagram  (\ref{diag-251107-a1}).  
      Then  the  simplicial  map  ${\rm  Ind}(f)$  is  given  by  a  map  $f:  V\longrightarrow  S$.  
      Similar  to  the  above   argument,  
      $f$  is  a  $k$-regular  embedding  of  $G$  into  $\mathbb{F}^N$
      such  that  $f(V)\subseteq  S$. 
\end{proof}

\begin{corollary}
\label{co-251107-1}
There  is  a   $G$-regular  map  
$f:  G\longrightarrow  \mathbb{F}^N$  such  that  $f(V)\subseteq  S$ 
if  and  only  if 
there  is   a  commutative  diagram  
\begin{eqnarray}\label{diag-251107-a2}
\xymatrix{
 \overrightarrow    {\rm  Ind}(G)  \ar[rr]^-{\overrightarrow    {\rm  Ind}( f)}  \ar[d]_-{\pi}
&&
 \vec{\mathcal{M}}\ar[d]^-{\pi}\\
 {\rm  Ind}(G)  \ar[rr]^-{{\rm  Ind}(f)}    &&\mathcal{M}
}
\end{eqnarray}
such  that  $\overrightarrow {\rm  Ind}( f)$  is   
a  directed  simplicial  map    
sending  a  directed  simplex 
 $(v_1,\ldots,v_l)\in   \overrightarrow    {\rm  Ind}(G)$
        to  an  independent   sequence  
 $(f(v_1),\ldots,f(v_l))\in \vec{\mathcal{I}}$
and    ${\rm  Ind}( f)$   is  a  simplicial  map     
 sending    a   simplex 
 $\{v_1,\ldots,v_l\}\in  {\rm  Ind}(G) $   to  an  independent   set     
 $\{f(v_1),\ldots,f(v_l)\}\in \mathcal{I}$  for  any  $l\geq  1$.  
Consequently, 
there  is  a   $G$-regular  map  
$f:  G\longrightarrow  \mathbb{F}^N$  such  that  $f(V)\subseteq  S$ 
if  and  only  if  there  is   a  commutative  diagram  
\begin{eqnarray}\label{diag-251121-h2}
\xymatrix{
\vec{\mathcal{H}} \ar[rr]^-{\overrightarrow    {\rm  Ind}( f)\mid_{\vec{\mathcal{H}}}}  \ar[d]_-{\pi}
&&
 \vec{\mathcal{M}}\ar[d]^-{\pi}\\
\mathcal{H} \ar[rr]^-{{\rm  Ind}(f)\mid_{\mathcal{H}}}    &&\mathcal{M}
}
\end{eqnarray}
for  any  hyperdigraph  $\vec{\mathcal{H}}\subseteq   \overrightarrow    {\rm  Ind}(G)$
with  its  underlying  hypergraph  $\mathcal{H}\subseteq  {\rm  sk}^{k-1}(   {\rm  Ind}(G))$.  
\end{corollary}

\begin{proof}
A        map  
$f:  G\longrightarrow  \mathbb{F}^N$  is  $G$-regular if  and  only  if  $f$  
is     $k$-regular  for  any  $k\geq  1$.  
On  the  other  hand,  
there  is  a  commutative  diagram 
   (\ref{diag-251107-a1})    for  any  $k\geq  1$ 
   if  and  only  if  
    there  is  a   commutative  diagram    (\ref{diag-251107-a2}); 
    and  
there  is a  commutative  diagram   (\ref{diag-251121-h1})    
for  any  hyperdigraph  $\vec{\mathcal{H}}\subseteq   
 {\rm  sk}^{k-1}(\overrightarrow    {\rm  Ind}(G))$  and  any  $k\geq  1$
if  and  only  if  
 there  is  a  commutative  diagram  (\ref{diag-251121-h2})
 for  any  hyperdigraph  $\vec{\mathcal{H}}\subseteq   \overrightarrow    {\rm  Ind}(G)$. 
 Therefore,  the  proof  follows  from 
   Proposition~\ref{pr-251107-1}.   
\end{proof}

\begin{theorem}\label{th-251107-1}
If   there  is  a  $k$-regular  map  
$f:  G\longrightarrow  \mathbb{F}^N$  such  that  $f(V)\subseteq  S$,   then 
we  have  a  commutative  diagram  of  homology groups 
\begin{eqnarray}\label{diag-251107-b1}
\xymatrix{
H_\bullet({\rm  sk}^{k-1}(\overrightarrow    {\rm  Ind}(G));R) 
\ar[rr]^-{\overrightarrow    {\rm  Ind}( f)_*}  \ar[d]_-{\pi_*}
&&
 H_\bullet(\vec{\mathcal{M}};R)\ar[d]^-{\pi_*}\\
 H_\bullet({\rm  sk}^{k-1}({\rm  Ind}(G));R) 
 \ar[rr]^-{{\rm  Ind}(f)_*}   
  &&H_\bullet(\mathcal{M};R)   
}
\end{eqnarray}
and  consequently  a  commutative  diagram  of  homology  groups 
\begin{eqnarray}\label{diag-251121-h3}
\xymatrix{
H_\bullet(\vec{\mathcal{H}};R) \ar[rr]^-{(\overrightarrow    {\rm  Ind}( f)\mid_{\vec{\mathcal{H}}})_*}  \ar[d]_-{\pi_*}
&&
 H_\bullet(\vec{\mathcal{M}};R)\ar[d]^-{\pi_*}\\
H_\bullet(\mathcal{H};R) \ar[rr]^-{({\rm  Ind}(f)\mid_{\mathcal{H}})_*}    &&
H_\bullet(\mathcal{M};R)
}
\end{eqnarray}
for  any  hyperdigraph  $\vec{\mathcal{H}}\subseteq  {\rm  sk}^{k-1}(\overrightarrow    {\rm  Ind}(G))$.  
Moreover,  
the  diagrams  (\ref{diag-251107-b1})   and  (\ref{diag-251121-h3})   are   
 functorial  with  respect  to 
 directed  simplicial  maps  between  directed  independence  complexes   
 and  their  induced  simplicial  maps   between  the  underlying  independence  complexes, 
 both  of  which  are 
 induced  by  filtrations  of  the  vertices  of  $G$.  
\end{theorem}

\begin{proof}
Apply  the  homology  functor to  the  diagram (\ref{diag-251107-a1})
and  apply  the  embedded  homology  functor to  the  diagram (\ref{diag-251121-h1}).  
With  the  help  of   the  homomorphisms  of  homology groups 
  (\ref{eq-25-10-8})  and   (\ref{eq-25-10-27-3}),  
 we  obtain  the  commutative  diagram  (\ref{diag-251107-b1}).  
 With  the  help  of   the  homomorphisms  of  homology groups 
  (\ref{diag-1116-3})  and   (\ref{eq-25-10-27-3}),  
 we  obtain  the  commutative  diagram   (\ref{diag-251121-h3}). 
 Moreover,   the  functoriality  of  the  diagram  (\ref{diag-251107-b1})   
 follows  from  
 the  functoriality  of   (\ref{eq-25-10-8})     in  Theorem~\ref{th-25-10-15-1}
 and  the  functoriality  of   (\ref{eq-25-10-27-3})  in  Theorem~\ref{co-25-10-27};
 and   the  functoriality  of  the  diagram  (\ref{diag-251121-h1})   
 follows  from  
 the  functoriality  of     (\ref{diag-1116-3})    in  Theorem~\ref{th-25-11-26-1}   
 and  the  functoriality  of   (\ref{eq-25-10-27-3}).  
\end{proof}

\begin{corollary}\label{co-251107-2}
If   there  is  a  $G$-regular  map  
$f:  G\longrightarrow  \mathbb{F}^N$  such  that  $f(V)\subseteq  S$,   then 
we  have  a  commutative  diagram  of  homology groups 
\begin{eqnarray}\label{diag-251107-b2}
\xymatrix{
H_\bullet(\overrightarrow    {\rm  Ind}(G);R) 
\ar[rr]^-{\overrightarrow    {\rm  Ind}( f)_*}  \ar[d]_-{\pi_*}
&&
 H_\bullet(\vec{\mathcal{M}};R)\ar[d]^-{\pi_*}\\
 H_\bullet({\rm  Ind}(G);R) 
 \ar[rr]^-{{\rm  Ind}(f)_*}   
  &&H_\bullet(\mathcal{M};R)  
}
\end{eqnarray}
and  consequently  a  commutative  diagram  of  homology  groups  
\begin{eqnarray}\label{diag-251121-h5}
\xymatrix{
H_\bullet(\vec{\mathcal{H}};R) \ar[rr]^-
{(\overrightarrow    {\rm  Ind}( f)\mid_{\vec{\mathcal{H}}})_*}  \ar[d]_-{\pi_*}
&&
H_\bullet( \vec{\mathcal{M}};R)\ar[d]^-{\pi_*}\\
H_\bullet(\mathcal{H};R) \ar[rr]^-{({\rm  Ind}(f)\mid_{\mathcal{H}})_*}   
 &&
H_\bullet(\mathcal{M};R)
}
\end{eqnarray}
for  any  hyperdigraph  $\vec{\mathcal{H}}\subseteq   \overrightarrow    {\rm  Ind}(G)$.  
Moreover,  
the  diagram  (\ref{diag-251107-b2})   is  
 functorial  with  respect  to 
 directed  simplicial  maps  between  directed  independence  complexes   
 and  their  induced  simplicial  maps   between  the  underlying  independence  complexes,
 both  of  which  are 
 induced  by  filtrations  of  the  vertices  of  $G$.  
\end{corollary}

\begin{proof}
The  proof  follows  from  Theorem~\ref{th-251107-1}  by  letting  $k$  run  over  all  positive  integers.  
\end{proof}

Let  $V_k(\mathbb{F}^N)$  be  the  Stiefel  manifold   consisting  of  
all  the  $k$-frames  
(a  $k$-frame  is an  ordered  $k$-tuple,  i.e.  a  sequence  of  $k$-elements,
of  linearly  independent  vectors)
 in   $\mathbb{F}^N$.  
 Let  $Gr_k(\mathbb{F}^N) $  be  the  Grassmannian  consisting  of  
all  the  $k$-dimensional  subspaces  
in  $\mathbb{F}^N$.  
We  have  a  canonical  projection  
\begin{eqnarray*}
p_{k,N}:  V_k(\mathbb{F}^N)\longrightarrow   V_k(\mathbb{F}^N)/\Sigma_k
\longrightarrow  Gr_k(\mathbb{F}^N)  
\end{eqnarray*} 
 sending  a  $k$-frame  in  $\mathbb{F}^N$  to  the  $k$-dimensional  
  subspace  of  $\mathbb{F}^N$  spanned  by  the  $k$-frame.  
  Let  $V_k(\mathbb{F}^\infty)$   be  the  colimit  of  $V_k(\mathbb{F}^N)$  
  and  let  $Gr_k(\mathbb{F}^\infty)$   be  the  colimit  of  $Gr_k(\mathbb{F}^N)$ 
  as $N$  goes  to  infinity, 
  where  $\mathbb{F}^N$  is  regarded as  the  subspace  of  $\mathbb{F}^\infty$  
  with     the  $k$-th  coordinate   zero  for  each  $k\geq  N+1$.  
  We  have   a  canonical  projection  
\begin{eqnarray*}
p_{k,\infty}:  V_k(\mathbb{F}^\infty)\longrightarrow   V_k(\mathbb{F}^\infty)/\Sigma_k
\longrightarrow  Gr_k(\mathbb{F}^\infty).       
\end{eqnarray*} 
By  (\ref{eq-25-10-25-2}) and   (\ref{eq-25-10-25-1}), 
we  have a   commutative  diagram 
\begin{eqnarray}\label{diag-251112-a78}
\xymatrix{
\vec{\mathcal{M}}_k  \ar[r]  \ar[d] &  V_k(\mathbb{F}^N) \ar[d]  
\ar[r]  &    V_k(\mathbb{F}^\infty)  \ar[d]  \\
 \mathcal{M}_k  \ar[r]   &  V_k(\mathbb{F}^N) /\Sigma_k \ar[d] \ar[r]  & V_k(\mathbb{F}^\infty)/\Sigma_k\ar[d] \\
 & Gr_k(\mathbb{F}^N)  \ar[r] &    Gr_k(\mathbb{F}^\infty) 
}
\end{eqnarray}
where  all  the  horizontal  maps  are  canonical  inclusions   
and  all  the   vertical  maps  are  canonical  projections.  
The  next 
corollary   is  a   discrete   analog  of  \cite[Proposition~2.1]{cohen1},  
\cite[Lemma~2.10]{high1}, 
\cite[Lemma~5.7]{high2}   and  \cite[Proposition~4.1]{reg-2018}.

\begin{corollary}\label{pr-251112-29}
If   there  is  a  $k$-regular  map  
$f:  G\longrightarrow  \mathbb{F}^N$  such  that  $f(V)\subseteq  S$,   then 
we  have  a  commutative  diagram  
\begin{eqnarray}\label{diag-251112-a5}
\xymatrix{
{\rm  Conf}_k(G)   \ar[rr]^-{   {\rm  Conf}( f)}  \ar[d]_-{\pi}
 &&
 \vec{\mathcal{M}}_k\ar[d]^-{\pi}\\
 {\rm  Conf}_k(G)/\Sigma_k \ar[rr]^-{{\rm  Conf}(f)/\Sigma_k}     
 &&
 \mathcal{M}_k  
}
\end{eqnarray}
and  consequently  a  commutative  diagram  
\begin{eqnarray}\label{diag-251121-h7}
\xymatrix{
\vec{\mathcal{H}}_k   \ar[rr]^-{   {\rm  Conf}( f)\mid_{\vec{\mathcal{H}}_k}}  \ar[d]_-{\pi}
 &&
 \vec{\mathcal{M}}_k\ar[d]^-{\pi}\\
\mathcal{H}_k \ar[rr]^-{({\rm  Conf}(f)/\Sigma_k)\mid_{\mathcal{H}_k}}     
 &&
 \mathcal{M}_k  
}
\end{eqnarray}
for  any  $k$-uniform  hyperdigraph  $\vec{\mathcal{H}}_k\subseteq {\rm  Conf}_k(G) $.  
\end{corollary}

\begin{proof}
Taking  the  collection   of  the  directed  $(k-1)$-simplices  in  $\overrightarrow{\rm  Ind}(G)$
   and  the  collection  of   the     $(k-1)$-simplices  in   ${\rm  Ind}(G)$,  
from   the  commutative  diagram   (\ref{diag-251107-a2}),  
 we  obtain  the  commutative  diagram  (\ref{diag-251112-a5})     from  
(\ref{diag-251107-a2}).  
Taking  a  $k$-uniform  sub-hyperdigraph $\vec{\mathcal{H}}_k$  of  $ {\rm  Conf}_k(G) $ 
 whose  underlying  $k$-uniform  hypergraph  is  ${\mathcal{H}}_k$,  
 we  obtain   the  commutative  diagram  (\ref{diag-251121-h7})  from  (\ref{diag-251112-a5}).  
\end{proof}

 \begin{corollary}\label{co-251121-b1}
 If  $\mathbb{F}=\mathbb{R}$  or  $\mathbb{C}$,  
then     (\ref{diag-251112-a78})  and  (\ref{diag-251112-a5})  
give  the  classifying  map  of  the  associated  (real  or  complex)  vector  bundle  of  the  covering  map 
(\ref{eq-canonical-covering})  as  
the  composition
\begin{eqnarray}\label{eq-251112-pb}
 {\rm  Conf}_k(G)/\Sigma_k\longrightarrow   \mathcal{M}_k\longrightarrow
 V_k(\mathbb{F}^N) /\Sigma_k    \longrightarrow  Gr_k(\mathbb{F}^N) 
 \longrightarrow
 Gr_k(\mathbb{F}^\infty); 
 \end{eqnarray}  
 and    (\ref{diag-251112-a78})  and    (\ref{diag-251121-h7})
give  the  classifying  map  of  the  associated  (real  or  complex)  vector  bundle  of  the  covering  map 
(\ref{eq-25-9.21.1})  as  
the  composition
\begin{eqnarray}\label{eq-251112-pb-h}
\mathcal{H}_k\longrightarrow  {\rm  Conf}_k(G)/\Sigma_k\longrightarrow   \mathcal{M}_k\longrightarrow
 V_k(\mathbb{F}^N) /\Sigma_k    \longrightarrow  Gr_k(\mathbb{F}^N) 
 \longrightarrow
 Gr_k(\mathbb{F}^\infty). 
 \end{eqnarray} 
   \end{corollary}
 
 \begin{proof}
Let $\mathbb{F}=\mathbb{R}$  or  $\mathbb{C}$.  
Note  that  each  square  in  
(\ref{diag-251112-a78}),    (\ref{diag-251112-a5})  and  (\ref{diag-251121-h7})  
is  a  pull-back.  
Hence  (\ref{eq-251112-pb})  is  the  classifying  map  of  the  associated  vector  bundle 
\begin{eqnarray*}
\mathbb{F}^k\longrightarrow  {\rm  Conf}_k(G)\times_{\Sigma_k} \mathbb{F}^k
\longrightarrow   {\rm  Conf}_k(G)/\Sigma_k  
\end{eqnarray*}
  of  the  covering  map 
(\ref{eq-canonical-covering});  
and   (\ref{eq-251112-pb-h})  is  the  classifying  map  of  the  associated  vector  bundle 
\begin{eqnarray*}
\mathbb{F}^k\longrightarrow  \vec{\mathcal{H}}_k\times_{\Sigma_k} \mathbb{F}^k
\longrightarrow   \mathcal{H}_k 
\end{eqnarray*}
  of  the  covering  map 
(\ref{eq-25-9.21.1}). 
 \end{proof}
 
\begin{remark}
Both  the  geometric realizations   $| {\rm  Conf}_k(G)/\Sigma_k|$     and   $|\mathcal{H}_k|$   
are  open  subsets  of  the  geometric  simplicial  complex  $|{\rm  Ind}(G)|$.  
Thus  the  vector  bundles  in  Corollary~\ref{co-251121-b1}  are  well-defined.    
\end{remark}

 \subsection{Regular  maps  and   the  Mayer-Vietoris  sequences }

Suppose   $f': G'\longrightarrow  \mathbb{F}^N$,   
 $f'': G''\longrightarrow  \mathbb{F}^N$  and  
 $f''': G'''\longrightarrow  \mathbb{F}^N$  are   $k$-regular  maps 
  on  the  graphs   $G'$,  $G''$  and  $G'''$  respectively
  such  that  
  \begin{eqnarray}\label{eq-251111-s1}
  f'(V'),  f''(V''),  f'''(V''')\subseteq  S,
  \end{eqnarray}  
  where  $V'$,  $V''$  and  $V'''$  
  are  the  sets  of  vertices  of  
  $G'$,   $G''$  and  $G'''$  respectively  and  are  mutually  disjoint.   
  Let  $G$   be  the   reduced  join  of  $G'$,  $G''$  and  $G'''$  given  by  
  (\ref{eq-10-19-31}).  
  Then  we  have  an   induced  $k$-regular  map   
  \begin{eqnarray*}
  f'\tilde * f''\tilde  *f''':  G\longrightarrow  \mathbb{F}^N.     
  \end{eqnarray*}
  The  set   of  vertices  of  $G$ is   
  $V'\sqcup  V''\sqcup  V'''$   such  that    
    \begin{eqnarray*}
    (f'\tilde * f''\tilde  *f''') ( V'\sqcup  V''\sqcup  V''') = f'(V')\cup  f''(V'')\cup  f'''(V''')\subseteq  S.  
   \end{eqnarray*}
 On  the  other  hand,    let  $\vec{ \mathcal{M}}'$    
  be  the directed  matroid   of  all  the  independent  sequences   of  the  vectors 
   in  $ f'(V') \cup  f'''(V''')$  and   let  $\vec{ \mathcal{M}}''$   be  the directed  matroid   of  
  all  the  independent  sequences   of  the  vectors in  $ f'(V'') \cup  f'''(V''')$.  
  Let    $ \mathcal{M}'$  and   $ \mathcal{M}''$  be  the  underlying  matroids  of 
 $\vec{ \mathcal{M}}'$  and  $\vec{ \mathcal{M}}''$   respectively.    
   
 \begin{theorem}\label{th-mv-reg}
 For  any  $k$-regular  maps  
 $f': G'\longrightarrow  \mathbb{F}^N$,   
 $f'': G''\longrightarrow  \mathbb{F}^N$  and  
 $f''': G'''\longrightarrow  \mathbb{F}^N$ satisfying  (\ref{eq-251111-s1}),    
 we  have  a  commutative  diagram  (\ref{diag-251230-5})  of  long  exact  sequences 
where  the  horizontal  maps  are  induced  by  $f', f'',  f'''$.  
 Moreover,  all  these   homomorphisms  of  homology  groups  are  
 functorial  in the  sense  of  
   Theorem~\ref{th-251107-1}.  
 \end{theorem}
 
 \begin{proof}
 By 
Corollary~\ref{pr-251125-mv1},   
 Theorem~\ref{co-10-27-mv-1}  and   
  Theorem~\ref{th-251107-1},  
 the  homomorphisms  of  homology  groups   in  (\ref{diag-251107-b1})  
 induce   the   homomorphisms  of  
 the  long  exact  sequences  of   
 homology  groups  in   (\ref{diag-251230-5}).  
 The  commutativity  and  the  functoriality  of   the  diagram  (\ref{diag-251107-b1})   imply  
 the  commutativity  and  the  functoriality   of  the  diagram   (\ref{diag-251230-5})
 respectively.  
 \end{proof}
 
  \begin{theorem}\label{th-mv-reg-hg}
Let 
 $f': G'\longrightarrow  \mathbb{F}^N$,   
 $f'': G''\longrightarrow  \mathbb{F}^N$  and  
 $f''': G'''\longrightarrow  \mathbb{F}^N$ 
 be  $k$-regular  maps  satisfying  (\ref{eq-251111-s1}). 
 Then   for  any  sub-hyperdigraph    $\vec{\mathcal{H}}'$  of  
             ${\rm  sk}^{k-1}(\overrightarrow{\rm  Ind}(G'  \tilde{*}  G'''))$        
           and  any   sub-hyperdigraph    $\vec{\mathcal{H}}''$  of  
             $ {\rm  sk}^{k-1}(\overrightarrow{\rm  Ind}(G''  \tilde{*}  G'''))$
             with  their  underlying  hypergraphs  $\mathcal{H}'$  and  $\mathcal{H}''$  respectively 
             such  that  (I)   and  (II)  are  satisfied  for  the  pair  of  hyperdigraphs  
             $(\vec{\mathcal{H}}', \vec{\mathcal{H}}'')$
    as  well  as  the  pair  of  hypergraphs  $(\mathcal{H}', \mathcal{H}'')$,  
 we  have  a  commutative  diagram (\ref{diag-251230-3})  of  long  exact  sequences 
where  the  horizontal  maps  are  induced  by  $f', f'',  f'''$.   
 \end{theorem}

  \begin{proof}
  By  Theorem~\ref{th-251118-mv1},  Theorem~\ref{th-251107-1}  and  Theorem~\ref{th-mv-reg},    
        we  have  the   commutative  diagram (\ref{diag-251230-3}).             
  \end{proof}
 
  \subsection{Regular  maps  and  the    K\"unneth-type  formulae }

Let  $V$  and  $V'$  be  disjoint  sets  of  vertices.  
Let  $G$  be a graph  on  $V$ and  let  $G'$  be  a  graph  on  $V'$.  
As  a  discrete  analog  of   \cite[Definition~1.6]{reg-2018}, 
we say  a  map  
$ g:  G\sqcup  G'\longrightarrow \mathbb{F}^M$  
is   {\it  $(G, k; G', k')$-regular}  if   for  any  distinct  $k$-vertices  
$v_1,\ldots,v_k\in  V$    that     are  mutually  non-adjacent in  $G$ 
and   any  distinct  $k'$-vertices  
$v'_1,\ldots,v'_{k'}\in  V'$    that     are  mutually  non-adjacent in  $G'$,  
their  images 
$g(v_1)$, $\ldots$, $g(v_k)$,   
$g(v'_1)$, $\ldots$, $g(v'_{k'})$  
are  linearly  independent  in  $\mathbb{F}^M$.

\begin{lemma}\label{pr-251111-9}
There  is  a   $(G, k; G', k')$-regular  map  
$g:  G\sqcup  G'\longrightarrow \mathbb{F}^M$  such  that  $g(V\sqcup  V')\subseteq  S$  
if  and  only  if   
we  have  a  commutative  diagram  
\begin{eqnarray}\label{diag-251111-a1}
\xymatrix{
{\rm  sk}^{k-1}(\overrightarrow    {\rm  Ind}(G))* {\rm  sk}^{k'-1}(\overrightarrow    {\rm  Ind}(G'))
\ar[rr]^-{\overrightarrow    {\rm  Ind}(g)}  \ar[d]_-{\pi}
&&
 \vec{\mathcal{M}}\ar[d]^-{\pi}\\
 {\rm  sk}^{k-1}({\rm  Ind}(G))* {\rm  sk}^{k'-1}({\rm  Ind}(G')) 
  \ar[rr]^-{{\rm  Ind}(g)}    &&\mathcal{M}
}
\end{eqnarray}
such  that  $\overrightarrow {\rm  Ind}( g)$  is   
a  directed  simplicial  map    
sending  a  directed  simplex 
 \begin{eqnarray*}
 (v_1,\ldots,v_l,v'_1,\ldots,v'_{l'})\in
   {\rm  sk}^{k-1}(\overrightarrow    {\rm  Ind}(G))* {\rm  sk}^{k'-1}(\overrightarrow    {\rm  Ind}(G')),
   \end{eqnarray*}  
 where  $l\leq  k$  and  $l'\leq  k'$,     to  an  independent   sequence  
 $(g(v_1),\ldots,g(v_l),g(v'_1),\ldots,'(v'_{l'}))\in \vec{\mathcal{I}}*\vec{\mathcal{I}}'$
and    ${\rm  Ind}( g)$   is  a  simplicial  map     
 sending    a   simplex 
 \begin{eqnarray*}
 \{v_1,\ldots,v_l,v'_1,\ldots,v'_{l'}\}\in 
  {\rm  sk}^{k-1}({\rm  Ind}(G))* {\rm  sk}^{k'-1}(     {\rm  Ind}(G')) 
    \end{eqnarray*}  
   to  an  independent   set     
 $\{g(v_1),\ldots,g(v_l),g(v'_1),\ldots,'(v'_{l'})\}\in \mathcal{I}*\mathcal{I}'$.  
\end{lemma}  

\begin{proof}
The  proof  is  an  analog  of  Proposition~\ref{pr-251107-1}.  
For  any  $k$-distinct  vertices  $v_1,\ldots,v_k\in  V$   and  
any    $k'$-distinct  vertices  $v'_1,\ldots,v'_{k'}\in  V'$,  
  $v_1,\ldots,v_k$  are  non-adjacent  in  $G$   if  and  only  if  
    $(v_1,\ldots,v_k)\in  {\rm  sk}^{k-1}(\overrightarrow    {\rm  Ind}(G)) $
    if  and  only   if $\{v_1,\ldots,v_k\}\in  {\rm  sk}^{k-1}(  {\rm  Ind}(G)) $, 
 and  
 $v'_1,\ldots,v'_{k'}$  are  non-adjacent  in  $G'$   if  and  only  if  
    $(v'_1,\ldots,v'_{k'})\in  {\rm  sk}^{k'-1}(\overrightarrow    {\rm  Ind}(G')) $
     if  and  only  if  
    $\{v'_1,\ldots,v'_{k'}\}\in  {\rm  sk}^{k'-1}(     {\rm  Ind}(G')) $. 
    Hence 
    $v_1,\ldots,v_k$  are  non-adjacent  in  $G$  and  
     $v'_1,\ldots,v'_{k'}$  are  non-adjacent  in  $G'$  
     if  and  only  if 
      $(v_1,\ldots,v_k,v'_1,\ldots,v'_{k'})\in 
  {\rm  sk}^{k-1}(\overrightarrow   {\rm  Ind}(G))* {\rm  sk}^{k'-1}(\overrightarrow    {\rm  Ind}(G'))  $
   if  and  only  if 
      $\{v_1,\ldots,v_k,v'_1,\ldots,v'_{k'}\}\in 
  {\rm  sk}^{k-1}(   {\rm  Ind}(G))* {\rm  sk}^{k'-1}(    {\rm  Ind}(G'))  $. 
  Therefore,  
  the  existence  of $(G, k; G', k')$-regular  maps  $g$  can  be  equivalently  characterized  
  by  the  existence  of  the directed simplicial  embeddings   $\overrightarrow {{\rm  Ind}}(g)$
  or  the  existence  of    the  simplicial  embeddings  ${\rm  Ind}(g)$  in  (\ref{diag-251111-a1}).  
\end{proof}

Let  $f:  G\longrightarrow  \mathbb{F}^N$  be  a  $k$-regular  map.  
Let  $f':  G'\longrightarrow  \mathbb{F}^{N'}$  be  a  $k'$-regular  map.  
 Then  we  have  an  induced    $(G, k; G', k')$-regular  map 
 \begin{eqnarray*}
 (f, f'):  G  \sqcup   G'\longrightarrow  \mathbb{F}^{N+N'}   
 \end{eqnarray*}
sending  any  $v\in  V$  to  $(f(v),0)\in  \mathbb{F}^{N}\times \mathbb{F}^{N'}$  and  sending  any  $v'\in  V'$  to  $(0,f'(v'))\in  \mathbb{F}^{N}\times \mathbb{F}^{N'}$.  
Moreover,  if  $S$  is  a  finite  subset  of  $\mathbb{F}^N$  and  
$S'$  is  a  finite  subset  of  $\mathbb{F}^{N'}$  such  that  
$f(V)\subseteq  S$  and  $f'(V')\subseteq  S'$,  
then  $(f,f')(V\sqcup  V')\subseteq  (S,0) \sqcup  (0,S')$.  

\begin{theorem}
\label{th-251111-ku}
For  any  $k$-regular  map   
 $f: G\longrightarrow  \mathbb{F}^N$ 
 and  any  $k'$-regular map    
 $f': G'\longrightarrow  \mathbb{F}^{N'}$
 such  that  $f(V)\subseteq  S$  and  $ f'(V')\subseteq  S'$,
 we  have    a   commutative  diagram  (\ref{diag-251230-7})  of  short  exact  sequences  
 where  the  horizontal  maps  are  induced  by  $f$, $f'$  and  $(f,f')$.  
  Moreover,  all  these   homomorphisms  of  homology  groups  are  
 functorial  in the  sense  of  
   Theorem~\ref{th-251107-1}.  
\end{theorem}

\begin{proof} 
  By 
Corollary~\ref{pr-251112-ku},   
 Theorem~\ref{co-10-27-kf-1}  and   
  Theorem~\ref{th-251107-1},  
 the  homomorphisms  of  homology  groups   in  (\ref{diag-251107-b1})  
 induce   the   homomorphisms  of  
 the  short  exact  sequences  of   
 homology  groups  in  (\ref{diag-251230-7}).  
 The  commutativity  and  the  functoriality  of   the  diagram  (\ref{diag-251107-b1})   imply  
 the  commutativity  and  the  functoriality   of  the  diagram   (\ref{diag-251230-7}) 
 respectively.  
\end{proof}

\begin{theorem}
\label{th-251111-ku-hg}
Let     
 $f: G\longrightarrow  \mathbb{F}^N$ 
 be       $k$-regular  
  and  let  
 $f': G'\longrightarrow  \mathbb{F}^{N'}$
 be    $k'$-regular  
     such  that  $f(V)\subseteq  S$  and  $ f'(V')\subseteq  S'$.
     Then  
     for  any sub-hyperdigraph  $\vec{\mathcal{H}}$
  of  ${\rm  sk}^{k-1}(\overrightarrow{\rm  Ind}(G))$  and  
any sub-hyperdigraph  $\vec{\mathcal{H}}'$
  of  ${\rm  sk}^{k-1}(\overrightarrow{\rm  Ind}(G'))$
with  their  underlying  hypergraphs  
$\mathcal{H}$  and  $\mathcal{H}'$  respectively,     
 we  have    a   commutative  diagram 
 (\ref{diag-251230-6})   of  short  exact  sequences  
 where  the  horizontal  maps  are  induced  by  $f$, $f'$  and  $(f,f')$.  
  Moreover,  all  these   homomorphisms  of  homology  groups  are  
 functorial  in the  sense  of  
   Theorem~\ref{th-251107-1}.  
\end{theorem}

\begin{proof} 
The  proof  follows  from  Theorem~\ref{th-251127-hku5},   Theorem~\ref{th-251107-1}  and  
Theorem~\ref{th-251111-ku}.  
\end{proof}

    \bigskip

Shiquan Ren

Address:
School  of  Mathematics and Statistics,  Henan University,  Kaifeng   475004,  China.

e-mail:  renshiquan@henu.edu.cn

  \end{document}